\documentclass[english]{article}
\usepackage{amsmath,amssymb,latexsym,amsthm,mathrsfs,appendix}
\usepackage{color}

\date{}

\def\theenumi{\arabic{enumi}}

\def\theenumii{\alph{enumii}}
\def\p@enumii{\theenumi.}

\def\theenumiii{\arabic{enumiii}}
\def\p@enumiii{(\theenumi)(\theenumii)}

\def\p@enumiv{\p@enumiii.\theenumiii}
\newtheorem{theorem}{Theorem}[section]

\newtheorem{corollary}[theorem]{Corollary}

\newtheorem{proposition}[theorem]{Proposition}
\newtheorem{lemma}[theorem]{Lemma}

\theoremstyle{definition}
\newtheorem{example}[theorem]{Example}

\newtheorem{notation}[theorem]{Notation}

\newtheorem{definition}[theorem]{Definition}
\newtheorem{remark}[theorem]{Remark}

\makeatother

\usepackage{babel}
\begin{document}

\title{Isoperimetric profiles and random walks on some 
permutation wreath products}
\author{Laurent Saloff-Coste\thanks{%
Partially supported by NSF grant DMS 1004771 and DMS 1404435} \\
{\small Department of Mathematics}\\
{\small Cornell University} \and Tianyi Zheng \\
{\small Department of Mathematics}\\
{\small Stanford University} }
\maketitle

\begin{abstract}
We study the isoperimetric profiles of certain families of finitely 
generated  groups defined via 
marked Schreier graphs and permutation wreath products. 
The groups we study are among the ``simplest'' examples within a 
much larger class of groups, all defined via marked Schreier graphs 
and/or action 
on rooted trees, which  
includes such examples as the long range group, Grigorchuck group and 
the basillica group.  The highly non-linear structure of these groups 
make them both interesting and difficult to study. 
Because of the 
relative simplicity of the Schreier graphs that define the groups we study here
(the key fact is that they contained very 
large regions that are ``one dimensional''), we are able to obtain sharp 
explicit bounds on the $L^1$ and $L^2$ isoperimetric profiles of these 
groups. As usual, 
these sharp isoperimetric profile estimates provide sharp bounds on the 
probability of return of simple random walk.
Nevertheless, within each of the families of groups we study there 
are also many cases 
for which the existing techniques appear inadequate and this leads to a variety 
of open problems.
\end{abstract}

\section{Introduction}\setcounter{equation}{0}

In the study of random walks on groups, some of the most basic and 
compelling questions are to understand what structural properties of 
a group impact the behavior of random walk and how this impact 
can be captured precisely. Naturally, this also leads 
to the question of describing
what random walk behaviors can possibly occur. 

To any finitely generated group, one can associate
the monotone non-increasing functions 
$$\Lambda_{1,G},\Lambda_{2,G} \mbox{ and } \Phi_G$$ which, respectively, 
describe the $L^1$- and $L^2$-isoperimetric profiles
and the return probability (or heat kernel decay) associated with 
the group $G$ (precise definitions are recalled below in Section
\ref{sec-isoandreturn}). 
From a coarse point of view explained in Section \ref{sec-isoandreturn}, 
these are group invariants in the 
sense that they do not depend on the particular choice of symmetric finite 
generating set that is used to define them. A
celebrated results of F\o lner and Kesten asserts that the 
dichotomy between amenable and non-amenable groups can be captured precisely 
using any one of these three invariants: A group is amenable if and only if 
$\Lambda_{1,G}$ (equivalently, $\Lambda_{2,G}$) is bounded below away from $0$,
and this is also equivalent to having $\Phi_G$ decay exponentially fast.

This paper focuses on these 
invariants and how they depend on the structure of the underlying 
group in the context of several interesting family of amenable groups. Let us 
stress that there are other related random walk characteristics such as 
entropy and speed that are of great interest but are not 
discussed here.  

To put this work in perspective, recall that among polycyclic groups or 
(almost equivalently) finitely generated 
discrete amenable subgroups of linear groups, the behaviors of 
$\Lambda_{1,G}$, $\Lambda_{2,G}$ and  $\Phi_G$ are well understood and fall 
in exactly 2 possible categories: 
\begin{itemize}
\item The group $G$ has exponential volume growth and
$$\Lambda_{1,G}(v)^2 \simeq \Lambda_{2,G}\simeq  \frac{1}{[\log (1+v)]^2}
\;\mbox{ and } \Phi_G(n)\simeq \exp(- n^{1/3}).$$ 
\item The volume growth $V_G$ satisfies $V_G(r)\simeq (1+r)^d$ for 
some integer $d$ and
$$\Lambda_{1,G}(v)^2 \simeq \Lambda_{2,G}(v)\simeq (1+v)^{-2/d} 
\;\mbox{ and } \Phi_G(n)\simeq (1+n)^{-d/2}.$$ 
\end{itemize} 
These can be considered as the ``classical'' behaviors. 
See \cite{Tessera} for the description of a larger class of groups 
for which only these behaviors can occur. 

By now it is well-known that, for more general groups, 
other behaviors can occur. See 
\cite{Erschler2006,Pittet2002,SCnotices}. For instance, 
the authors show in \cite{Saloff-Coste2013b} that the free solvable group 
$\mathbf S_{d,r}$ of derived length $d>2$ on $r$ generators satisfies
$$\Lambda_{1,\mathbf S_{d,r}}(v)^2 \simeq
\Lambda_{2,\mathbf S_{d,r}}(v)
\simeq \left(\frac{\log_d(v)}{\log_{d-1}(v)}\right)^{2/r}$$
and 
$$ \Phi_{\mathbf S_{d,r}}
(n)\simeq  \exp\left(-n \left(\frac{\log_{d-1}(n)}{\log_{d-2}(n)}\right)^{2/r}\right).$$
Here, $\log _d$ denotes the iterated logarithm, $d$-times.
See \cite{Saloff-Coste2013b} for the statement when $d=2$ in which case 
the formula for $\Phi_{\mathbf S_{d,r}}$ must be modified (the estimates for 
$\Lambda_{1,\mathbf S_{d,r}}$ and $\Lambda_{2,\mathbf S_{d,r}}$ remain valid as stated above).

Following the groundbreaking work of R. Grigorchuck regarding 
groups of intermediate volume growth and the many works that followed,
it has become apparent that it is important to
consider the case of subgroups of the automorphism 
group of a rooted tree as well as groups 
defined via explicit marked Schreier graphs 
(the word ``explicit'' is important here as any finitely generated 
group can be ``defined'' by its 
action on one of its marked Cayley graph).

In this paper we study a collections of examples of such groups and 
their associated  permutation wreath products. Most of the paper is devoted to two families. The first family has been considered in \cite{Kotowski2014}
where they are called ``bubble groups'' 
(see Section \ref{sec-bubble}). We call the other family 
``cyclic Neumann-Segal groups'' (see Section \ref{sec-NS}). It is part of  
a larger family considered by a number of different authors 
after some interesting properties were pointed out in  
\cite{Neumann1986}.
We obtain estimates on the $L^1$- and 
$L^2$-isoperimetric profiles of groups in these families. 
As is well-known, this yields assorted results for the probability of return 
of associated random walks. The examples we consider are, in a sense, 
among the simplest in the very large family alluded to above. 
The key feature that distinguishes these examples is that 
the Schreier graphs that are used to defined them have large one-dimensional 
pieces at all scales. This allow us to capture the isoperimetric profiles of 
the wreath extensions of some of these groups in a rather precise way.  
The resulting observed behaviors are diverse and quite interesting.

For example, 
we study a family  of groups depending on a continuous 
parameter $\gamma>1$ and for which we show that
$$\Phi_G(n)\simeq \exp
\left(- n 2^{-2((1+\gamma)\log_2(n))^{\gamma/(1+\gamma)}} \right).$$
See Remark \ref{rem-exaNS} (the coarse equivalence $\simeq$ is defined in \ref{sec-isoandreturn}).  We also obtain examples with
\begin{equation}\label{exagamma}
\Phi_G(n)\simeq\exp\left(-n^{\gamma}(\log n)^{1-\gamma}\right)
\end{equation}
for each $ \gamma\in [1/3,1)$ (see Example \ref{exa-bubb1} with 
$\gamma=\frac{\beta+1}{3\beta+1}$, $\beta\in (0,\infty)$) 
and 
$$\Phi_G(n) \simeq  \exp\left( -n^{1/3} (\log n)^{2(1+\kappa)/3}\right)$$
for each $\kappa>0$ (see Example \ref{exa-bubb2}).
We discuss a family of groups such that, for any $\gamma\in (1/3,1)$,
there is a group $G$ in the family for which 
\begin{equation} \label{osc} 
\exp\left(
-n^{\frac{\gamma+1}{3-\gamma}}  
\left(\log n \right)^{\frac{2(1-\gamma)}{3-\gamma}}  \right)
\lesssim 
\Phi_G (n) \lesssim 
\exp\left(- n^{\gamma} 
\left(\log n \right)^{1-\gamma}\right)
\end{equation}
with both extremes attained at certain times and a detailed coarse 
description of $\Phi_G$ available at all times. See \ref{exa-01} 
with $\gamma= \varkappa/(3\varkappa-2)$, $\varkappa\in (1,\infty)$.
These examples demonstrate the existence of 
a continuum of distinct (and explicit) behaviors and, for each, 
we obtain corresponding estimates for the 
functions $\Lambda_{1,G}$ and $\Lambda_{2,G}$.

The groups we study provided a host of additional behaviors. In particular,
they provides examples of pairs of groups $G_1,G_2$ for which the behaviors 
of the functions $\Lambda_{1,G_i},\Lambda_{2,G_i}$ and $\Phi_{G_i}$ can be 
described explicitly and such that the functions 
$\Lambda_{1,G_1}$ and $\Lambda_{1,G_2}$
are not comparable in the sense that
$$\liminf_{n\rightarrow \infty}\frac{\Lambda_{1,G_1}}{\Lambda_{1,G_2}}=0 
\mbox{ and }
\limsup_{n\rightarrow \infty}\frac{\Lambda_{1,G_1}}{\Lambda_{1,G_2}}=\infty,$$
with similar statement holding as well for $\Lambda_{2,G_i}$ and $\Phi_{G_i}$,
$i=1,2$.  Explicitly, fix $\gamma\in (1/3,1)$
and pick $G_1$ so that \ref{exagamma} holds. Pick $\gamma'\in (1/3,\gamma)$
so that 
$$\gamma'<\gamma < \frac{\gamma'+1}{3-\gamma'}= 
\gamma'+ \frac{(\gamma'-1)^2}{3-\gamma'},$$
(for instance $\gamma'=2/5<\gamma=1/2< 7/13$) and let $G_2$ be  such that
(\ref{osc}) holds with parameter $\gamma'$.  The various family of groups we study offer many further 
opportunities, which we do not pursue explicitly, 
to construct such examples.

The families we consider are rich enough that 
many intriguing questions remain open. For one thing, most of our sharp 
results concern wreath extensions 
$\mathbb Z\wr_{\mathcal S} \Gamma$ where $\Gamma$ is defined by its action on 
a mark Schreier graph $\mathcal S$ and, for the most part, 
understanding the isoperimetric profile of $\Gamma$ itself is an open question. 
Also, the groups in the two main families we consider 
depend on the choice of one (or two) infinite sequence(s) 
of integers and our analysis provides sharp results only in a certain 
parameter range.  This leaves much space for further studies.

This article is motivated by the works of a number of authors to whom 
we have borrowed results, problems and ideas. We make fundamental 
use of the work of A. Erschler on wreath products and adapt some of the ideas 
developed by K. Juschenko and her co-authors 
\cite{JBMS,Juschenko2013a,Juschenko2013}. 
The examples we consider (or some related cousins) 
have been studied before (with somewhat different viewpoints and focuses) 
in works 
including \cite{Bartholdigrowthtorsion,
Brieussel2011,FabGupta,Grigorchuk2000,Grigorchuk2011,
Neumann1986,Sidki2000}.

This article is organized as follows. Section \ref{sec-Prel} 
introduces definition and notation regarding  group actions, Schreier graphs,
groups of automorphisms of rooted trees and random walks. It contains the 
definitions of the $L^1$- and $L^2$-isoperimetric profiles and that of the 
random walk invariant $\Phi_G$. Section \ref{sec-low} provide techniques  
that produce lower bounds on $\Lambda_{1,G}$ and $\Lambda_{2,G}$. 
We make significant use of Erschler's wreath product isoperimetric inequality 
and of  related ideas developed in \cite{Saloff-Coste2014}.
By known techniques (e.g., \cite{CNash} and 
\cite[Section 2.1]{Saloff-Coste2014}), 
these translate into upper bounds on $\Phi_G$.

Section \ref{sec-UB} develops abstract considerations with the goal of 
providing upper bounds on the $L^2$-isoperimetric profiles and assorted 
lower bound for $\Phi_G$.  Regarding the $L^1$- and $L^2$-isoperimetric 
profiles, in all the cases where we obtain sharp bounds, 
it turns out that $\Lambda_{1,G}^2\simeq \Lambda_{2,G}$ (whether or not this 
is true in general is a well-known and important open question).  
Since we always have $\Lambda_{1,G}^{2}\lesssim \Lambda_{2,G}$, 
proving that  $\Lambda \lesssim \Lambda_{1,G}$ and 
$\Lambda_{2,G}\lesssim \Lambda^2$ for some function $\Lambda$ is
sufficient to prove that $\Lambda_{1,G}^2\simeq \Lambda_{2,G}\simeq \Lambda^2$.

Section \ref{sec-bubble} introduces the ``bubble group'' family.  
This is an uncountable family of groups parametrized by two integers 
sequences $(\mathbf a=(a_1,a_2,\dots)$ and $\mathbf b=(b_1,b_2,\dots)$.
We show that most of these groups have exponential volume growth 
(Lemma \ref{lem-bubexp}) and that 
they are amenable as long as $a_n$ tends to infinity
(Proposition \ref{pro-am}). The result of Section \ref{sec-low} apply readily 
but works is required to show how the 
abstract results of Section \ref{sec-UB} applies to these examples. 
When $\mathbf b$ is constant ($b_i=b>2$),
and $a_n$ is increasing fast enough, we obtain matching two-sided bounds on 
$\Lambda_{1,G},\Lambda_{2,G}$ and $\Phi_G$ where 
$G=\mathbb Z\wr _\mathcal S \Gamma_{\mathbf a,\mathbf b}$. Here $\Gamma_{\mathbf a,\mathbf b}$ is the bubble group associated with $\mathbf a ,\mathbf b$ and 
$\mathcal S$ is its defining Schreier graph.
See Theorem \ref{bubblebound}.

Section \ref{sec-NS} is devoted to a sub-family (cyclic Neumann-Segal groups) 
of the family of Neumann-Segal groups. The cyclic Neumann-Segal groups  
are parametrized by a sequence of even integers. To study the isoperimetric 
profile of these groups and their wreath extensions,  we apply the results
of Section \ref{sec-low}  (again, this is straightforward), and apply the 
result and ideas of Section \ref{sec-UB}.  In one particular case of interest, 
we are able to obtain sharp results not only not only for the wreath extension
$\mathbb Z\wr_{\mathcal S}\Gamma$ but also for the cyclic Neumann-Segal group 
$\Gamma$ itself. In fact, in this particular case, the groups $\Gamma$ and 
its wreath extension 
$\mathbb Z\wr _\mathcal S \Gamma$ have essentially the same behavior. See 
Theorem \ref{theo-NSgamma} and Remark \ref{rem-exaNS}.

\section{Preliminaries} \setcounter{equation}{0} \label{sec-Prel}
\subsection{Group actions}

A left action of a group $\Gamma$ on a space $X$ is a map 
$\varphi_{L}:\Gamma\times X \to X$ such that
$$\varphi_{L}(e_{\Gamma},x) =x \mbox{ and }
\varphi_{L}\left(g_{1}g_{2},x\right)  =\varphi_{L}\left(g_{1},\varphi_{L}\left(g_{2},x\right)\right).$$
Similarly a right action of $\Gamma$ on $X$ is a map 
$\varphi_{R}:X\times\Gamma \to X$ such that
$$\varphi_{R}(x,e_{\Gamma})  =x \mbox{ and }
\varphi_{R}\left(x,g_{1}g_{2}\right)  
=\varphi_{R}\left(\varphi_{R}\left(x,g_{1}\right),g_{2}\right).$$

In this article, it is natural to consider examples of both left and 
right actions, depending of the context. Given an action $\varphi_{*}$ of 
$\Gamma$ on $X$ where $*$ is either $L$ or $R$, we set
$$g\cdot x=\left\{\begin{array}{ll} \varphi_{L}(g,x) & \mbox{ if } *=L,\\    
\varphi_{R}(x,g^{-1}) & \mbox{ if } *=R.    
\end{array}\right.$$

Let 
$$\mathcal{P}_{f}(X)
=\oplus_{X}\mathbb{Z}_{2}$$ be the set of all
finite subsets of $X$. Any action of $\Gamma$ on $X$ extends to an action of $\Gamma$ on $\mathcal{P}_{f}(X)$.
If $f:X\to X$ is a function on $X$, and $g\in \Gamma$, 
we let $g\cdot f:X\to X$ be defined by
$$g\cdot f(x)=f(g^{-1}\cdot x).$$
It follows that
$\mbox{supp}(g\cdot f)=g\cdot\, \mbox{supp}(f)$. 

Note that it is not rare that a group is, in fact, described by its action
on a space $X$ (i.e., as a permutation group). We will encounter 
many such examples.

\begin{example}[The finite dihedral groups $D_{2n}$]\label{exa-finiteD} 
The following two figures define the dihedral group $D_{2n}$
of order $2n=20$ generated by two elements of order $2$, $s$ and $t$, with 
$(st)^{n}=e$. Figure \ref{D1} shows a marked Schreier graph on which 
the action of the group is faithful. It identifies the dihedral group as 
a particular subgroup of the symmetric group $S_{10}$ 
where the 10 objects that are permuted are the 
the vertices of the pictured line graph and the action of the generator 
$s$ and $t$ are described by the marked edges. In this particular case, since $s$ and $t$ are involutions, we do not need to indicate the edge orientation.  

Figure \ref{D2} gives the marked Cayley graph of the same group 
with the same generators $s,t$ and with the identity element $e$ identified.

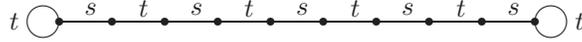
\begin{figure}[h]
\begin{center}\caption{Marked Schreier graph}\label{D1} 
\begin{picture}(200,30)(0,0) 
\put(10,10){\line(1,0){180}} 
\multiput(10,10)(20,0){10}{\circle*{3}}
\put(4,10){\circle{12}}\put(196,10){\circle{12}}
\put(-7,10){\makebox(0,0){$t$}}\put(207,10){\makebox(0,0){$t$}}
\multiput(22,15)(40,0){5}{\makebox(0,0){$s$}}
\multiput(42,15)(40,0){4}{\makebox(0,0){$t$}}
\end{picture}\end{center}\end{figure}

\begin{figure}[h]
\begin{center}\caption{Marked Cayley graph}\label{D2}
\begin{picture}(200,30)(0,0) 
\put(10,10){\line(1,0){180}} 
\multiput(10,10)(20,0){10}{\circle*{3}}

\put(10,-10){\line(1,0){180}} 
\multiput(10,-10)(20,0){10}{\circle*{3}}
\put(10,-10){\line(0,1){20}} 
\put(190,-10){\line(0,1){20}} 
\put(10,-10){\circle*{5}}
\put(4,0){\makebox(0,0){$t$}}\put(2,-15){\makebox(0,0){$e$}}
\put(194,0){\makebox(0,0){$t$}}

\multiput(22,-15)(40,0){5}{\makebox(0,0){$s$}}
\multiput(42,-15)(40,0){4}{\makebox(0,0){$t$}}
\multiput(22,15)(40,0){5}{\makebox(0,0){$s$}}
\multiput(42,15)(40,0){4}{\makebox(0,0){$t$}}
\end{picture}\end{center}\end{figure}
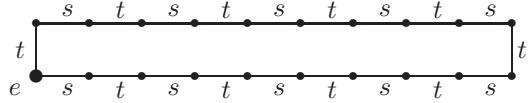

\begin{figure}[h]
\begin{center}\caption{A function $f$ and the function $(ts)\cdot f$}\label{D11} 
\begin{picture}(200,70)(0,0) 
\put(0,40){
\put(10,10){\line(1,0){180}} 
\multiput(10,10)(20,0){10}{\circle*{3}}
\put(4,10){\circle{12}}\put(196,10){\circle{12}}
\put(-7,10){\makebox(0,0){$t$}}\put(207,10){\makebox(0,0){$t$}}
\multiput(22,15)(40,0){5}{\makebox(0,0){$s$}}
\multiput(42,15)(40,0){4}{\makebox(0,0){$t$}}

{\color{red}
\put(10,18){\makebox(0,0){$0$}}
\put(30,18){\makebox(0,0){$5$}}
\multiput(50,18)(20,0){5}{\makebox(0,0){$0$}}
\multiput(170,18)(20,0){2}{\makebox(0,0){$0$}}
\put(150,18){\makebox(0,0){$2$}}}}

\put(10,10){\line(1,0){180}} 
\multiput(10,10)(20,0){10}{\circle*{3}}
\put(4,10){\circle{12}}\put(196,10){\circle{12}}
\put(-7,10){\makebox(0,0){$t$}}\put(207,10){\makebox(0,0){$t$}}
\multiput(22,15)(40,0){5}{\makebox(0,0){$s$}}
\multiput(42,15)(40,0){4}{\makebox(0,0){$t$}}

{\color{red}
\put(10,18){\makebox(0,0){$5$}}
\put(30,18){\makebox(0,0){$0$}}
\multiput(50,18)(20,0){3}{\makebox(0,0){$0$}}
\multiput(130,18)(20,0){4}{\makebox(0,0){$0$}}
\put(110,18){\makebox(0,0){$2$}}}

\end{picture}\end{center}\end{figure}
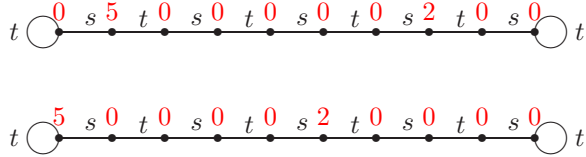

To understand how Picture \ref{D1} defines a subgroup of $\mathbb S_{10}$,
imagine the vertex of the graph as fixed boxes which contains the labels
$\{1,\dots,10\}$. The marking of the edges indicates how the group elements 
$s$ and $t$ each move the labels contained in the boxes, 
producing a permutation of the labels. A function $f$ on the Schreier graph
is really a function on the labels $\{1,\dots,10\}$ and it can be pictured 
by indicating the value of $f(x)$ above $x$ (where label $x$ is in box $x$).  
Suppose $f=\mathbf 1_{x_0}$ and $g\in \Gamma$. Then the function 
$g\cdot f $ is equal to $(g\cdot f)(x)=\mathbf 1_{x_0}(g^{-1}\cdot x)= 
\mathbf 1_{g\cdot x_0}(x)$. In words, to write down the picture describing 
the function $g\cdot f$, move the indicated values of $f$ along the 
Schreier graph according to the action of $g$. 
\end{example}

\subsection{Permutational wreath product}

Let $\Gamma$ be a finitely generated
group acting on a space $X$. Take a reference point $o$ in
$X$ and let $\mathcal{S}$ be the orbital Schreier graph of $o$ under
action of $\Gamma$. Given an auxiliary (finite or countable) group $H$,
the permutational wreath product $H\wr_{\mathcal{S}}\Gamma$
is the semidirect product 
$$H\wr_{\mathcal{S}}\Gamma = \left(\oplus_{\mathcal{S}}(H)_{x}\right)\rtimes\Gamma$$
where $\Gamma$ acts on the direct sum by permuting coordinates according
to the action $(g,x)\mapsto g \cdot x$ of $\Gamma $ on $\mathcal S$. 
More precisely, the multiplication rule is given
by 
\[
(f,g)(f',g')=(f [g\cdot f'],gg'),
\]
where $f,f'$ are functions $\mathcal{S}\rightarrow H$ with finite
support, $g,g'\in G$, and $(g\cdot f')(x)=f'(g^{-1}\cdot x)$ 
as defined earlier. 
Note that on the left-hand side of the equation displayed above, 
$(f,g)(f',g')$
indicates multiplication in $H\wr_{\mathcal{S}}\Gamma$ and that on the 
right-hand side, $f[g\cdot f']$ indicates multiplication of $f$ by $g\cdot f'$ in 
$\oplus_{\mathcal{S}}(H)_{x}$ whereas $gg'$ indicates multiplication 
in $\Gamma$.

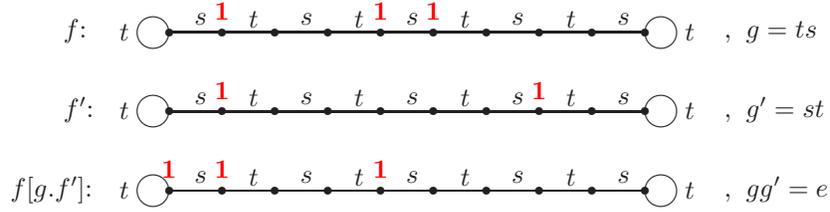
\begin{figure}[h]
\begin{center}\caption{Two elements $(f,g)$ and $(f',g')$ of 
$\mathbb Z_2\wr_{\mathcal S}\mathbb D_{20}
$ and their product} \label{D1*}  
\begin{picture}(200,90)(0,0) 

\put(-50,8){\makebox{$f[g.f']$:}}\put(220,8){\makebox{,\;\;$gg'=e$}}
\put(10,10){\line(1,0){180}} 
\multiput(10,10)(20,0){10}{\circle*{3}}
\put(4,10){\circle{12}}\put(196,10){\circle{12}}
\put(-7,10){\makebox(0,0){$t$}}\put(207,10){\makebox(0,0){$t$}}
\multiput(22,15)(40,0){5}{\makebox(0,0){$s$}}
\multiput(42,15)(40,0){4}{\makebox(0,0){$t$}}
{\color{red}
\put(10,18){\makebox(0,0){$\mathbf 1$}}
\put(30,18){\makebox(0,0){$\mathbf 1$}}
\put(90,18){\makebox(0,0){$\mathbf 1$}}}

\put(0,30){\put(-30,8){\makebox{$f'$:}}\put(220,8){\makebox{,\;\;$g'=st$}}
\put(10,10){\line(1,0){180}} 
\multiput(10,10)(20,0){10}{\circle*{3}}
\put(4,10){\circle{12}}\put(196,10){\circle{12}}
\put(-7,10){\makebox(0,0){$t$}}\put(207,10){\makebox(0,0){$t$}}
\multiput(22,15)(40,0){5}{\makebox(0,0){$s$}}
\multiput(42,15)(40,0){4}{\makebox(0,0){$t$}}
\color{red}{\put(30,18){\makebox(0,0){$\mathbf 1$}}
\put(150,18){\makebox(0,0){$\mathbf 1$}}}
}

\put(0,60){ \put(-30,8){\makebox{$f$:}}\put(220,8){\makebox{,\;\;$g=ts$}}
\put(10,10){\line(1,0){180}} 
\multiput(10,10)(20,0){10}{\circle*{3}}
\put(4,10){\circle{12}}\put(196,10){\circle{12}}
\put(-7,10){\makebox(0,0){$t$}}
\put(207,10){\makebox(0,0){$t$}}
\multiput(22,15)(40,0){5}{\makebox(0,0){$s$}}
\multiput(42,15)(40,0){4}{\makebox(0,0){$t$}}
{\color{red}\put(30,18){\makebox(0,0){$\mathbf 1$}}
\put(90,18){\makebox(0,0){$\mathbf 1$}}
\put(110,18){\makebox(0,0){$\mathbf 1$}}}
}

\end{picture}\end{center}\end{figure}

The ordinary wreath product
with ``lamp group'' $H$ and ``base group'' $\Gamma$ corresponds to the case
when $\Gamma$ acts on its own Cayley graph
by left multiplication. We write $H\wr\Gamma$ for this ordinary wreath
product. 

The groups $H$ and $\Gamma$ are naturally embedded in 
$H\wr_{\mathcal{S}}\Gamma$
via the injective maps
\[
h\mapsto\left(\boldsymbol{1}_{h}^{o},e_{\Gamma}\right),\ \mbox{where }\boldsymbol{1}_{h}^{o}(o)=h,\ \boldsymbol{1}_{k}^{o}(x)=e_{H}\ \mbox{if }x\neq o,
\]
and
\[
\gamma\mapsto\left(\mathbf{e}_{H}^{\mathcal{S}},\gamma\right),\ \mbox{where }\mathbf{e}_{H}^{\mathcal{S}}(x)=e_{H}\ \mbox{for every }x\in\mathcal{S}.
\]

Let $\mu_{H}$ and $\mu_{\Gamma}$ be symmetric probability measures
on $H$ and $\Gamma$ respectively. Using the above embeddings, $\mu_{H}$
and $\mu_{\Gamma}$ can be viewed as probability measures on 
$H\wr_{\mathcal{S}}\Gamma$ and we will often make this identification.
The measure 
\[
q=\mu_{H}\ast\mu_{\Gamma}\ast\mu_{H}
\]
on $H\wr_{\mathcal{S}}\Gamma$ 
is often referred to as the 
switch-walk-switch random walk where as 
$$\mathfrak q=\frac{1}{2}\left(\mu_{H}+\mu_{\Gamma}\right)$$
is known as the the switch-or-walk measure on $H\wr_{\mathcal{S}}\Gamma$.
  
The random walk driven by $q$ on $H\wr_{\mathcal{S}}\Gamma$
shares some similarities with the switch-walk-switch random walk on
ordinary wreath products. But there are some important differences.
On the ordinary wreath product, if we consider random walk on the
right, that is $\left\{ X_{1}X_{2}...X_{n}\right\} _{n}$ where $X_{i}$
are i.i.d. random variables distributed as $\mathfrak q$, the random walk can
be understood in terms of a walker on the Cayley graph of $\Gamma$ 
who changes the lamp-configuration along its path.  
On the permutation wreath product $H\wr_{\mathcal{S}}\Gamma$ based on the 
Schreier graph $\mathcal{S}$, the switches
happen along the inverted orbit. This creates much difficulty in analyzing
the behavior of such random walks in the general case. 

In what follows, we are going to consider random
walk on the left, that is $S_{n}=X_{n}X_{n-1}...X_{1}$ and much attention
will be paid to translation of the support of the lamp configuration,
rather than the inverted orbits directly. We will mostly work with 
the measure $\mathfrak q$.

\subsection{Groups acting on trees, activity} \label{sec-tree}

Let $\bar{d}=(d_{j})_{j\ge1}$ be a sequence of integers
$d_{j}\ge2$. The spherical homogeneous rooted tree $\mathbb{T}_{\bar{d}}$
is the tree where each vertex at level $j$ has $d_{j+1}$ children
in level $j+1$. The tree has a root at level $0$, which is denoted
by the empty sequence $\emptyset$. A vertex at level $k$ is encoded
by word $v=x_{1}...x_{k-1}x_{k}$, where $x_{i}$ is a letter in the
alphabet $\{0,...,d_{i}-1\}$. Here a word is read from left to right.
The integer $k$ is called the depth or level of the vertex $v$. We set  
$|v|=k$ and
let $\mathbb{T}_{\bar{d}}^{k}$ denote the set of vertices in level
$k$ of the tree $\mathbb{T}_{\bar{d}}$.

The group $\mbox{Aut}(\mathbb{T}_{\bar{d}})$ is the group of automorphisms
of $\mathbb{T}_{\bar{d}}$ fixing the root $\emptyset$. It is uncountable
and is often equipped 
with the topology of convergence on finite sets which turns it into a locally 
compact group. 
For $v\in\mathbb{T}_{\bar{d}}$,
consider the subtree $\mathbb{T}_{v,\bar{d}}$ of $\mathbb{T}_{\bar{d}}$
rooted at $v$. If $v$ is at level $k$, then this subtree is isomorphic
to the spherical homogeneous rooted tree 
$\mathbb{T}_{\text{\ensuremath{\tau}}^{k}\bar{d}}$,
where $\tau$ denotes the shift operator 
\[
\tau(m_{1},m_{2},...)=(m_{2},m_{3},...).
\]
The automorphism group $\mbox{Aut}(\mathbb{T}_{\bar{d}})$ admits
the following canonical description called ``wreath recursion'', 
\[
\mbox{Aut}(\mathbb{T}_{\bar{d}})
\simeq\mbox{Aut}(\mathbb{T}_{\tau\bar{d}})\wr_{\mathbb{T}_{\bar{d}}^{1}}S_{d_{1}},
\]

\[
g\mapsto(g_{0},...,g_{d_{1}-1})\sigma.
\]
Each $g_{i}$ is called the \emph{section }(or \emph{restriction})
of $g$ at vertex $i$, it gives the action of $g$ on the subtree
rooted at $i$. The rooted component $\sigma$ describes how these
subtrees are permuted. Given a word $v=x_{1}...x_{k}$, the right
action is defined recursively by  
\[
\varphi_{R}\left(x_{1}...x_{k},g\right)= \sigma\left(x_{1}\right)
\varphi_R(x_{2}\dots x_{k},g_{x_{1}}).
\]

One can also write the wreath recursion at level $k$ to have a
canonical isomorphism 
\[
\mbox{Aut}\left(\mathbb{T}_{\bar{d}}\right)\simeq\mbox{Aut}\left(\mathbb{T}_{\tau^{k}\bar{d}}\right)\wr_{\mathbb{T}_{\bar{d}}^{k}}\mbox{Aut}\left(\mathbb{T}_{\bar{d}}^{k}\right),
\]
\[
g\mapsto\left(g^k_{v},v\in\mathbb{T}_{\bar{d}}^{k}\right)\sigma^k.
\]
where the subscript $k$ indicates the level (not a power). In most case, 
the subscript $k$ is omitted because it is clear from the context that the 
decomposition is done at level $k$. Each $g_{v}=g_v^k$ 
is called the \emph{section }(or \emph{restriction})
of $g$ at vertex $v$, it describes the action of $g$ on the subtree
rooted at $v$. The permutation $\sigma=\sigma^k$ describes how vertices on
level $k$ are permuted.

Since the isomorphisms are canonical, we identify $g$ with its image
under the wreath recursions and, given a level $k$,  write 
\[
g=\left(g_{v},v\in\mathbb{T}_{\bar{d}}^{k}\right)\sigma.
\]

For every $g\in\mbox{Aut}(\mathbb{T}_{\bar{d}})$ the \emph{activity
function} $a_{g}(k)$ defined in \cite[Section 2.4]{Sidki2000} 
counts the number of nontrivial
sections $g_{v}$ at each level $k$, that is 
\[
a_{g}(k)=\#\left\{ v\in\mathbb{T}_{\bar{d}}^{k}:\ g_{v}\neq e\right\} .
\]

We say $G<\mbox{Aut}(\mathbb{T}_{\bar{d}})$, is a group of bounded
activity if for every element $g\in G$, $\sup_{k}a_{g}(k)<\infty$.
When $G$ is finitely generated, it's sufficient to check for each
generator of $G$ if it's activity growth is bounded.

\begin{example}\label{exa-Dinfty}
Figures \ref{D3}--\ref{D5} describe (a) the Cayley graph of the infinite 
dihedral group $D_\infty=<s,t: s^2=t^2=e>$,
(b) a marked Schreier graph that can be used to define $D_\infty$, 
and (c) the generator $t$ as an element 
of the automorphism group of the rooted binary tree
$\mbox{Aut}(\mathbb{T}_{\overline{2}})$. Here we set 
$s(x_1x_2\dots x_k)
=\overline{x_1}x_2\dots x_k$ with $\overline{x_1}=x_1+1 \mod 2$, that is
$s=(e,e)\tau$ where $\tau$ is the transposition at level 1, and
$t=(t,s) 1$ where $1$ stands the identity permutation at level $1$.  
Obviously, the definition of $t$ is recursive. If $x=x_1\dots x_k$
with $j$ being the first index such that $x_j=1$ then 
$t(x)=x_1\dots x_j \overline{x_{j+1}},\dots,x_k$.
The activity $a_s(k)$ vanishes for $k>1$. The activity $a_t(k)=1$ for all $k$. 
The very notable difference between $s$ and $t$ viewed as 
automorphisms of the tree is an artifact of this representation. 

Note that $t$ leaves invariant the end $o=0^\infty$ 
of the tree which 
corresponds to the left most vertex in the Schreier graph depicted on 
Figure \ref{D4}. 
In order to understand Figure \ref{D4} in terms of the binary tree and 
Figure \ref{D5}, one simply consider the orbit of $o=0^\infty$ under 
$D_\infty=<s,t>\subset \mbox{Aut}(\mathbb{T}_{\overline{2}})$.

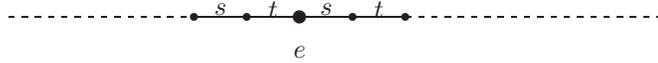
\begin{figure}[h]
\begin{center}\caption{Marked Cayley graph for $D_\infty
=<s,t: s^2=t^2=e>$}\label{D3}
\begin{picture}(250,30)(0,0) 
\put(80,10){\line(1,0){80}}
 \multiput(10,10)(5,0){50}{\line(1,0){2}}
\multiput(80,10)(20,0){5}{\circle*{3}}
\put(90,13){\makebox(0,0){$s$}}\put(130,13){\makebox(0,0){$s$}}
\put(110,13){\makebox(0,0){$t$}}\put(150,13){\makebox(0,0){$t$}}
\put(120,-3){\makebox(0,0){$e$}}\put(120,10){\circle*{5}}

\end{picture}\end{center}\end{figure}

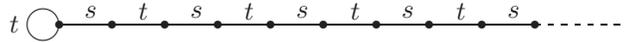
\begin{figure}[h]
\begin{center}\caption{Marked Schreier graph for $D_\infty$}\label{D4} 
\begin{picture}(200,30)(0,0) 
\put(10,10){\line(1,0){180}} 
\multiput(10,10)(20,0){10}{\circle*{3}}
\multiput(30,10)(5,0){40}{\line(1,0){2}}
\put(4,10){\circle{12}}
\put(-7,10){\makebox(0,0){$t$}}
\multiput(22,15)(40,0){5}{\makebox(0,0){$s$}}
\multiput(42,15)(40,0){4}{\makebox(0,0){$t$}}
\end{picture}\end{center}\end{figure}

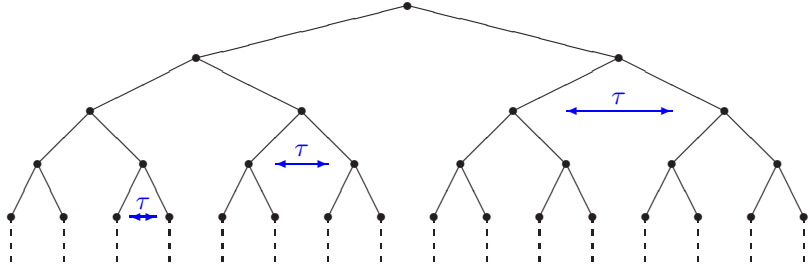
\begin{figure}[t]
\begin{center}\caption{The generator $t$ viewed as an element of 
$\mbox{Aut}(\mathbb T_{\overline{2}})$}\label{D5} 
\begin{picture}(300,100)(-60,0) 
 \put(100,90){\circle*{3}}
\put(100,90){\line(4,-1){80}}\put(100,90){\line(-4,-1){80}}

\put(20,70){\circle*{3}}
\put(20,70){\line(2,-1){40}}\put(20,70){\line(-2,-1){40}}

\put(180,70){\circle*{3}}
\put(180,70){\line(2,-1){40}}\put(180,70){\line(-2,-1){40}}

\put(-20,50){\circle*{3}}
\put(-20,50){\line(1,-1){20}}\put(-20,50){\line(-1,-1){20}}
\put(-40,30){\circle*{3}}
\put(0,30){\circle*{3}}
\put(-40,30){\line(1,-2){10}}\put(-40,30){\line(-1,-2){10}}
\put(-50,10){\circle*{3}}\put(-30,10){\circle*{3}}
\multiput(-50,10)(0,-5){4}{\line(0,-1){2}}
\multiput(-30,10)(0,-5){4}{\line(0,-1){2}}
\put(-0,30){\line(1,-2){10}}\put(0,30){\line(-1,-2){10}}
\put(-10,10){\circle*{3}}\put(10,10){\circle*{3}}
\multiput(-10,10)(0,-5){4}{\line(0,-1){2}}
\multiput(10,10)(0,-5){4}{\line(0,-1){2}}

\put(80,0){\put(-20,50){\circle*{3}}
\put(-20,50){\line(1,-1){20}}\put(-20,50){\line(-1,-1){20}}
\put(-40,30){\circle*{3}}
\put(0,30){\circle*{3}}
\put(-40,30){\line(1,-2){10}}\put(-40,30){\line(-1,-2){10}}
\put(-50,10){\circle*{3}}\put(-30,10){\circle*{3}}
\multiput(-50,10)(0,-5){4}{\line(0,-1){2}}
\multiput(-30,10)(0,-5){4}{\line(0,-1){2}}
\put(-0,30){\line(1,-2){10}}\put(0,30){\line(-1,-2){10}}
\put(-10,10){\circle*{3}}\put(10,10){\circle*{3}}
\multiput(-10,10)(0,-5){4}{\line(0,-1){2}}
\multiput(10,10)(0,-5){4}{\line(0,-1){2}}
}

\put(160,0){\put(-20,50){\circle*{3}}
\put(-20,50){\line(1,-1){20}}\put(-20,50){\line(-1,-1){20}}
\put(-40,30){\circle*{3}}
\put(0,30){\circle*{3}}
\put(-40,30){\line(1,-2){10}}\put(-40,30){\line(-1,-2){10}}
\put(-50,10){\circle*{3}}\put(-30,10){\circle*{3}}
\multiput(-50,10)(0,-5){4}{\line(0,-1){2}}
\multiput(-30,10)(0,-5){4}{\line(0,-1){2}}
\put(-0,30){\line(1,-2){10}}\put(0,30){\line(-1,-2){10}}
\put(-10,10){\circle*{3}}\put(10,10){\circle*{3}}
\multiput(-10,10)(0,-5){4}{\line(0,-1){2}}
\multiput(10,10)(0,-5){4}{\line(0,-1){2}}}

\put(240,0){\put(-20,50){\circle*{3}}
\put(-20,50){\line(1,-1){20}}\put(-20,50){\line(-1,-1){20}}
\put(-40,30){\circle*{3}}
\put(0,30){\circle*{3}}
\put(-40,30){\line(1,-2){10}}\put(-40,30){\line(-1,-2){10}}
\put(-50,10){\circle*{3}}\put(-30,10){\circle*{3}}
\multiput(-50,10)(0,-5){4}{\line(0,-1){2}}
\multiput(-30,10)(0,-5){4}{\line(0,-1){2}}
\put(-0,30){\line(1,-2){10}}\put(0,30){\line(-1,-2){10}}
\put(-10,10){\circle*{3}}\put(10,10){\circle*{3}}
\multiput(-10,10)(0,-5){4}{\line(0,-1){2}}
\multiput(10,10)(0,-5){4}{\line(0,-1){2}}
}

{\color{blue}
\put(180,50){\vector(1,0){20}}\put(180,50){\vector(-1,0){20}}
\put(60,30){\vector(1,0){10}}\put(60,30){\vector(-1,0){10}}
\put(0,10){\vector(1,0){5}}\put(0,10){\vector(-1,0){5}}

\put(180,55){\makebox(0,0){$\tau$}}
\put(60,35){\makebox(0,0){$\tau$}}
\put(0,15){\makebox(0,0){$\tau$}}}
\end{picture}\end{center}\end{figure}
\end{example}

The stabilizers and rigid stabilizers are very important in the study
of branch groups (\cite{Grigorchuk2000,Grigorchuk2011}).
Let $\Gamma $ be a subgroup of $\mbox{Aut}(\mathbb{T}_{\bar{d}})$.  
Given a vertex $u\in\mathbb{T}_{\bar{d}}$, the stabilizer
of $u$ in $\Gamma$ is the subgroup 
\[
\mbox{St}_{\Gamma}(u)=\left\{ g\in\Gamma:\ g\cdot u=u\right\} 
\]
of $\mathbb{T}_{\bar{d}}$.  The subgroup
\[
\mbox{St}_{\Gamma}(k)=\cap_{u\in\mathbb{T}_{\bar{d}^k}}\mbox{St}_{\Gamma}(u)
\]
 is called the level $k$ stabilizer. The rigid vertex stabilizer
is 
\[
\mbox{rist}_{\Gamma}(u)=\left\{ g\in\Gamma:\ g\cdot v=v\ \mbox{for all }
v\notin\mathbb{T}_{u,\bar{d}}\right\} .
\]
The level $k$ rigid stabilizer is the direct product 
\[
\mbox{rist}_{\Gamma}(k)=\prod_{u\in\mathbb{T}_{\bar{d}}^{k}}\mbox{rist}_{\Gamma}(u).
\]

\subsection{$L^{2}$-isoperimetric profile and return probability}
\label{sec-isoandreturn}

Given a probability measure $\phi$ on a group $G$, 
let $(S^l_n)_0^\infty$ (resp, $(S^r_n)_0^\infty$)
denotes the trajectory of the left (resp. right) random walk driven by $\phi$ 
(often started at the identity element $e$).  More precisely, if 
$(X_n)_1^\infty$ are independent identically distributed $G$-valued 
random variables with law $\phi$, then
$$S_n^l= X_n\dots X_1X_0  \;\;(\mbox{ resp. }  S_n^r=X_0X_1\dots X_n).$$
Let $\mathbf P^x_{*,\phi},\; *=l \mbox{ or } r$ be the associated measure 
on $G^\mathbb N$ with $X_0=x$  and 
$\mathbf E^x _{*,\phi}$ the corresponding expectation 
$\mathbf E^x_{*, \phi}(F)=\int_{G^{\mathbf N}}
F(\omega)d\mathbf P^x_{*,\phi}(\omega)$. In particular,
$$\mathbf P^e_{*,\phi}(S_n=x)=
\mathbf E^e_{*,\phi}(\mathbf 1_{x}(S_n))=\phi^{(n)}(x).$$
In this work, we find it convenient to work (mostly, but not always) 
with the left version of the 
random walk and we will drop the subscript $l$ in the notation introduced above 
unless we need to emphasize the differences between left and right.
Observe that the random walk on the left is a right-invariant process
since $(X_n\dots X_0)y=X_n\dots (X_0y)$. 
When the measure $\phi$ is symmetric in the sense that 
$\phi(x)=\phi(x^{-1})$ for all $x\in G$, its Dirichlet form is defined by
$$\mathcal E_{\phi}(f,f)=\mathcal E_{G,\phi}(f,f)
=\frac{1}{2}\sum_{x,y\in G}|f(yx)-f(x)|^2\phi(y).$$
This is the Dirichlet form associated with random walk on the left, 
$\mathcal E_\phi^l=\mathcal E_\phi$ and $\mathcal E^r_\phi$ is defined similarly.

Given two functions $f_1,f_2$ taking real values but defined on an arbitrary 
domain (not necessarily a subset of $\mathbb R$), we write  $f\asymp g$
to signify that there are constants $c_1,c_2\in (0,\infty)$ such that  
$c_1f_1\le f_2\le c_2f_1$.
Given two monotone real functions $f_1,f_2$, write $f_1\simeq f_2$ 
if there exists $c_i\in (0,\infty)$ such that 
$c_1f_1(c_2t)\le f_2(t)\le c_3f_1(c_4t)$ on the domain of definition of 
$f_1,f_2$ (usually, $f_1,f_2$ will be defined on a neighborhood of $0$ or 
infinity and tend to $0$ or infinity at either $0$ or infinity. In some cases, 
one or both functions are defined only on a countable set such as $\mathbb N$). 
We denote the associated order by $\lesssim$.
Note that the equivalence relation $\simeq$ distinguishes between power 
functions of different degrees 
and between stretched exponentials $\exp(-t^\alpha)$ 
of different exponent $\alpha>0$ but does not distinguishes between 
different rates of exponential growth or decay.  In this paper, we will 
consider functions of the type $\exp(-n/\omega(n))$ where $\omega$
is slowly variant at infinity so that $\omega(ct)\sim \omega(t)$ 
for any fixed $c>0$ and $t$ tending to infinity.

Our main interest in this paper concerns the random walk group invariant
$\Phi_G$, a positive decreasing function defined on $[0,\infty)$ 
up to the equivalence relation $\simeq$ and  
which, according to \cite{PSCstab} 
describes the probability of return of any random walk on the 
group $G$ driven by a measure $\phi$ that is symmetric, 
as generating support, and a finite second moment with respect 
to a fixed word metric on $G$.  
Namely, for any finitely generated group $G$ and any measure $\phi$ 
as just described,
$$\forall\, n=1,2,\dots,\;\;\phi^{(2n)}(e)=\mathbf P_\phi^e(S_{2n}=e)
\simeq \Phi_G(n).$$ 

Given a symmetric probability 
measure $\phi$, set
$$\Lambda_{2,G,\phi}(v)=\Lambda_{2,\phi}(v)=\inf\{ 
\lambda_\phi(\Omega): \Omega\subset G,\; |\Omega|\le v\}$$
where
\begin{equation}\label{def-eig}
\lambda_\phi(\Omega)=
\inf\{ \mathcal E_{\phi}(f,f): \mbox{support}(f)\subset \Omega, \|f\|_2=1\}.
\end{equation}
In words, $\lambda_\phi(\Omega)$ is the lowest eigenvalue of 
the operator of convolution by $\delta_e-\phi$ with Dirichlet boundary
condition in $\Omega$. This operator is associated with the discrete 
time Markov process corresponding to the $\phi$-random walk killed 
outside $\Omega$.  The function $v\mapsto \Lambda_{2,\phi}(v)$ is called the 
$L^2$-isoperimetric profile or spectral profile of $\phi$ 
(it really is an iso-volumic profile). The $L^2$-isoperimetric profile 
of a group $G$ is defined as the 
$\simeq$-equivalence class $\Lambda_G$ of the functions
$\Lambda_\phi$ associated to any symmetric probability measure 
$\phi$ with finite generating support.

The $L^2$-isoperimetric profile $\Lambda_{2,\phi}$
is related to the analogous $L^1$-profile 
\begin{eqnarray*}
\Lambda_{1,G,\phi}(v)&=&\Lambda_{1,\phi}(v)\\
&=&\inf\left\{\frac{1}{2}\sum_{x,y}|f(yx)-f(x)|\phi(y): 
|\mbox{support}(f)|\le v,\;\;\|f\|_1=1\right\}.\end{eqnarray*}
Using an appropriate discrete co-area formula, 
$\Lambda_{1,\phi}$ can equivalently be defined by
$$\Lambda_{1,\phi}(v)= \inf
\left\{|\Omega|^{-1}\sum_{x,y}\mathbf 1_\Omega(x)
\mathbf 1_{G\setminus \Omega}(xy)\phi(y): 
|\Omega|\le v\right\}.$$
If we define the boundary of $\Omega$ to be the set   
$$\partial\Omega=\left\{(x,y)\in G\times G: x\in \Omega, y\in G\setminus \Omega
\right\}$$ and set
$\phi(\partial \Omega)= \sum_{x\in \Omega,xy\in G\setminus \Omega}\phi(y)$
then $\Lambda_{1,\phi}(v)= \inf\{\phi(\partial\Omega)/|\Omega|: |\Omega|\le v\}$.
It is well-known that
\begin{equation} \label{Cheeger}
\frac{1}{2}\Lambda^2_{1,\phi}\le \Lambda_{2,\phi}\le \Lambda_{1,\phi}.
\end{equation}

Recall that the F\o lner function $\mbox{F\o l}_{G,\phi}
$ can be defined by
$$\mbox{F\o l}_{G,\phi}(t)=\inf\{ v: \Lambda_{1,\phi}(v)\le  1/t\}$$
so that $\mbox{F\o l}_{G,\phi}(t)
=\Lambda_{1,\phi}^{-1}(1/t)$ (i.e., $\mbox{F\o l}_{G,\phi}$ is
the right-continuous inverse of the non-decreasing function $\Lambda_{1,\phi}$).

\begin{notation} By elementary comparison arguments, for any two symmetric 
finitely supported probability measure $\phi_1,\phi_2$ with generating support on a group $G$,
we have 
$$\Lambda_{1,G,\phi_1}\simeq \Lambda_{1,G,\phi_2}
\mbox{ and } \Lambda_{2,G,\phi_1}\simeq \Lambda_{2,G,\phi_2}.$$
For this reason we often denote by
$$ \Lambda_{1,G}\;\; (\mbox{ resp. }\; \Lambda_{2,G})$$
the $\simeq$-equivalence class of $\Lambda_{1,G,\phi_1}$ (resp. $\Lambda_{1,G,\phi_2}$) with $\phi$ as above. By abuse of notation, we sometimes write $$\Lambda_{p,G}=\Lambda_{p,G,\phi_1}$$
or understand $\Lambda_{p,G}$ as standing for a fixed representative. 
\end{notation}

\begin{remark} \label{enlarge}
In the definition of $\Lambda_{p,G,\phi}$ (here, $p=1,2$), 
it is not required that $\phi$ generates $G$. In particular, if 
$G_1$ is a subgroup of a group $G_2$ and $\phi$ 
is a symmetric measure supported on $G_1$
then we can consider $\Lambda_{p,G_i,\phi}$ for $i=1,2$. Simple considerations
imply that $\Lambda_{p,G_1,\phi}=\Lambda_{p,G_2,\phi}$. In some instance, 
it might nevertheless be much easier to estimate $\Lambda_{p,G_2,\phi}$ than 
$\Lambda_{p,G_1,\phi}$ directly.  If $\phi$ is finitely supported and
$G_2$ is finitely generated then a simple comparison argument yields 
$\Lambda_{p,G_1,\phi}
\le C(\phi,G_1,G_2) \Lambda_{p,G_2}$.
\end{remark}

\section{A comparison lower bound for $\Lambda$} \setcounter{equation}{0}
\label{sec-low}

Erschler's isoperimetric inequality \cite{Erschleriso} provides a 
lower bound on the isoperimetric profile $\Lambda_{1,G,\mu}$ 
when $G$ is an ordinary 
wreath product and the measure $\mu$ is well adapted to that structure.
See also 
\cite{Saloff-Coste2014} for corresponding results for $\Lambda_{2,G, \mu}$.
It is an important
tool which provides an upper bound on the return probability of the random walk
driven by $\mu$.  However the method does
not work directly on permutational wreath products. Here we develop
a simple but flexible method based on comparison with product
Markov chains. The lower bound on $\Lambda$ obtained in this way is
not always sharp, but it still provides useful information in
many interesting examples of permutation wreath products, 
see explicit estimates in Sections \ref{sec:Explicit-estimates} and 
\ref{sec-NS}.

We will need the following result regarding the isoperimetric profile of product
chains. The statement is an easily consequence of 
Erschler's isoperimetric inequality on wreath product of Markov chains. 
Here, for simplicity, we only consider the case needed for our purpose.

Let $X$ be a finite product $X=\prod_{i\in I}H_{i}$, where $I$ is a finite 
index set and each $H_i$ is a copy a given group $H$. On $H$, fix a symmetric 
probability measure $\eta$. For $i\in I$, let $\eta_i$ be the probability
measure on $X$ defined by $\eta_i(x)=\eta(x_i)$ if  $x=(x_j)_{j\in I}$
with $x_j=e_H$ for $j\neq i$, and $\eta(x)=0$ otherwise.

Let $\zeta$ be the probability measure 
\begin{equation}\label{zetaprod}
\zeta =\frac{1}{|I|}\sum_{i\in I}\eta_i. 
\end{equation}

\begin{proposition}
\label{product} There exists a constant $C_1$ such that, 
for any finite index set $I$, group $H$ and symmetric probability measure $\eta$ as above, 
and for any positive reals $v_0,s$ satisfying $s\in(0, 1/2]$ and
\[ \Lambda_{1,H,\eta}(v_0) >  s\]
the isoperimetric profile function $\Lambda_{1,X,\zeta}$ satisfies 
\[ \Lambda_{1,X,\zeta}(v)\ge s/C_1 \mbox{ for any } v\le C_1^{-1}
\left(v_0\right)^{|I|/C_1}.
\]
 \end{proposition}

\begin{proof}
Equip the index set $I$ with, say, a cyclic group structure and 
its complete graph structure (including self-loops). Embed the 
product $\prod_{i\in I}H_{i}$
into the  wreath product $\left\{ H_{i}\right\} \wr I$ with $H_{i}$ as
lamps over $i$ and $I$ as base. See \cite{Erschler2006}.
Now, $\prod_{i\in I}H_{i}$ can be thought of as the space of lamp 
configurations. Let $Y\subset \prod_{i\in I}H_{i}$ be a finite subset. Define
\[
\widetilde{Y}=\left\{ (f,i):\ f\in Y,\ i\in I\right\} .
\]
Consider the transition kernel 
$\bar{p}$ on $\left\{ H_{i}\right\} \wr I$,
\[
\bar{p}((f,i),(f',i'))=\left\{ \begin{array}{cl}
\frac{1}{2}\eta_i(f'(i)f(i)^{-1}) & \mbox{ if } i=i'\mbox{ and } f=f'\ \mbox{except at }i,\\
\frac{1}{2|I|} & \mbox{ if }f=f'.
\end{array}\right.
\]
Then 
\[
\bar{p}(\tilde{Y},\tilde{Y}^c)=\frac{1}{2}|I|\zeta(\partial Y).
\]

In the base $I$ equipped with its complete graph structure, any set $U$ with
$u$ elements has boundary of weight $|I|^{-1}u(|I|-u)$ so that
$\Lambda_{1,I}(u)\simeq (1-u/|I|)$.

Thus \cite[Theorem 1]{Erschler2006}  together with the hypotheses that 
$ \Lambda_{1,H,\eta}(v_0) >  s$ and the fact that 
$\Lambda^{-1}_{1,I}(s)\ge  |I|(1-s)$ 
yield that
\[
\frac{\bar{p}(\tilde{Y},\tilde{Y}^c)}{|\tilde{Y}|} <s/K
\;\mbox{ implies }\;\;
|\tilde{Y}|
\ge\left(v_0\right)^{(1-s)|I|/K}.
\]
Since
$$|I||Y|= |\tilde{Y}| \;\mbox{ and }\;\;
\frac{\zeta(\partial Y)}{|Y|}=
\frac{\bar{p}(\tilde{Y},\tilde{Y}^c)}{2|\tilde{Y}|},$$
we obtain
\[ \Lambda_{1,X,\zeta}(v)\ge s/K \mbox{ for any } v\le \frac{1}{|I|}
\left(v_0\right)^{|I|/(2K)}
\]
because $s\in (0,1/2]$. 
The desired inequality follows by choosing $C_1>K$ large enough.
\end{proof}

\begin{remark}
Proposition \ref{product} 
concerns expansion of small sets
in $X$. When $X$ is a finite set, the proposition does not provide  
information about expansion of large
sets or the spectral gap. For example suppose that for each $i\in I$, 
$H_i=\mathbb{Z}_{2}$, the bound on $v$
in Proposition \ref{product} is 
 $C_1^{-1}2^{|I|/C_1}= C_1^{-1}|X|^{1/C_1}$. Therefore the stated isoperimetric
inequality is only effective for small sets.
\end{remark}

When $H=\mathbb{Z}_{2}$ or $\mathbb{Z}$, more precise estimates
are known thanks to the sharp isoperimetric inequalities on hypercubes
and Euclidean lattices.

\begin{example}\label{hypercube}
For the hypercube $X=\mathbb{Z}_{2}^{|I|}$, $\eta(1)=1$, from edge
expansion results on the hypercube, we have (see, e.g., \cite{Harper}),
\[
\Lambda_{1,X,\zeta}(v)\ge1-\frac{\log_{2}v}{|I|}.
\]
Thus, for any $K>1$ and $v\le 2^{|I|/K}$, 
\[
\Lambda_{1,X,\zeta}(v)\ge 1-\frac{1}{K}.
\]
This is to be compared with the conclusion provided by 
Proposition \ref{product} which states that
\[
\Lambda_{1,X,\zeta}(v)\ge 
 C^{-1}_1 \mbox{ for all } \; v\le C_1^{-1} 2^{|I|/C_1}. 
\]

\end{example}

\begin{example}\label{lattice}
For the Euclidean lattice $X=\mathbb{Z}^{|I|}$, $\eta(\pm1)=1/2$,
the sharp discrete isoperimetric inequality states that for any finite
set $A$ (see, e.g., \cite[Theorem 6.22]{LyonsPeres})
\[
2\left|\partial_{\zeta}A\right|\ge|A|^{\frac{|I|-1}{|I|}}.
\]
Therefore 
\[
\frac{\left|\partial_{\zeta}A\right|}{|A|}\le s \mbox{ implies }
|A|\ge\left(1/(2s)\right)^{|I|}.
\]
Equivalently,
$$\Lambda_{1,\mathbb Z^{|I|},\zeta}(v)\ge s \mbox{ for all } v\le (1/(2s))^{|I|}.$$
In this case, 
Proposition \ref{product} gives the weaker statement that
$$\Lambda_{1,\mathbb Z^{|I|},\zeta}(v)\ge s/C_1 
\mbox{ for all } v\le C_1^{-1}(1/s)^{|I|/C_1}.$$

\end{example}

For applications we have in mind, it is useful to consider the case when 
$\eta$ is a spread-out measure as in the following example.

\begin{example}
On the Euclidean lattice $X=\mathbb{Z}^{|I|}$, let 
$\eta=\frac{\boldsymbol{1}_{[-|I|,|I|]}}{2|I|+1}$,
i.e., $\eta$ is uniform on the interval $[-|I|,|I|]$ in $\mathbb{Z}$.
In this case, Proposition \ref{product} yields 
\[
\Lambda_{1,X,\zeta}(v)\ge 1/C_1\ \mbox{for all }\; v\le C_1^{-1}|I|^{|I|/C_1}.
\]
The important point here is that this estimate is uniform in $|I|$.
\end{example}

\begin{proposition}\label{volume-diameter}
There exists a constant $C_1$ such that the following holds.
Let $G$ be a finitely generated group equipped with a finite symmetric generating set
$S$. Suppose $X$ is a subgroup of 
$G$ of the form $X=\prod_{i\in I}H_{i}$ with  $H_i\simeq H$ for some group 
$H$. Let $\eta$ be a symmetric probability measure on $H$ with the 
property that 
$$\Lambda_{1,H,\eta}(v_0)\ge s_0$$
for positive real $v_0,s_0$ with $s_0\in (0,1/2]$. 
Let $\zeta$ be defined in terms of $\eta$ by {\em (\ref{zetaprod})} 
and assume further that 
\[
\max_{g\in\mbox{\em supp}(\zeta)}|g|_{(G,S)}\le R. 
\]
Then, letting $\mathbf u_1$ be the uniform probability measure on 
$\{e\}\cup S$, we have 
\[
\mbox{\ensuremath{\Lambda}}_{1,G,\mathbf u_1}(v)\ge\frac{1}{C_1|S|R},\ \mbox{for }v\le
C_1^{-1}\left(v_0\right)^{|I|/C_1}.
\]
\end{proposition}
\begin{proof}This follows immediately from Proposition \ref{product} together 
with an easy comparison argument to pass from the measure 
$\zeta$ defined at (\ref{zetaprod}) to the uniform probability measure 
$\mathbf u_1$ on $\{e\}\cup S$. 
\end{proof}

\begin{example}We show how this bound works on ordinary wreath products.
Consider the wreath product $G=H\wr\Gamma$, where $H$ and $\Gamma$
are finitely generated groups. Given a radius $R$, a natural
choice for $X$ is provided by the lamps over the ball of radius $R$ in the base
group $\Gamma$, that is 
\[
X=\prod_{x\in B_{\Gamma}(e_\Gamma,R)}(H)_{x}.
\]
Take $\eta=\mathbf{u}_{H,R}$ to be the uniform measure on the ball of radius
$R$ in $H$ centered at $e_H$. 
Then Proposition \ref{volume-diameter} implies that
\[
\mbox{\ensuremath{\Lambda}}_{1,G,\mathbf u_1}(v)\ge\frac{1}{C_1 R},\ \mbox{for all}\;
v\le C_1^{-1}\left(|B_{H}(e_H,R)|\right)^{|B_{\Gamma}(e_\Gamma,R)|/C_1}.
\]
In the case when both $\Gamma$ and $H$ are groups whose isoperimetric
function is sharply determined by the volume growth then this bound is
sharp and is equivalent to the bound provided directly by Erschler's 
inequality. For instance, if $\Gamma=H=\mathbb Z$ so that 
$G=\mathbb Z\wr \mathbb Z$, this gives
$$\Lambda_{1,\mathbb Z\wr \mathbb Z}(v)\ge c_1 \frac{\log\log v}{\log v}.$$
\end{example}

Now, we show that a similar bound works on the  permutational wreath products 
$G=H\wr_{\mathcal{S}}\Gamma$ where $\Gamma, H$ are finitely generated groups 
equipped with  finite symmetric generating sets $S_\Gamma, S_H$. 

Let $x\in\mathcal{S}$ be a vertex in the Schreier graph, let $g\in\Gamma$
be a group element such that 
\[
\left|g\right|_{(\Gamma}=d_{\mathcal{S}}(o,x),\ g.o=x.
\]
Then 
\begin{equation}
\left(\mathbf{1}_{h}^{x},e_{\Gamma}\right)=\left(\mathbf{e}_{H}^{\mathcal{S}},g\right)\left(\mathbf{1}_{h}^{o},e_{\Gamma}\right)\left(\mathbf{e}_{H}^{\mathcal{S}},g^{-1}\right),\label{eq:conjugation}
\end{equation}
thus 
\begin{equation}
\left|\left(\mathbf{1}_{h}^{x},e_{\Gamma}\right)\right|_{G}\le2\left|g\right|_{\Gamma}+\left|h\right|_{H}.\label{eq:length}
\end{equation}Let $\mathbf u_{\Gamma,r}$ (resp. $\mathbf{u}_{H,r}$)
denote the uniform probability measure on the ball $B_{\Gamma}\left(e_{\Gamma},r\right)$
in $\Gamma$ (resp. $B_{H}\left(e_{H},r\right)$ in $H$). 

\begin{corollary} Referring to the setting introduced above,
Let $\mathfrak q=\frac{1}{2}(\mathbf u_{\Gamma,1}+\mathbf{u}_{H,1})$
on $G=H\wr_{\mathcal{S}}\Gamma$. For each $r$, 
let $\eta_r$ be a symmetric probability measure on $H$ with 
support in the ball $B(e_H,r)$. Then
\[
\Lambda_{1,G,\mathfrak q}\left(v\right)
\ge\frac{1}{Cr}\ 
\mbox{for }v\le C^{-1}
\left(\Lambda^{-1}_{1,H,\eta_{r}}(1/2)\right)^{\left|B_{\mathcal{S}}(o,r)\right|/C}.
\]
\end{corollary}

\begin{proof}

Consider the measure 
\[
\zeta=\frac{1}{\left|B_{\mathcal{S}}(o,r)\right|}\sum_{x\in B_{\mathcal{S}}(o,r)}
\left(\eta_{r}\right)_{x}.
\]
then by (\ref{eq:length}) we have
\[
\max_{g\in\mbox{supp}(\zeta)}|g|_{G}\le3r.
\]
The conclusion follows from Proposition \ref{volume-diameter}.
\end{proof}

\begin{example}\label{exa-Dpro-low}
On the permutational wreath product $G=(\mathbb{Z}/2\mathbb{Z})\wr_{\mathcal{S}}\Gamma$,
we have a lower bound for $L^{1}$-isoperimetric profile 
\[
\Lambda_{1,G}\left(v\right)\ge\frac{1}{C_1r}\ \mbox{for }v
\le C_1^{-1} 2^{\left|B_{\mathcal{S}}(o,r)\right|/C_1}.
\]
If instead take $\tilde{G}=\mathbb{Z}\wr_{\mathcal{S}}\Gamma$, we take 
$\eta_r$ to be the uniform probability on $\mathbb Z\cap [-r,r]$.
This yields
\[
\Lambda_{1,\widetilde{G}}\left(v\right)\ge\frac{1}{C_1r}\ \mbox{for }v
\le C_1^{-1}r^{\left|B_{\mathcal{S}}(o,r)\right|/C_1}.
\]

For instance, when $\Gamma$ is the infinite Dihedral group 
$D_\infty$ of 
Example \ref{exa-Dinfty}, this gives
$$\Lambda_{1,(\mathbb Z/2\mathbb Z)\wr D_\infty}(v)\gtrsim \frac{1}{\log v}
\mbox{ and }
\Lambda_{1,\mathbb Z\wr D_\infty}(v)\gtrsim \frac{\log\log v}{\log v}.$$

More interesting examples are provided by \cite{Bondarenko}. 
Bondarenko proves that any group $\Gamma$ generated 
by a polynomial automaton of degree $d$ has 
a defining marked Schreier graph $(\mathcal S, o)$  such that  
$V_{\mathcal S}(o,r) \le A^{(\log r)^{d+1}}$. 
The estimate is sharp (\cite[Theorem 4]{Bondarenko}) in the sense that 
Bondarenko gives examples of degree $d$ automaton groups for which 
$B^{(\log r)^{d+1}}\le V_{\mathcal S}(o,r) \le A^{(\log r)^{d+1}}$ for some 
$1<B\le A<\infty$.  The so-called long range group belong to this class, 
with degree $d=1$. See also \cite{Amir2009}.
For such a group $\Gamma$, the Schreier graph volume lower bound 
$B^{(\log r)^{d+1}}\le V_{\mathcal S}(o,r) $ 
 yields (with lamp group $H$ equals to $\mathbb Z/ 2\mathbb Z$ or
$\mathbb Z$ or any group of polynomial volume growth),
$$\Lambda_{1,H\wr_{\mathcal S}\Gamma}
(v)\ge \exp\left(-C (\log \log v)^{1/(d+1)}
\right)$$ 
and, consequently, 
$$\Phi_{H\wr_{\mathcal S}\Gamma}(n) \le \exp
\left(- \frac{n}{\exp\left(C(\log n)^{1/(d+1)}\right)}
\right).$$ 

Regarding $\Lambda_{2}$ upper bounds and $\Phi$ lower bounds, we 
note that \cite{Amir2009} shows that any degree $d$ automaton group embeds into an 
appropriate degree $d$ mother group. For degree $d=0,1,2$, the mother groups 
can be studied by the technique of section \ref{sec-UB} as long as 
some appropriate ``resistance estimates'' can be derived for the associated 
Schreier graph. We will not pursue 
this here, in part because none of the resulting bounds appear to capture 
the real behavior of $\Lambda_1$, $\Lambda_2$ and $\Phi$ for these examples.
We note the degree $d$ automaton groups are known to be amenable for 
$d=0,1,2$ and that the amenability  question is open in degree greater 
than $2$.     
\end{example}

\begin{proposition}\label{PermutationProp}
On the permutational wreath product $G=H\wr_{\mathcal{S}}\Gamma$ 
with $\Gamma,H$ finitely generated,
let $\mathfrak q=\frac{1}{2}(\mu+\nu)$, where 
$\mu=\mathbf u_{\Gamma,1}$ on $\Gamma$ and $\nu$ 
is a symmetric probability measure
on $H$. Then the  $L^{2}$-isoperimetric profile of $\mathfrak q$ on $G$ 
satisfies
\[
\Lambda_{2,G,q}\left(v\right)\ge\frac{1}{C_1r^{2}}\ \mbox{for }v\le 
C^{-1}\left(\Lambda_{2,H,\nu}^{-1}\left(1/r^{2}\right)\right)^{\left|B_{\mathcal{S}}(o,r)\right|/C}.
\]

\end{proposition}

\begin{proof}
For the symmetric probability measure $\nu$ on $H$ with 
$L^2$-isoperimetric profile $\Lambda_{2,H,\nu}$ and any $r>1$, 
\cite[Theorem 4.7]{Saloff-Coste2014} provides a symmetric probability measure 
$\eta_r$ (it depends very much on $\nu$ as well) such that
$$\mathcal E_{H,\nu}\ge cr^{-2} \mathcal E_{H,\eta_{r}} $$
and
$$\Lambda_{2,H,\eta_r}(v)\ge 1/2 
\mbox{ for 
all } v \le C^{-1} \left(\Lambda_{2,H,\nu}(1/r^2)\right)^{|B_{\mathcal S}(o,r)|/C}.$$
Consider the symmetric probability measure
\[
\zeta_r=\frac{1}{\left|B_{\mathcal{S}}(o,r)\right|}\sum_{x\in B_{\mathcal{S}}(o,r)}\left(\eta_{r}\right)_{x}.
\]
Proposition \ref{product} applied to  $\zeta_r$ provides  the estimate
$$\Lambda_{2,G,\zeta_r}(v)\ge 1/C  \mbox{ for 
all } v \le C^{-1} \left(\Lambda_{2,H,\nu}(1/r^2)\right)^{|B_{\mathcal S}(o,r)|/C}$$
for some large constant $C>$. Further,
(\ref{eq:conjugation}) and the property
$\mathcal E_{H,\nu}\ge cr^{-2} \mathcal E_{H,\eta_{r}}$ show that 
(with a different constant $c>0$)
\[
\mathcal{E}_{G,\mathfrak q}\ge c'r^{-2}\mathcal{E}_{G,\zeta_r}.
\]
Putting these two estimates together provides the desired conclusion.
\end{proof}

The next proposition applies the technique of this section in the context of 
finitely generated subgroups of the automorphism group of a rooted tree 
$\mathbb T_{\bar{d}}$.   

\begin{proposition}\label{rigidlamp-1}
Let $\Gamma$ be a finitely generated subgroup of 
$\mbox{\em Aut}\left(\mathbb{T}_{\bar{d}}\right)$ as in {\em Section \ref{sec-tree}}. Let
$u=x_{0}x_{1}...x_{n-1}$ be a vertex in level $n$ of the tree 
$\mathbb{T}_{\bar{d}}$,
and let $\mathcal{S}_{n}(u)$ be the orbital Schreier graph of
$u$ under action of $\Gamma$. Let $\rho_{n}\in\mbox{\em rist}_{\Gamma}(u)$
be a nontrivial element, $\rho_{n}\neq e_{\Gamma}$, of length 
$|\rho_n|_{\Gamma}$ in $\Gamma$. Then there exists a constant $C\ge 1$ such that
\[
\Lambda_{1,\Gamma}(v)\ge\frac{1}{C\max\left\{ \left|\rho_{n}\right|_\Gamma,r\right\} }\ \mbox{for all }v\le C^{-1} 2^{|B_{\mathcal{S}_{n}(u)}(u,r)|/C}.
\]
\end{proposition}
\begin{proof}
For every vertex $v\in\mathcal{S}_{n}(u)$, fix an element $g^{v}\in\Gamma$
such that $g^{v}.u=v$ and 
\[
\left|g^{v}\right|_\Gamma=d_{\mathcal{S}_{n}(u)}(u,v).
\]
In terms of sections and permutations
at level $n$ (see Section \ref{sec-tree}), $g^{v}$  is
$$g^{v} =\left(g_{x}^{v}\right)_{x\in\mathbb{T}^{n}}\sigma^{v}$$
where as
$$\rho_{n} = \left(\tilde{\rho}_x\right)_{x\in\mathbb{T}^{n}}\mbox{id} 
$$
where $\tilde{\rho}_x$ is the identity except for $x=u$ where 
$\tilde{\rho}_x=\rho_n(u)$.
Then 
\[
g^{v}\rho_{n}\left(g^{v}\right)^{-1}=\left(e,..,e,g_{\sigma^{v}(u)}^{v}\rho_{n}(u)g_{\sigma^{v}(u)}^{-1},e,...,e\right) \mbox{id} \]
 where the only nontrivial section is at site 
$\sigma^{v}(u)=v.$
Since $\rho_{n}$ is nontrivial, it follows that the conjugation $g^v_v
\rho_{n}(g^v_{v})^{-1}$
is also a nontrivial element in $\mbox{rist}_{\Gamma}(v)$. 

Now, let $\zeta$ be the symmetric probability measure on
the subgroup $$\left\langle g^{v}\rho_{n}\left(g^{v}\right)^{-1}:\ v\in B(u,r)\right\rangle $$
defined be
\[
\zeta(\gamma)=\frac{1}{2\left|B_{\mathcal S_n(u)}(u,r)\right|}
\sum_{v\in B_{\mathcal S_n(u)}(u,r)}\mathbf{1}_{\left\{ g^{v}\rho_{n}^{\pm1}\left(g^{v}\right)^{-1}\right\} }.
\]
Then $\zeta$ has the form (\ref{zetaprod}) on the product
\[
X=\prod_{v\in B_{\mathcal S_n(u)}(u,r)}\mbox{rist}_{\Gamma}(v).
\]
The desired result follows by comparison between simple random walk 
on $\Gamma$ and $\zeta$ together with Proposition \ref{product}.
\end{proof}

In \cite{ErschlerPiecewise}, Erschler uses vertex stabilizers to estimate the 
F\o lner function of certain groups. It is harder to reach a rigid 
stabilizer but, if one does, one can make use of the product 
structure as described in Proposition \ref{rigidlamp-1}. 
For a concrete application of Proposition \ref{rigidlamp-1}, 
see Corollary \ref{cor-RL}.

\section{Upper-bound for $\Lambda$} \setcounter{equation}{0}
\label{sec-UB}

In this section we present a technique that provides an upper-bound 
on the $L^2$-isoperimetric profile $\Lambda_{2,G,\mathfrak q}$ when 
$G= H\wr_{\mathcal S} \Gamma$
is a permutation wreath product with infinite amenable lamp group $H$
(e.g., $H=\mathbb Z$) and the action of $\Gamma$ on $\mathcal S$ is 
extensively amenable in the sense of \cite{JBMS}.

\begin{definition}[{\cite[Definition 1.1]{JBMS}}] 
The action of a group $\Gamma$ on a set $\mathcal S$ 
is extensively amenable if there is
a $\Gamma$-invariant mean on $\mathcal P_f (\mathcal S)$ 
giving full weight to the collection of subsets that contain
an arbitrary given element of $\mathcal P_f (\mathcal S)$.
\end{definition}

This property appears implicitly in \cite{Juschenko2013a} 
where it is used to prove the 
amenability of the topological full group of any minimal Cantor system. 
This argument was later extended in \cite{Juschenko2013,JBMS}. 

Here we use a basic version of this idea to obtain explicit 
upper-bound on the $L^2$-isoperimetric profile $\Lambda_{2,G,\mathfrak q}$.
Qualitatively, this is equivalent to giving upper bound on 
the F\o lner function.

\subsection{A unifying framework\label{sub:A-unifying-framework}}

Let $\Gamma$ be a finitely generated group acting on a space $X$.
In the examples we consider, $X$ will be a countable set, 
the vertex set of a graph. Let 
$$S=\{s_1^{\pm1},\dots,s_k^{\pm 1}\}$$
be a finite symmetric generating set 
of $\Gamma$. In what follows we need to also consider the finite alphabet
of $2k$ distinct letters 
$$\mathbf S=\{s_1^{\pm 1},\dots,s^{\pm 1}_k\}.$$
Note that we do use the same notation $s_i^{\pm 1}$ for the 
letters and their evaluation in $\Gamma$. 
We let
$$\mathbf S^{(\infty)}$$
be the set of all finite words on $\mathbf S$.

Let
$o$ be a point in $X$ chosen as a reference point. Let $\mathcal{S}$
denote the orbital Schreier graph of $o$ under the action of $\Gamma$.
It is well understood that the inverted orbits of certain points in $X$
play a key role in understanding the group $\Gamma$, see for example
\cite{Bartholdi2012}. In some examples we will need to keep track
of inverted orbits of more than one point. 

\begin{notation}[Inverted orbit, $\mathcal{O}(w;J),$ $\Omega(J,A,B)$]
Let $w=w_{1}w_{2}...w_{l}$, $w_{i}\in \mathbf S$, 
be an arbitrary word of length $l$.
\begin{itemize}
\item Given $x\in X$, the inverted orbit $\mathcal{O}(w;x)$
of $x$ under $w$ is the subset of $X$ defined by
\[
\mathcal{O}(w;x)
=\left\{ {w}_{1}\cdots {w}_{l}\cdot x,\ {w}_{1}\cdots
{w}_{l-1}\cdot x,\dots, {w}_{1}\cdot x,\ x\right\} .
\]
\item More generally, given a countable subset $J$ of $X$, set 
\[
\mathcal{O}(w;J)=\bigcup_{x\in J}\mathcal{O}(w,x).
\]

\item For any two subsets 
$J\subset B\subset X$, define
$\Omega(J,B)$ to be the subset of $\mathbf S^{(\infty)} $ defined by
\[
\Omega(J,B)
:=\left\{ w\in \mathbf S^{(\infty)}:\ \mathcal{O}(w,J)\subset B\right\} .
\]
Let $\overline{\Omega}\left(J,B\right)\subset \Gamma$ 
be the image of $\Omega\left(J,B\right)\subset \mathbf S^{(\infty)}$  
under evaluation in $\Gamma$. 
\end{itemize}
\end{notation} 

\begin{definition}[{$(J,B)$-admissible function}]\label{adm-F}
Equip $\mathcal{P}_{f}(X)$ with the usual counting
measure. Given subsets $J\subset B$ of $X$, we say that a function 
$$F:\mathcal{P}_{f}(X)\to\mathbb{R}$$ is
$(J,B)$-admissible if there exist a subset $A\subset X$ such that
for any  $Y\in\mbox{supp} (F)$ we have 
$Y=g\cdot A$
for some $g\in G$ and $J\subset Y=g\cdot A\subset B$.
\end{definition}

\begin{definition}[Hypothesis ($\Omega$)] \label{def-Omega}
We say that $(\Gamma,X,o)$ satisfies  ($\Omega$) if 
the following conditions hold: 
\begin{enumerate}
\item There exists a sequence of couples $ ((J_{n},B_{n}))$ 
where $J_{n},B_{n}$ are finite subsets of $\mathcal{S}$ such that
\[
o\in J_{n}\subset B_{n}\ \mbox{and }\;\;\bigcup_{n}\overline{\Omega}(J_{n},B_{n})
=\Gamma.
\]

\item There exists a sequence $(\mathbf \Gamma^{n})$ of amenable groups  
and maps 
\[
\vartheta_{n}:\Omega(J_{n},B_{n})\to \mathbf \Gamma^{n}
\]
 such that \begin{itemize}
\item For any $s\in \mathbf S$ and integer $n$, 
we have $s,s^{-1}\in \Omega(J_n,B_n)$. 
\item For any pair $w_1,w_2\in\Omega(J_{n},B_{n})$ whose evaluations in 
$\Gamma$ are equal, we have 
$$\vartheta_{n}(w_1)=\vartheta_{n}(w_2).$$
\item  For any pair  $w_{1},w_{2}$ such that $w_1, w_2,w_{1}w_{2}\in\Omega\left(J_{n},B_{n}\right)$,  we have
\[
\vartheta_{n}(w_{1}w_{2})=\vartheta_{n}(w_{1})\vartheta_{n}(w_{2}).
\]
\end{itemize}
\end{enumerate}
\end{definition}

\begin{remark} Under Hypothesis $(\Omega)$, each of the map $\vartheta_n$
induces an injective map from $\overline{\Omega}(J_n,B_n)$ to
$\mathbf \Gamma^n$. For this reason we will refer to $\vartheta_n$ 
as the ``local embeddings'' provided by Hypothesis ($\Omega$).
\end{remark}

The point of the hypothesis ($\Omega$) is that the groups 
$\mathbf \Gamma^{n}$ are potentially much easier to understand
than the group $\Gamma$. Typically, the map $\vartheta_{n}$ 
cannot be extended globally to $\Gamma$, but in computations when 
we can restrict attention to the finite subset
$\overline{\Omega}(J_{n},B_{n})$ of $\Gamma$, 
elements there can be identified with
elements in $\mathbf\Gamma^{n}$ via the map $\vartheta_{n}$. This will allow us
to combine functions on $\mathcal{P}_f(\mathcal S)$ and
functions on $\mathbf\Gamma^{n}$ to bound the isoperimetric profile 
$\Lambda_{\mathbb{Z}\wr_{\mathcal{S}}\Gamma}$
of the permutation wreath product $\mathbb{Z}\wr_{\mathcal{S}}\Gamma$. 

\begin{example} \label{exa-Dinfty1}
Consider the Dihedral group $D_{\infty}=<s,t>$ of Example 
\ref{exa-Dinfty}
defined in terms of the Schreier graph of Figure \ref{D4}. We show how
hypothesis $(\Omega)$ is satisfies in this simple case. Set $J=\{o=0\}$ 
be the left-most vertex in $\mathcal S=\{0,1,2,\dots\}$ and set 
$B_n=\{0,\dots, 2^n-1\}$.  The set $\Omega(o,B_n)$ is the collection of 
words $w=w_1\dots w_i$ on $s^{\pm 1}$ and $t^{\pm 1}$ so that 
for any $j\le i$, the associated reduced word on $s,t$ (each of order two)
is   of length at most $2^{n}$ if the reduced  word ends in $t$ and $2^{n}-1$
if the reduced word end in $s$.  

For $\mathbf \Gamma^n$, take the finite dihedral group $D_{2^n}$ as discussed 
in Example \ref{exa-finiteD} 
and let $\vartheta_n$ be the natural projection map 
$\mathbf S^{(\infty)}\to D_\infty\to D_{2^n}$. The required properties are easily seen to be 
satisfied.   
\end{example}

Recall that the elements of the group $G=\mathbb Z\wr_{\mathcal S} \Gamma$  
are pairs $(\Upsilon,g)$ where $g\in \Gamma$ and $\Upsilon$ is a 
finitely supported $\mathbb Z$-lamp configuration  on $\mathcal S$. 
The action of $\Gamma$ on $X$ is easily extended to an action of $G$ by setting 
$(\Upsilon, g)\cdot x=g\cdot x$ for all $(\Upsilon,g)\in G$ and $x\in X$.

\begin{lemma}[Lamps control inverted
orbits] \label{controls-inverted}
Assume that $(\Gamma,X,o)$ satisfies Hypothesis $(\Omega)$. Fix 
a function $F_{n}$ on $\mathcal{P}_{f}(\mathcal{S})$
which is $\left(J_{n},B_{n}\right)$-admissible. Set
\[
\widetilde{U}_{n}=\left\{ f\in\oplus_{x\in\mathcal{S}}(\mathbb{Z})_{x}:
\mbox{\em supp}(f)\in\mbox{\em supp}(F_{n})\right\} .
\]
For  
each integer $n$, consider
$g_{n,0}=(\Upsilon_{n,0},e)\in \mathbb Z\wr_{\mathcal S} \Gamma$ 
and a sequence  $( g_{i}) _{i\ge 1}$ of group elements
in $\mathbb{Z}\wr_{\mathcal{S}}\Gamma$
with $g_{i}\in\left\{ (\pm\mathbf{1}_{1}^{o},e)\right\} \cup S $
for all $i\ge 1$. Set 
\[
g_{k}\dots g_{1}g_{n,0}  = \left(\Upsilon_{n,k},\gamma_{n,k}\right).\] 
Assume that for some $m$,
\[ \forall\, k=0,1,\dots,m,\;\;
\Upsilon_{n,k}\in\widetilde{U}_{n}.\] 
Then, considering the $g_j$ both as formal letters and group elements, 
\[\mathcal{O}\left(g_{m}\dots g_{1};J_{n}\right)\subset
\mbox{\em supp}( \Upsilon_{n,m})=
g_{m}\dots g_{1}\cdot \, \mbox{\em supp}(\Upsilon_0).
\]
\end{lemma}

\begin{proof} In this lemma, $n$ is fixed so, for the proof, we drop
the reference to $n$. The proof is by induction on $m$. 
The claim is obviously true for $m=0$ given the
choice of $\Upsilon_{0}$. Assume that the property holds for length $m$. For
$m+1$, we discuss two cases separately.

In the first case, when $g_{m+1}\in\left\{ \pm\mathbf{1}_{1}^{o}\right\} $,
the move $g_{m+1}$ does not change the inverted orbit of the trajectory, 
that is,
\[
\mathcal{O}\left(g_{m+1}\dots g_{1};J\right)=\mathcal{O}\left(g_{m}\dots g_{1};J\right).
\]
By the induction hypothesis
$$\mathcal{O}\left(g_{m}\dots g_{1};J\right)\subset
\mbox{ supp } \Upsilon_{m}=
g_{m}\dots g_{1}\cdot\, \mbox{supp}(\Upsilon_0).$$
From the definition of $\widetilde{U}$ and since, by assumption, 
$\Upsilon_{m+1}\in\widetilde{U}$, we must have 
$\mbox{supp}(\Upsilon_{m+1})= g\cdot\, \mbox{supp}(\Upsilon_0)$ 
for some $g\in \Gamma$.  
This implies that $|\mbox{supp}(\Upsilon_{m+1})|=|\mbox{supp}(\Upsilon_m)|$
and it follows that
$g_{m+1}$ did not produce any change in the support of the lamp configuration, 
that is
$$\mbox{supp} (\Upsilon_{m+1})=\mbox{supp}(\Upsilon_m)=
g_{m}\dots g_{1}\cdot\, \mbox{supp}(\Upsilon_0)
=g_{m+1}\dots g_{1}\cdot\, \mbox{supp}(\Upsilon_0).$$
It follows that
\[\mathcal{O}\left(g_{m+1}\dots g_{1};J\right)\subset
\mbox{ supp } \Upsilon_{m+1}=
g_{m+1}\dots g_{1}\cdot\, \mbox{supp}(\Upsilon_0)\]
as desired.

In the second case, when $g_{m+1}\in S$, the induction hypothesis gives 
\[
\mathcal{O}\left(g_{m}\dots g_{1};J\right)\subset\mbox{supp}(\Upsilon_{m})
=g_m\dots g_1\cdot\, \mbox{supp}(\Upsilon_0).
\]
Also, we have
\[
\mathcal{\mathcal{O}}\left(g_{m+1}\dots g_{1};J_{n}\right)=
J\cup g_{m+1}\cdot\,\mathcal{\mathcal{O}}\left(g_{m}\dots g_{1};J\right)\]
and\[
\ \mbox{supp}(\Upsilon_{m+1})=g_{m+1}\cdot\,\mbox{ supp}(\Upsilon_{m}).
\]
Since $F$ is $\left(J,B\right)$-admissible  and by assumption,
$\Upsilon_{m+1}=g_{m+1}\cdot \Upsilon_{m}\in\widetilde{U}$, it follows that
$g_{m+1}\cdot\mbox{ supp} (\Upsilon_{m})\in\mbox{supp}(F)$ and
$J\subset g_{m+1}\cdot \mbox{ supp}(\Upsilon_{m})$. 
Therefore, we have
\begin{eqnarray*}
\mathcal{\mathcal{O}}\left(g_{m+1}\dots g_{1};J_{n}\right) & = & J\cup
g_{m+1}\cdot\,
\mathcal{\mathcal{O}}\left(g_{m}\dots g_{1};J\right)\\
&\subset & J\cup \left(g_{m+1}\cdot \,\mbox{ supp} (\Upsilon_{m})\right)\\
&=&
\mbox{supp}(\Upsilon_{m+1})= g_{m+1}\dots g_1\cdot \, \mbox{supp}(\Upsilon_0).
\end{eqnarray*}
Combining the two cases, the lemma follows. 
\end{proof}

\subsection{Comparison of two local graph structures}

In this section we compare two graph structures, one on 
$G=\mathbb{Z}\wr_{\mathcal{S}}\Gamma$
and the other on product set
$\left(\oplus_{x\in\mathcal{S}}(\mathbb{Z})_{x}\right)\times \mathbf \Gamma^{n}$. 
Our goal is
to show that certain certain subgraphs in these two graphs are isomorphic.

In $G=\mathbb{Z}\wr_{\mathcal{S}}\Gamma$,
let $Q=\left\{ \left(\pm \mathbf{1}_{1}^{o},e_{\Gamma}\right)\right\} \cup S$
be the standard switch-or-walk generating set (recall that $S$ is a finite
symmetric generating set of $\Gamma$). The (left) directed Cayley graph of $G$
with respect to $Q$ is obtained by putting a directed edge $\left(g,qg\right)$
between $g\in G$ and $qg$ where $q\in Q$. The edge $(g,qg)$ is
labeled by $q$, where $q\in Q$ is either a move induced by 
$\mathbf{1}_{1}^{o}$ in the lamp configuration, or a move induced by $s\in S$. 
We write $\left(G,Q\right)$
for this directed labeled Cayley graph of $G$ with respect to generating
set $Q$.

On the product 
$\left(\oplus_{x\in\mathcal{S}}(\mathbb{Z})_{x}\right)\times \mathbf \Gamma^{n}$,
define a graph structure by connecting vertices with one of the following
two types of edges: 
\begin{itemize}
\item connect $(f,\gamma)$ and $(f\pm \mathbf{1}_{1}^{o},\gamma)$
with a directed edge labeled with $\pm \mathbf{1}_{1}^{o}$; 
\item connect
$\left(f,\gamma\right)$ and $\left(s^{-1}\cdot f,\vartheta_{n}(s)\gamma\right)$
with a directed edge labeled with $s\in S$. 
\end{itemize}
We denote the resulting
directed graph by $\left(\left(\oplus_{x\in\mathcal{S}}(\mathbb{Z})_{x}\right)\times 
\mathbf \Gamma^{n},P\right)$. Note that it is not necessarily connected. 

On both graphs edges are labeled by either with  $\mathbf{1}_{1}^{o}$ 
(which corresponds to changing the lamp value at $o$) 
or with a generator $s\in S$.
Given a word $\omega=p_{l}...p_{1}$ in the alphabet 
$Q$,
it makes sense to follow $\omega$ as paths on both graphs. Starting
from the same initial lamp configuration on both graphs, from definition
of the graph structures, following the same word $\omega$, the trajectories
of the lamp function $\Upsilon_{k}$ are exactly the same on the two
graphs.

\begin{notation}[Definition of $\mathcal{G}_{n}\left(\Upsilon_{0}\right)$
and $\mathbf \Gamma^{n}\left(\Upsilon_{0}\right)$]\label{component}
Let $F_n$ be an $(J_n,B_n)$-admissible function. Let 
$\Upsilon_{0}\in\oplus_{x\in\mathcal{S}}(\mathbb{Z})_{x}$ be 
such that $\mbox{ supp}(\Upsilon_{0})\in\mbox{ supp}(F_{n})$. 
\begin{itemize}
\item On the graph $(G,Q)$ consider the subgraph with vertex set 
$$\left\{ \left(f,\gamma\right):\ \mbox{ supp}(f)\in
\mbox{ supp}(F_{n})\right\} $$
and edge set made of those original edges of $(G,Q)$ that connects points in
the selected vertex set. Let $\mathcal{G}_{n}\left(\Upsilon_{0}\right)$
be the connected component of this subgraph that contains the
vertex $\left(\Upsilon_{0},e_{G}\right)$. 
\item On the graph $\left(\left(\oplus_{x\in\mathcal{S}}(\mathbb{Z})_{x}\right)\times \mathbf \Gamma^{n},P\right)$,
consider the subgraph with vertex set 
$$\left\{ \left(f,\gamma\right):\ \mbox{ supp}(f)\in\mbox{ supp}(F_{n})\right\} $$
and edge set made of those original edges of $\left(\left(\oplus_{x\in\mathcal{S}}(\mathbb{Z})_{x}\right)\times \mathbf \Gamma^{n},P\right)$ that connects points in
the selected vertex set. Let $\mathbf \Gamma^{n}\left(\Upsilon_{0}\right)$ 
be the connected
component of this subgraph that contains the vertex 
$\left(\Upsilon_{0},e_{\Gamma}\right)$. 
\end{itemize}
\end{notation}

Next we show that the graphs $\mathcal{G}_{n}\left(\Upsilon_{0}\right)$
and $\mathbf \Gamma^{n}\left(\Upsilon_{0}\right)$ can be identified. 

\begin{lemma}\label{graph-iden}
Suppose $(G,X,o)$ satisfies  $(\Omega)$.
Let the subgraphs $\mathcal{G}_{n}\left(\Upsilon_{0}\right)$ and
$\mathbf \Gamma^{n}\left(\Upsilon_{0}\right)$ be defined as above. 
\begin{description}
\item [{(i)}] The map $\vartheta_{n}$ induces a graph isomorphism 
\[
\Theta_{n}:\mathcal{G}_{n}\left(\Upsilon_{0}\right)\to
\mathbf \Gamma^{n}\left(\Upsilon_{0}\right).
\]

\item [{(ii)}] For any vertex $g\in\mathcal{G}_{n}\left(\Upsilon_{0}\right)$,
$g$ is connected by a labeled edge with an exterior vertex in 
$\mathcal{G}_{n}\left(\Upsilon_{0}\right)^{c}$
in $\left(G,Q\right)$ if and only if $\Theta_{n}(g)$ is connected
to $\mathbf \Gamma^{n}\left(\Upsilon_{0}\right)^{c}$ by an edge with
the same label. 
\end{description}
\end{lemma}

\begin{proof}
To simplify notation and since as $\Upsilon_{0}$ is fixed,
we omit the reference to $\Upsilon_{0}$ in
$\mathcal{G}_{n}\left(\Upsilon_{0}\right)$ and $\mathbf \Gamma^{n}
\left(\Upsilon_{0}\right)$. 

Define the map $\Theta_{n}$ as follows. First, set 
$\Theta_{n}\left(\Upsilon_{0},e_{\Gamma}\right)=\left(\Upsilon_{0},e_{H}\right)$.
For a vertex $\left(f,\gamma\right)$ in $\mathcal{G}_{n}$, 
let $g_{k}\ldots g_{1}$
denote a path inside $\mathcal{G}_{n}$ that starts at 
$\left(\Upsilon_{0},e_{\Gamma}\right)$
and ends at $(f,\gamma)$, that is 
\begin{align*}
g_{k}\ldots g_{1}\left(\Upsilon_{0},e_{\Gamma}\right) & =(f,\gamma),\\
g_{j}\ldots g_{1}\left(\Upsilon_{0},e_{\Gamma}\right)\in\mathcal{G}_{n} & \mbox{ for all }1\le j\le k.
\end{align*}
Let $\omega$ be the word $\omega=\widehat{g}_{k}\ldots\widehat{g}_{1}$
where $\widehat{g}_{i}=g_{i}$ if $g_{i}\in S$, and $\widehat{g_{i}}=e$
if $g_{i}\in\left\{ \left(\pm\mathbf{1}_{1}^{o},e_{\Gamma}\right)\right\} $,
that is generators in $S$ are kept and generators in the lamp direction
are removed. With this notation, set
\[
\Theta_{n}\left(f,\gamma\right)=\left(f,\vartheta_{n}(\omega)\right).
\]
Since $g_{k}\ldots g_{1}$ is a path inside $\mathcal{G}_{n}$, Lemma
\ref{controls-inverted} implies that 
\[
\mathcal{O}\left(g_{k}\ldots g_{1};J_{n}\right)=\mathcal{O}\left(\omega;J_{n}\right)\subseteq\mbox{supp}(f)\subset B_{n}.
\]
Thus $\omega\in\Omega\left(J_{n},B_{n}\right)$ and, consequently, $\omega$ belongs to the domain of
$\vartheta_{n}$. To show that the map $\Theta_{n}$ is well defined, suppose
$g'_{l}\ldots g'_{1}$ is another path inside $\mathcal{G}_{n}(\Upsilon_0)$ that
starts at $\left(\Upsilon_{0},e_{\Gamma}\right)$ and ends at $(f,\gamma)$,
then the corresponding $\omega'=\widehat{g}'_{l}\ldots\widehat{g}'_{1}$
is an element in $\Omega\left(J_{n},B_{n}\right)$ as well, and $\omega$ and 
$\omega'$ have the same evaluation in $\Gamma$.
Then Hypothesis ($\Omega$) 
implies that $\vartheta_{n}(\omega)=\vartheta_{n}(\omega')$. That is, $\Theta_{n}(f,\gamma)$ does not depend on the choice of path as desired. 

Hypothesis ($\Omega$) guarantees that $\Theta_{n}$
is injective. To show that $\Theta_{n}$ is surjective, for any vertex
$(f,\gamma)$ in $\mathbf\Gamma^{n}(\Upsilon_0)$, let $p_{l}...p_{1}$ be a path inside
$\mathbf \Gamma^{n}(\Upsilon_0)$ that starts at $\left(\Upsilon_{0},e_{H}\right)$
and ends at $(f,\gamma)$.  If one follows the same path in $(G,Q)$
starting from $\left(\Upsilon_{0},e_{\Gamma}\right)$, by definition
of the two graph structures, the lamp component $\Upsilon_{k}$
along the two paths are exactly the same. Thus the path defined by 
$p_{l}...p_{1}$ in $(G,Q)$ and starting at 
$\left(\Upsilon_{0},e_{\Gamma}\right)$ remains in $\mathcal{G}_{n}(\Upsilon_0)$. It ends
at some element $(f,\gamma)$. Then, by the definition of $\Theta_{n}$,
we have 
$\Theta_{n}\left(f,\gamma\right)=(f,\gamma)$. This proves that 
$\Theta_n$ is surjective.

By Hypothesis $(\Omega$), when $\omega_{1},\omega_{2},\omega_{1}\omega_{2}$
are all in $\Omega\left(J_{n},B_{n}\right)$, we have 
\[
\vartheta_{n}(\omega_{1}\omega_{2})=\vartheta_{n}(\omega_{1})\vartheta_{n}(\omega_{2}).
\]
Hence, if there is a directed edge $\left((f,\gamma),(f',\gamma')\right)$
in $\mathcal{G}_{n}$, then there is a directed edge $\left(\Theta_{n}(f,\gamma),\Theta_{n}(f',\gamma')\right)$
in $\mathbf \Gamma^{n}$ with the same label. Therefore $\Theta_{n}$
is a bijection between $\mathcal{G}_{n}$ and $\mathbf \Gamma^{n}$ that
preserves edge relations. It is a graph isomorphism as desired.

Part (ii) is an immediate consequence of the fact that the lamp
component changes in the same way on both graphs $(G,Q)$ and 
$\left(\left(\oplus_{x\in\mathcal{S}}(\mathbb{Z})_{x}\right)
\times \mathbf \Gamma^{n},P\right)$ and that the boundaries of 
$\mathcal{G}_{n}$ and 
$\mathbf \Gamma^{n}$ are defined in terms of the lamp component only, 
without referring to $\gamma$. 
\end{proof}

\subsection{Test functions} \label{sec-Test}

We now consider transition kernels and test functions
on the graphs  $(G,Q)$ and $\left(\left(\oplus_{x\in\mathcal{S}}(\mathbb{Z})_{x}\right)\times \mathbf \Gamma^{n},P\right)$. 

We equip $G=\mathbb{Z}\wr_{\mathcal{S}}\Gamma$ with the right-invariant 
Markov transition kernel associated with the
switch-or-walk measure $\mathfrak{q}=\frac{1}{2}(\eta+\mu)$ where
$\eta$ is uniform on $\left(\pm\mathbf{1}_{1}^{o},e_{\Gamma}\right)$
and $\mu$ is uniform on $S$. 

On the graph $\left(\left(\oplus_{x\in\mathcal{S}}(\mathbb{Z})_{x}\right)
\times \mathbf \Gamma^{n},P\right)$,
define a Markov transition kernel $\mathfrak{p}$ as 
\[
\mathfrak{p}\left((f,\gamma),(f',\gamma')\right)=
\frac{1}{4}\mathbf{1}_{\left\{ f\pm\mathbf{1}_{1}^{o}\right\} }(f')+\frac{1}{2}\sum_{s\in S}\mathbf{1}_{\{\vartheta_{n}(s)\gamma\}}(\gamma')\mathbf{1}_{s^{-1}\cdot f\}}(f')\mu(s).
\]
That is $\mathfrak{p}$ either changes the configuration at $o$ by
$\pm1$, or translate $(f,\gamma)$ by a generator $s\in S$.
It follows that for $(f,\gamma)$, $(f',\gamma')$ in $\mathcal{G}_{n}$,
\[
\mathfrak p\left(\Theta_{n}(f,\gamma),\Theta_{n}(f',\gamma')\right)=
\mathfrak q\left(\left(f',\gamma'\right)(f,\gamma)^{-1}\right).
\]

We now focus on the subgraphs $\mathcal G_n(\Upsilon_0) $ and 
$\mathbf \Gamma^n(\Upsilon_0)$ introduced in Notation  \ref{component}. 
The ingredients we will use to build a test function supported on $\mathcal{G}_{n}\left(\Upsilon_{0}\right)$
include 
\begin{description}
\item [{(a)}] a $\left(J_{n},B_{n}\right)$-admissible function $F_{n}$
on $\mathcal{P}_{f}(X)$, 
\item [{(b)}] a function $\psi_{n}$ on $\mathbf \Gamma^{n}$ with finite support.
\end{description}
Given a symmetric probability measure $\mu$
supported on the generating set $S$, we introduce the Rayleigh quotient
of the functions $F_{n}$ and $\psi_{n}$ as 
\begin{align*}
\mathcal{Q}_{\mathcal{P}_{f}(\mathcal{S}),\mu}\left(F_{n}\right) & :=\frac{\sum_{s\in\Gamma}\mu(s)\left\Vert s\cdot F_{n}-F_{n}\right\Vert _{L^{2}(\mathcal{P}_{f}(X))}^{2}}{\left\Vert F_{n}\right\Vert _{L^{2}(\mathcal{P}_{f}(X))}^{2}}\\
\mathcal{Q}_{\mathbf \Gamma^{n},\mu}\left(\psi_{n}\right) & :=\frac{\sum_{s\in\Gamma,x\in \mathbf \Gamma^{n}}\mu(s)\left|\psi_{n}(x)-\psi_{n}\left(\vartheta_{n}(s)x\right)\right|^{2}}{\left\Vert \psi_{n}\right\Vert _{L^{2}(\mathbf \Gamma^{n})}^{2}}.
\end{align*}
We also set
\[
Q_{\mu}(n):=\mathcal{Q}_{\mathcal{P}_{f}(\mathcal{S}_{n}),\mu}\left(F_{n}\right)+\mathcal{Q}_{\mathbf\Gamma^{n},\mu}\left(\psi_{n}\right).
\] 

\begin{remark} Assuming that there is a sequence of $(J_n,B_n)$-admissible
functions $F_n$ with $\mathcal Q_{\mathcal P_f(\mathcal S),\mu}(F_n)$ 
tending to zero implies that the action of $\Gamma$ on $\mathcal S$ is 
extensively amenable. Lemma \ref{lem-ResA} in the Appendix below provide appropriate test functions $F_n$ based on resistance estimates.
\end{remark}

\begin{proposition}\label{function}
Assume that $(\Gamma,X,o)$ satisfies  $(\Omega)$.
Let $F_{n}$ be a $\left(J_{n},B_{n}\right)$-admissible function
on $\mathcal{P}_{f}(\mathcal{S})$ and $\psi_{n}$ a function on $\mathbf
\Gamma^{n}$
with finite support. Then for the standard switch-or-walk measure 
$\mathfrak q=\frac{1}{2}(\eta+\mu)$
on $G=\mathbb{Z}\wr_{\mathcal{S}}\Gamma$, we have
\[
\Lambda_{2,G,q}(v)\le  Q_{\mu}(n)\]
for all $v$ such that
\[v\ge |\vartheta_n(\Omega(J_n,B_n)) \cap \mbox{\em supp}(\psi_n)h_0^{-1}|
Q_{\mu}(n)^{-\left|B_{n}\right|}\]
for some $h_0\in \mbox{\em supp}(\psi_n)\subset
\mathbf\Gamma^n$.
\end{proposition}
Note that
\[|\vartheta_n(\Omega(J_n,B_n)) \cap \mbox{supp}(\psi_n)h_0^{-1}|\le 
\min\left\{ \left|\overline{\Omega}\left(J_{n},B_{n}\right)\right|,\left|\mbox{supp}(\psi_{n})\right|\right\}
\]
and also
\[|\vartheta_n(\Omega(J_n,B_n)) \cap \mbox{supp}(\psi_n)h_0^{-1}|\le 
\left|\overline{\Omega}\left(J_{n},B_{n}\right)\cap
\mbox{supp}(\psi_{n})\mbox{supp}(\psi_{n})^{-1}
\right|.
\]
Both upper bounds are independent of $h_0$.
\begin{proof}

The proof consists of two steps. First we find a test function whose
support is contained in the connected component 
$\mathbf \Gamma^{n}\left(\Upsilon_{0}\right)$;
in the second step the test function is transferred to 
$\mathcal{G}_n\left(\Upsilon_{0}\right)$
via the isomorphism $\Theta_{n}$ of Lemma \ref{graph-iden}. 

Recall that we are given a $(J_n,B_n)$-admissible function $F_n$ on
on $\mathcal{P}_{f}(\mathcal{S})$ and  a 
function $\psi_{n}$ with finite support on $\Gamma^{n}$. Using this data, 
we construct a test function
$\Phi_{n}$ on $\left(\oplus_{x\in\mathcal{S}}(\mathbb{Z})_{x}\right)\times 
\mathbf \Gamma^{n}$
by setting 
\[
\Phi_{n}\left(f,\gamma\right)=\psi_{n}(\gamma)F_{n}(\mbox{supp}f)\prod_{x\in\mbox{supp}f}\mathbf{1}_{\left[1,Q_{\mu}(n)^{-1}\right]}(f(x)).
\]
Since $F_n$ is $(J_n,B_n)$-admissible, there exists a finite set $A_n$ 
such that every set $A$ in the support of $F_n$ is of the form $A=g_A.A_n$.
It follows that 
\[
\left\Vert \Phi_{n}\right\Vert _{L^{2}\left(
\left(\oplus_{x\in\mathcal{S}}(\mathbb{Z})_{x}\right)\times \mathbf 
\Gamma^{n}\right)}^{2}=
\left\Vert \psi_{n}\right\Vert _{L^{2}\left(\mathbf
\Gamma^{n}\right)}^{2}\left\Vert F_{n}\right\Vert _{L^{2}\left(\mathcal{P}_{f}\left(\mathcal{S}\right)\right)}^{2}\left\lfloor Q_{\mu}(n)^{-1}\right\rfloor ^{\left|A_{n}\right|}.
\]
Setting $[1,Q_\mu(n)^{-1}]=I_n$,
we also note that
\begin{eqnarray*}
\lefteqn{\sum_{(f,\gamma)}
\left(\psi_{n}(\gamma)F_{n}(\mbox{supp}f)\right)^{2}\prod_{x\in\mbox{supp}f}\mathbf{1}_{I_n}(f(x))\mathbf{1}_{I_n^c}(f(o)+1)} &&\\
& =&
\left\Vert \psi_{n}\right\Vert _{L^{2}\left(\mathbf\Gamma^{n}\right)}^{2}\left\Vert F_{n}\right\Vert _{L^{2}\left(\mathcal{P}_{f}\left(\mathcal{S}\right)\right)}^{2}\left\lfloor Q_{\mu}(n)^{-1}\right\rfloor ^{\left|A_{n}\right|-1}.
\end{eqnarray*}

the energy of the function $\Phi_{n}$ with respect to the transition
kernel $\mathfrak{p}$ can be estimated as
\begin{eqnarray*}
2\lefteqn{\mathcal{E}_{\mathfrak p}\left(\Phi_{n},\Phi_{n}\right) =
\sum_{(f,h),(f'h')}\left(\Phi_{n}(f,h)-\Phi_{n}(f',h')\right)^{2}\mathfrak{p}((f,h),(f',h'))}&&
\\
&\le & \frac{1}{2}
\sum_{(f,h)}\sum_{\epsilon= \pm1 }\left(\psi_{n}(h)F_{n}(\mbox{supp}f)\right)^{2}\prod_{x\in\mbox{supp}f}\mathbf{1}_{I_n}(f(x))\mathbf{1}_{I_n^c}(f(o)+\epsilon)\\
&+ & \sum_{(f,h)}\sum_{s\in S}\left(\psi_{n}(h)-\psi_{n}\left(\vartheta_{n}(s)h\right)\right)^{2}F_{n}(\mbox{supp}f)^{2}\prod_{x\in\mbox{supp}f}\mathbf{1}_{I_n}(f(x))\mu(s)\\
&+ & \sum_{(f,h)}\sum_{s\in S}\psi_{n}(h)^{2}\left(F_{n}(\mbox{supp}f)-F_{n}(s.\mbox{supp}f)\right)^{2}\prod_{x\in\mbox{supp}f}\mathbf{1}_{I_n}(f(x))\mu(s)\\
&\le & 2 Q_{\mu}(n)
\left\Vert \Phi_{n}\right\Vert _{
L^{2}\left(\left(\oplus_{x\in\mathcal{S}}(\mathbb{Z})_{x}\right)\times \mathbf\Gamma^{n}
\right)}^{2}
\end{eqnarray*}

The Markov chain with transition kernel $\mathfrak{p}$ decomposes
the space $\left(\oplus_{x\in\mathcal{S}}(\mathbb{Z})_{x}\right)\times 
\mathbf \Gamma^{n}$
into connected components. There must exist a $P$-connected subgraph
$\mathcal C=\mathcal{C}_{n}^{P}$ with vertex set contained in 
$\mbox{supp}(\Phi_{n})$
such that the restriction of $\Phi_{n}$ to $\mathcal{C}$,
denoted by $\Phi_{n}^{\mathcal{C}}$, satisfies 
\[
\frac{\mathcal{E}_{\mathfrak p}\left(\Phi_{n}^{\mathcal{C}},\Phi_{n}^{\mathcal{C}}\right)}{\left\Vert \Phi_{n}^{\mathcal{C}}\right\Vert _{2}^{2}}\le\frac{\mathcal{E}_{p}\left(\Phi_{n},\Phi_{n}\right)}{\left\Vert \Phi_{n}\right\Vert _{2}^{2}}.
\]
Pick some  $(f_{0},h_{0})\in\mathcal{C}$ and translate
the $\mathbf\Gamma^{n}$ component on the right by $h_{0}^{-1}$ by setting
\begin{align*}
\mathcal{C}_{h_{0}} & =\left\{ \left(f,hh_{0}^{-1}\right):\ (f,h)\in\mathcal{C}
\right\} ,\\
\Phi_{n,h_{0}}^{\mathcal{C}}(f,h) & =\begin{cases}
\Phi_{n}^{\mathcal{C}}(f,hh_{0}) & \mbox{if }(f,h)\in\mathcal{C}_{h_{0}}\\
0 & \mbox{otherwise}.
\end{cases}
\end{align*}
If the component $\mathcal{C}$ contains an element of the
form $\left(f_{0},e_{\mathbf\Gamma^{n}}\right)$, we can pick that point and
$\mathcal{C}, \Phi_n^{\mathcal C}$ stay as they were. 
In any case, one readily checks that this
right translation does not change the Rayleigh quotient of the function,
since the transition kernel acts on the left. Since 
$\mbox{supp}(f_{0})\in\mbox{supp}(F_{n})$, we have obtained a test
function $\Phi_{n,h_{0}}^{\mathcal{C}}$ whose support is contained
in the subgraph $\mathbf\Gamma^{n}\left(f_{0}\right)$ given by Notation 
\ref{component}. This test function satisfies
$$\mathcal{E}_{\mathfrak p}\left(\Phi^{\mathcal C}_{n,h_0},
\Phi^{\mathcal C}_{n,h_0}\right)
\le  Q_{\mu}(n)
\left\Vert \Phi^{\mathcal C}_{n,h_0}
\right\Vert _{L^{2}\left(\left(\oplus_{x\in\mathcal{S}}(\mathbb{Z})_{x}\right)\times 
\mathbf\Gamma^{n}
\right)}^{2}.$$

We now proceed with the second step of the proof. 
Lemma \ref{graph-iden} describes a graph isomorphism
$\Theta_{n}:\mathcal{G}_{n}\left(f_{0}\right)
\to\mathbf \Gamma^{n}\left(f_{0}\right)$.
We can then define a test function $\Psi_{n}$ supported on the subgraph
$\mathcal{G}_{n}\left(f_{0}\right)$ in the Cayley graph of $G$ by
\[
\Psi_{n}\left(f,\gamma\right)=\Phi_{n,h_{0}}^{\mathcal{C}}\left(\Theta_{n}\left(f,\gamma\right)\right).
\]
Since the domain of $\Phi_{n,h_{0}}^{\mathcal{C}}$ is contained in
the range of $\Theta_{n}$, it follows that 
\[
\Phi_{n,h_{0}}^{\mathcal{C}}(f,h)=\begin{cases}
\Psi_{n}\left(\Theta_{n}^{-1}\left(f,h\right)\right) & \mbox{if }(f,h)\in\Gamma^{n}(f_{0})\\
0 & \mbox{otherwise}.
\end{cases}
\]
Using Lemma \ref{graph-iden}, one checks that indeed
\begin{align*}
\mathcal{E}_{q}\left(\Psi_{n},\Psi_{n}\right) & =\mathcal{E}_{p}\left(\Phi_{n,h_{0}}^{\mathcal{C}},\Phi_{n,h_{0}}^{\mathcal{C}}\right)\\
\left\Vert \Psi_{n}\right\Vert _{L^{2}\left(G\right)}^{2} & =\left\Vert \Phi_{n,h_{0}}^{\mathcal{C}}\right\Vert _{L^{2}\left(\left(\oplus_{x\in\mathcal{S}}(\mathbb{Z})_{x}\right)\times \mathbf \Gamma^{n}\right)}^{2}.
\end{align*}
It follows that 
\[
\frac{\mathcal{E}_{q}\left(\Psi_{n},\Psi_{n}\right)}{\left\Vert \Psi_{n}\right\Vert _{L^{2}\left(G\right)}^{2}}\le Q_{\mu}(n).
\]

To obtain the estimate on $\Lambda_{2,G,\mathfrak q}(v)$ stated in the 
proposition, it suffices to use an upper bound on 
the volume of the support of $\Psi_{n}$. By Lemma \ref{controls-inverted},
we have
\begin{eqnarray*}
\left|\mbox{supp}(\Psi_{n})\right|  &=&
\left|\mbox{supp}(\Phi_{n,h_{0}}^{\mathcal{C}})\right|\\
&\le & |\vartheta_n(\Omega(J_n,B_n)) \cap \mbox{supp}(\psi_n)h_0^{-1}|
\, Q_{\mu}(n)^{-\left|B_{n}\right|}.
\end{eqnarray*}
Here we also used the fact that, since $F_n$ is $(J_n,B_n)$-admissible, 
$|A_n|\le |B_n|$. \end{proof}

\begin{example}[The $L^2$-isoperimetric profile of $\mathbb Z\wr D_\infty$]
\label{exa-Dinfty2}
To illustrate how the technique described above works in a simple case, 
we consider the toy example of the infinite dihedral group 
$D_{\infty}=\left\langle s,t|s^{2}=t^{2}=1\right\rangle $ of Example 
\ref{exa-Dinfty}. Think of it has defined by the Schreier graph
of Figure \ref{D4}. In Example \ref{exa-Dinfty1}, we noted that Hypothesis 
$(\Omega)$ is satisfied with $J_n=\{o=0\}$, $B_n=\{0,\dots,2^{n}-1\}$
and $\mathbf \Gamma^n= D_{2^n}$.  
Let $\mu$ be the uniform probability measure on 
$\{s,t\}$. To construct our test function as in Proposition 
\ref{function}, we set $\psi_n\equiv 1$ on $\Gamma^n$ so that 
$Q_{\Gamma^n,\mu}(\psi_n)=0$. 

The function $F_n:\mathcal P_f(\mathcal S)\to \mathbb R$ 
must be $(J_n,B_n)$-admissible which means it is supported 
on subsets  $Y$ of the form $Y=g\cdot A_n$ with $o\in Y \subset B_n$ 
for some $A_n$.  Let
$$A_n= \{0,\dots, 2^{n-1}-1\}.$$
Define $F_n$ so that $F_n(Y)=0$ unless $Y=g\cdot A_n$ where $g$ is a 
reduced word in $s,t$ of length $\ell\le 2^{n-1}$ which terminates with $t$, in which case
$$F_n(Y)= 1-  2^{-n+1}l.$$
We need to verify that $F_n$ is $(J_n,B_n)$-admissible. This follows
by inspection because
$g \cdot A_{n}$ is 
\[\{0,...,2^{n-1}-1-l\}\cup\{2^{n-1}-l-1+2k:\ 1\le k\le l\} \]
when $0\le l\le2^{n-1}-1$ and
$\{2k-1:\ 1\le k\le l\}$
when $l= 2^{n-1}$.

The Rayleigh quotient of $F_{n}$ can be computed
\[
\mathcal{Q}_{\mathcal{P}_{f}(\mathcal{S}),\mu}\left(F_{n}\right)\sim
\frac{3}{\left(2^{n-1}\right)^{2}}.
\]

To apply Proposition \ref{function}, observe that 
$|\overline{\Omega}(J_n,B_n)|$, $|\mbox{supp}(\psi_n)| $ and $|B_n|$ are all 
equal to $2^n$. This yields 
\[
\Lambda_{2,\mathbb{Z}\wr_{\mathcal{S}}D_{\infty},q}(v)\le\frac{C}{\left(2^n\right)^{2}},\ \mbox{for any }v\ge C^{-1} 2^{n2^{n}/C},
\]
that is (with a different $C$),
\[\Lambda_{2,\mathbb{Z}\wr_{\mathcal{S}}D_{\infty},q}
(v)\lesssim 
\left(\frac{\log\log v}{\log v}\right)^2.\]
This gives the return probability lower bound 
\[
\Phi_{\mathbb{Z}\wr_{\mathcal{S}}D_{\infty}}(n)\gtrsim
\exp\left(-n^{\frac{1}{3}}\log^{\frac{2}{3}}n\right).
\]
By Example \ref{exa-Dpro-low} and the Cheeger inequality (\ref{Cheeger}), 
have a matching lower bound for $\Lambda_{2,\mathbb{Z}\wr_{\mathcal{S}}D_{\infty},q}$ 
so that 
\[\Lambda_{2,\mathbb{Z}\wr_{\mathcal{S}}D_{\infty},q}
(v)\simeq
\left(\frac{\log\log v}{\log v}\right)^2.\]
and
\[
\Phi_{\mathbb{Z}\wr_{\mathcal{S}}D_{\infty}}(n)\simeq
\exp\left(-n^{\frac{1}{3}}\log^{\frac{2}{3}}n\right).
\]

\end{example}

\section{Bubble groups } \setcounter{equation}{0}
\label{sec:Explicit-estimates}\label{sec-bubble}
This section is devoted to a family of groups considered in \cite{Kotowski2014} 
where the name ``bubble group'' is used.   The marked Schreier graph pictured in 
Figure \ref{B1} defines a finite bubble group.

\begin{figure}[h]
\begin{center}\caption{The marked Schreier graph of the finite 
bubble group with $\mathbf a= (2,5,9), \mathbf b=(3,3)$. 
The generators $a,b$ each acts clockwise along its respective cycle. 
Each vertex of degree $2$ carries a self-loop marked $b$ which is not shown 
except at the root. }\label{B1}
\begin{picture}(200,100)(0,0) 

\put(-10,50){\line(1,0){10}}
\put(-10,50){\circle*{3}}

\put(10,50){\circle{20}}
\put(0,50){\circle*{3}}
\multiput(54.5,21.1)(0,-44.2){2}{\put(10,50){\circle{20}}}
\put(54,19.5){\put(3,66){\makebox(0,0){$b$}}
\put(1,60){\vector(2,1){10}}
}
\put(54,-45){\put(3,60){\makebox(0,-5){$b$}}
\put(9,59){\vector(-2,1){10}}
}

\put(-13,50){\circle{6}}
\put(-20,50){\makebox(0,0){$b$}}
\put(-6,54){\makebox(0,0){$a$}}

\put(13,50){\makebox(0,0){$b$}}
\put(17,54){\vector(0,-1){9}}
\put(3,66){\makebox(0,0){$b$}}
\put(0,58.5){\vector(2,1){10}}
\put(5,35){\makebox(0,0){$b$}}
\put(10,36.5){\vector(-2,1){10}}
\put(34,73){\makebox(0,0){$a$}}
\put(29,68){\vector(1,0){9}}
\put(34,25){\makebox(0,0){$a$}}
\put(38,30){\vector(-1,0){9}}

\put(38,60){\oval(40,10)} 
\multiput(19,65)(9.2,0){5}{\circle*{3}}
\multiput(19,55)(9.2,0){5}{\circle*{3}}

\put(0,-22){\put(38,60){\oval(40,10)} 
\multiput(19,65)(9.2,0){5}{\circle*{3}}
\multiput(19,55)(9.2,0){5}{\circle*{3}}}

\multiput(-1,0)(0,-44){2}{\multiput(111,82)(0,-22){2}{\oval(77,10)}
\multiput(73,65)(9.25,0){9}{\circle*{3}}
\multiput(73,55)(9.25,0){9}{\circle*{3}}
\multiput(73,77)(9.25,0){9}{\circle*{3}}
\multiput(73,87)(9.25,0){9}{\circle*{3}}
}

\put(105,95){\makebox(0,0){$a$}}
\put(100.5,90){\vector(1,0){9}}
\put(105,3){\makebox(0,0){$a$}}
\put(109,7){\vector(-1,0){9}}


\end{picture}\end{center}\end{figure}
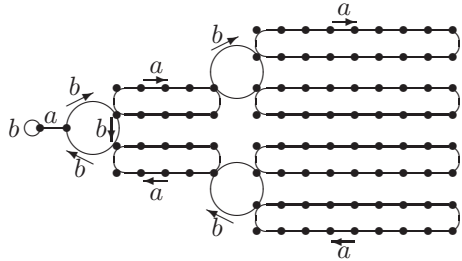

\subsection{The general bubble group}

\begin{figure}[h]
\begin{center}\caption{A piece of the marked Schreier graph of an infinite 
bubble group with $\mathbf a= (a_1,a_2,\dots), \mathbf b=(3,3,3\dots)$. }\label{B2}

\begin{picture}(340,100)(0,0) 

\put(25,50){\oval(170,15)[r]}
\put(120,50){\circle{20}}
\put(110,50){\circle*{3}}
\multiput(127,56.5)(0,-13){2}{\circle*{3}}
\multiput(105,57.5)(-20,0){5}{\circle*{3}}
\multiput(105,42.5)(-20,0){5}{\circle*{3}}
\multiput(105,60.5)(-20,0){5}{\circle{6}}
\multiput(105,39.5)(-20,0){5}{\circle{6}}

\multiput(30,57.5)(-10,0){3}{\line(-1,0){6}}
\multiput(30,42.5)(-10,0){3}{\line(-1,0){6}}

\put(225,64){\oval(200,15)}
\multiput(127,71.5)(19.4,0){11}{\circle*{3}}
\multiput(127,74.5)(19.4,0){10}{\circle{6}}
\multiput(146.4,56.5)(19.4,0){10}{\circle*{3}}
\multiput(146.4,53.5)(19.4,0){10}{\circle{6}}
\put(329.3,79){\oval(20,20)[bl]} 

\put(333,82){\oval(20,20)[bl]}
\put(323,82){\vector(0,1){2}}
\put(328,82){\makebox(0,0){$b$}}

\put(225,36){\oval(200,15)}
\multiput(146,42.5)(19.4,0){10}{\circle*{3}}
\multiput(146,45.5)(19.4,0){10}{\circle{6}}
\multiput(127.4,28)(19.4,0){11}{\circle*{3}}
\multiput(127.4,25)(19.4,0){10}{\circle{6}}
\put(330,21){\oval(18,18)[tl]}

\put(333,18){\oval(18,18)[tl]}
\put(333,27){\vector(1,0){2}}
\put(328,20){\makebox(0,0){$b$}}

\put(95,63){\makebox(0,0){$a$}}
\put(90,59.5){\vector(1,0){10}}

\put(156,22){\makebox(0,0){$a$}}
\put(161,26){\vector(-1,0){10}}

\put(111.5,58.5){\put(3,6.5){\makebox(0,0){$b$}}
\put(0,0){\vector(2,1){10}}}

\put(154,77){\makebox(0,0){$a$}}
\put(152,74){\vector(1,0){9}}
\put(186,82){\makebox(0,0){$b$}}

\end{picture}\end{center}\end{figure}
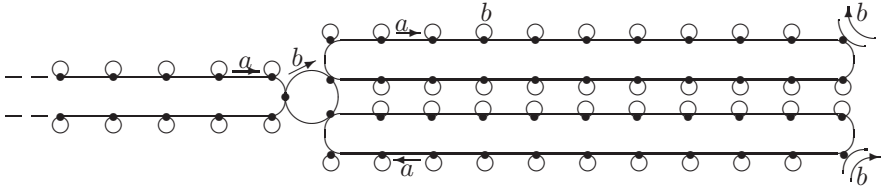

Let $\mathbf{a}=(a_{1},a_{2},...)$ and $\mathbf{b}=\left(b_{1},b_{2},..\right)$ 
be two natural integer sequences (finite  or infinite; if the sequence are 
finite, $\mathbf b $ is one element shorter than $\mathbf a$). 
The ``bubble group'' 
$\Gamma_{\mathbf a,\mathbf b}$ is associated with the tree like bubble graph 
$X_{\mathbf a,\mathbf b}$ 
were $X_{\mathbf a,\mathbf b}$ is obtained from the 
rooted tree $\mathbf T_{\mathbf b}$ 
with forward degree sequence $(1,b_1-1,b_2-1,\dots)$  as follows. 
Each edge at level $k\ge 1$ in the tree (we make the convention that
the level of an edge is the level of the child on that edge)
is replaced by a cycle of length $2a_k$ called a bubble. 
Each vertex at level $k\ge 1$ (we ignore the root which is now part of a circle 
of length $2a_1$) is blown-up 
to a $b_k$-cycle with each vertex of this cycle inheriting one of the 
associated $2a_{k+1}$-cycle. These $b_k$-cycles are called branching cycles. 
Finally, 
at each vertex which belong only to a bubble (but not to a branching cycle), 
we add a self loop.  The vertex set of the graph $X_{\mathbf a,\mathbf b}$ 
can be parametrized  using pairs $(w,u)$ with $w$ a finite word in 
$$\{\emptyset\} \cup (
\cup _{k=1}^\infty \{1,\dots,b_1-1\}\times\{1,\dots,b_2-1\}\times \cdots \times \{1,,\dots,b_k-1\})$$
and $u\in \{0,\dots,2a_{k+1}-1\}$ if $|w|=k$.  By definition, the vertex 
$o=\emptyset$ is the root. 

In the graph  $X_{\mathbf a,\mathbf b}$, we call ``level k'' 
the set of all the vertices $(w,u)$ with $|w|=k-1$, $0\le u\le 2a_{k}-1$.
If all the $a_k$ are distinct, this is the set of all vertices that 
belong to a bubble of length $2a_k$.  We say that a branching cycle
is at ``level k'' if it is attached  at the far end (i.e., 
furthest away from $o$) of a level-$k$ bubble.
Note that the vertices of any branching cycle at level $k$ are 
parametrized as follows:  \begin{itemize}
\item
$(w', a_k)$ with $|w'|=k-1$ for the vertex closest to the root $o$, a vertex 
which also belongs  to a level-$k$ bubble,
\item
$(w'z,0)$ with $z\in \{1,\dots, b_k-1\}$ for the other vertices on 
that branching cycle, each of which also belongs to a level-$(k+1)$ bubble.
\end{itemize}
  We let 
$$\mathfrak b(w')=\{(w',a_k),(w'1,0),\dots,(w'(b_k-1),0)\}$$  
denote the branching cycle at $(w',a_k)$. 

Having chosen an orientation along each cycle (say, clockwise), we label each 
edge of the bubble with the letter $a$ and each edge of the branching 
cycle with the letter $b$.

The group $\Gamma_{\mathbf a,\mathbf b}$ is a subgroup of 
the (full) permutation group of the vertex set of $X_{\mathbf a,\mathbf b}$ 
generated 
by two elements $\alpha$ and $\beta$.  Informally, $\alpha$ rotates the 
bubbles whereas $\beta$ rotates the branching cycles.  Formally, the action of 
the permutation $\alpha$ (resp. $\beta$) on any vertex $x$ in 
$X_{\mathbf a,\mathbf b}$ 
is indicated by the oriented labeled edge at $x$ marked with an $a$ 
(resp. a $b$). Obviously, we can replace the edge labels $a, b$ with the 
group elements $\alpha,\beta$, once these are defined.   

\begin{lemma} \label{lem-bubexp}
For any choice of the sequence $\mathbf a, \mathbf b$
with $b_i\ge 3$ for all $i$, the group
$\Gamma_{\mathbf a,\mathbf b}$ has exponential volume growth.
\end{lemma}
\begin{remark} If $\mathbf a,\mathbf b$ are constant sequences with
$a_i=a$ and $b_i=2$ for all  $i$ then the group $\Gamma_{\mathbf a,\mathbf 2}$
actually falls into the class of Neumann-Segal type groups discussed 
in Section \ref{sec-NS}.  In particular, it is a subgroup of the 
automorphism group of the tree $\mathbb T_{\overline{d}}$ with $\overline{d}=(d_i)_1^\infty$, $d_1=2a$, $d_{i}=2$, $i\ge 2$. This description 
of the group $\Gamma_{\mathbf a ,\mathbf 2}$ shows that it is a subgroup of
the group $D_\infty\wr (\mathbb Z/ 2a\mathbb Z)$. Hence $\Gamma_{\mathbf a ,\mathbf 2}$ has polynomial volume growth of degree at most $2a$.
\end{remark}
\begin{proof} If  $\mathbf a$ is bounded
and all $b_i\ge 3$, it is obvious that 
$\Gamma_{\mathbf a,\mathbf b}$ has exponential volume growth. If 
$\mathbf a$ is unbounded we show that the words
$$ \prod_1^{n} \alpha^{n-i}\beta^{\epsilon_i} \alpha^{n-i}
= \alpha^{n-1}\beta^{\epsilon_1} \alpha^{-1}\beta^{\epsilon_2}\alpha^{-1}\dots 
\beta^{\epsilon_{n-1}} ,\;\;\epsilon_i\in \{\pm 1\}$$
are all distinct.
We proceed by induction on $n$. If $n=k=1$ the property is obviously satisfied.
Assume $k\ge 1$ and the property is true  for all $n\le k$. Pick to elements
$g,g$ of the form above with $n=k+1$ with respective sequence 
$(\epsilon_j)_1^n,(\epsilon'_j)_1^n$.  Since $a_i$ is unbounded, 
we can find $i$ such that $a_i\ge k+2$. 
Let $w=1\dots 1$ of length $|w|=i-1$, $u=k+1$, and consider the point 
$x=(w,u)$ on the Schreier graph $X_{\mathbf a,\mathbf b}$. If 
$\epsilon_n\neq \epsilon'_n$, direct inspection shows that the elements 
 $g,g'$ move $x$
to two different bubbles, one at level $i-1$ and the other at level $i$. 
If $\epsilon_n=\epsilon'_n$, apply the induction hypothesis to 
$$ \prod_2^{n} \alpha^{n-i}\beta^{\epsilon_i} \alpha^{n-i},\;\;
\prod_2^{n} \alpha^{n-i}\beta^{\epsilon'_i} \alpha^{n-i}.$$
This concludes the proof.
\end{proof}

We also consider the associated 
finite bubble groups $\Gamma^{k}_{\mathbf a,\mathbf b}$
defined by action on truncated bubble
graphs. Let $X_{\mathbf{a},\mathbf{b}}^{k}$ denote the first $k$
levels of the bubble graph $X_{\mathbf{a,\mathbf{b}}}$. The finite bubble group 
$\Gamma_{\mathbf a,\mathbf b}^k$ is generated by $\alpha_{k}$ and $\beta_{k}$,
where $\alpha_{k}$ acts in the same way as $\alpha$ by rotating
the long bubbles, and $\beta_{k}$ acts in the same way as $\beta$
except that, at level $k$, $\beta_{k}$ stabilizes
the end points instead of moving them along the branching cycle. 
See Figure \ref{B1}.

\subsection{Local embedding under Assumption (A)}

\begin{definition}[Assumption (A)] \label{sequenceAssumption}
We say that assumption (A) is satisfied if 
the sequence $\mathbf b=(b_1,b_2,\dots)$ is constant 
(i.e., $b_i=b$ for all $i=1,\dots)$ and 
the scaling sequence $\mathbf{a}=(a_{1},a_{2},...)$
is monotone increasing to $\infty$, that is 
\[
1\le a_{1}\le a_{2}\le...,\ \lim_{n\to\infty}a_{n}=\infty.
\]
\end{definition}

Without the assumption that the sequence $a_n$ increases to infinity, 
the group $\Gamma_{\mathbf a,\mathbf b}$ may be non-amenable. For 
instance, consider the case when $\mathbf a$ is the constant sequence $a_i=2$
and $b=3$. Then $\Gamma_{\mathbf a,\mathbf b}= \mathbb Z_2*\mathbb Z_3$. 
More generally, 
$\Gamma_{\mathbf a,\mathbf b}$ is non-amenable whenever both sequences 
$\mathbf a,\mathbf b$ are bounded. For an arbitrary sequence $\mathbf b$, 
the condition $\liminf a_i=\infty$  suffices to imply that 
$\Gamma_{\mathbf a,\mathbf b}$ is amenable. See Section \ref{sec-am}.

The method of this section still 
applies if we modify Assumption(A) by replacing the hypothesis that 
$\mathbf b$ is constant by the requirement that it is bounded (hence, 
take only finitely many values). The length of the intervals 
between different occurrences of a complete collection of the values taken by 
$\mathbf b$ play a key role in the form of the estimates that can be 
produced by this method.  In the case of periodic (or quasi-periodic) 
$\mathbf b$, the estimates are the same as in the case of a constant sequence 
$\mathbf b$ but if a certain value, say $5$, appears only along a very sparse 
sequence of indices then the estimates deteriorate. See Section \ref{sec-am}.

In this section, we work under Assumption (A). We drop the explicit references 
to the sequences $\mathbf{a},\mathbf {b}$ and write
$$\Gamma_{\mathbf a,\mathbf b}=\Gamma,\;\;X_{\mathbf a,\mathbf b}=X.$$
Set
$$\mathfrak N (w,r)=\{x\in X: d(x,\mathfrak b(w)) \le r\},\;w\in 
\{0,\dots,b-1\}^{(\infty)},\;r>0 .$$
For any $k\le j$, $w$ of length $|w|=j$ 
and $0\le r \le a_{k-1}-1$, we have an obvious bijective map
$$\iota^w_k: \mathfrak N (w,r)\mapsto \mathfrak N(1^{k-1},r)$$ 
which can be used to identify these vertex sets.

In order to fit the bubble group $\Gamma$
into the general framework introduced in Section \ref{sec-UB}, 
we need to explain how to choose the sets $J_k, B_k$ and 
the approximation groups $\mathbf \Gamma^k$ so that  
the associated local embeddings $\vartheta_k$ required for Hypothesis 
($\Omega$)  
can be proved to exists. See Definition \ref{def-Omega}. 

Let $\mathbf S$ be the alphabet $\mathbf S=\{\alpha^{\pm1 },\beta^{\pm 1}\}$. 
Given subsets $J,B$ of $X$ with $J\subset B$,
recall that the set $\Omega\left(J,B\right)\subset\mathbf{S}^{(\infty)}$
is defined as 
\[
\Omega\left(J,B\right)=\left\{ 
\omega \in\mathbf{S}^{(\infty)}:\ 
\mathcal{O}\left(\omega,J\right)\subset B\right\} .
\]

For a given level $k$, we set 
$$\mathfrak{m}_k=(1^{k-1},a_{k}/2),\;\;J_k=\{\mathfrak{m}_k\}$$ 
and 
$$B_k(l)=\left\{ x\in X:\ d\left(x,\mathfrak{m}_k\right)\le 
l\right\}, \;0\le l\le (a_{k}/2)-1.$$

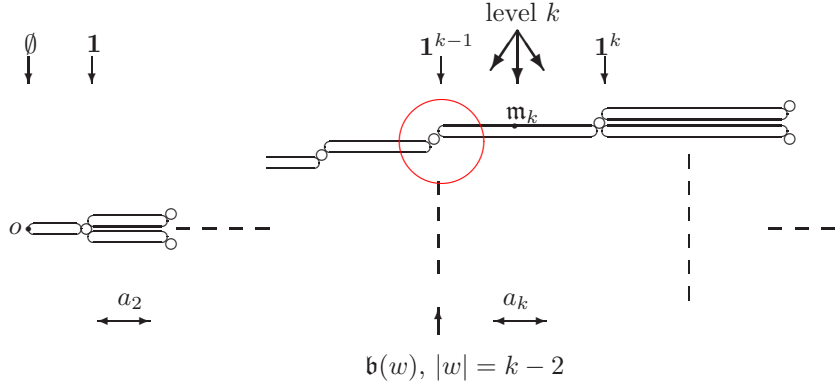
\begin{figure}[h]
\begin{center}\caption{Sketch 
of the Schreier graph $X$: levels, $\mathfrak b(w)$, $\mathfrak m_k$.
Details of the red circle region are shown in Figure \ref{BW2}.}\label{BW0}

\begin{picture}(320,150)(0,0) 

\put(10,0){\put(-5,50){\makebox(0,0){$o$}}
\put(0,50){\circle*{2}}
\put(10,50){\oval(20,4.2)}
\put(22,50){\circle{4}}

\put(37.5,53){\oval(30,4.2)}
\put(37.5,47){\oval(30,4.2)}
\put(54,55.5){\circle{4}}
\put(54,44.2){\circle{4}}

\multiput(56,50)(10,0){4}{\line(1,0){5}} 

\put(24,115){\vector(0,-1){10} 
\makebox(-5,10){$\mathbf 1$}}

\put(0,115){\vector(0,-1){10} 
\makebox(-5,10){$\emptyset$}}

\put(30,30){\makebox(17,-15){$a_2$} 
\put(-11,-15){\vector(1,0){10}} \put(-11,-15){\vector(-1,0){10}} 
}
}

\put(100,75){\put(0,0){\oval(40,4.2)[r]}
\put(21,3){\circle{4}}
}

\put(142,81){\put(0,0){\oval(40,4.2)}
\put(21.5,3.05){\circle{4}}
}

\multiput(165,68)(0,-10){4}{\line(0,-1){5}} 
\put(160,-5){\makebox(30,5){$\mathfrak b(w)$, $|w|=k-2$}} 
\put(165,10){\vector(0,1){10} 
}

\put(195,87){\put(0,0){\oval(60,4.2)}
\put(31,3){\circle{4}}}

\put(262,92.5){\put(0,1){\oval(70,4.2)}
\put(36,4){\circle{4}}
\put(0,-5.5){\oval(70,4.2)}
\put(36,-8.5){\circle{4}}

}

\multiput(260,78)(0,-10){6}{\line(0,-1){5}} 

\multiput(290,50)(10,0){3}{\line(1,0){5}}

\put(166,115){\vector(0,-1){10} 
\makebox(-2,10){$\mathbf 1^{k-1}$}}

\put(173,30){\makebox(43,-15){$a_k$} 
\put(-20,-15){\vector(1,0){10}} \put(-20,-15){\vector(-1,0){10}} 
}
\put(228,115){\vector(0,-1){10} 
\makebox(-4,10){$\mathbf 1^{k}$}}

\put(195,91){\makebox(4,4){$\mathfrak m_k$}}
\put(194,89){\circle*{2}}

\put(195,130){\makebox(4,4){level $k$}}
{\thicklines \put(195,125){\vector(2,-3){10}}
\put(195,125){\vector(0,-1){20}}
\put(195,125){\vector(-2,-3){10}}
}

{\color{red}  \put(163,83){\circle{30}}}

\end{picture}\end{center}\end{figure}

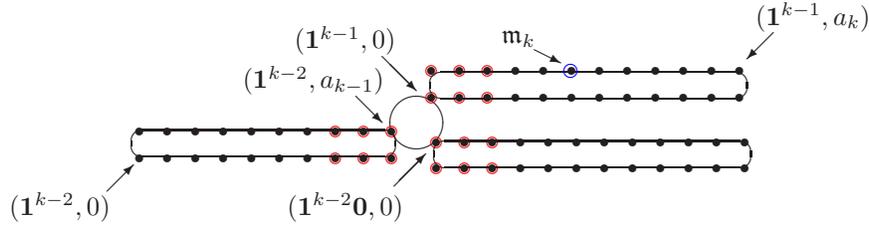
\begin{figure}[h]
\begin{center}\caption{Sketch 
showing 
$\color{red}\mathfrak N(\mathbf 1^{k-2},r)$,
$\mathfrak m_k$}\label{BW2}

\begin{picture}(320,150)(0,0)

\put(150,70){\circle{20}}
\multiput(140.5,66.5)(-10.6,0){10}{\circle*{3}}
\multiput(140.5,56.5)(-10.6,0){10}{\circle*{3}}
{\color{red}\multiput(140.5,66.5)(-10.6,0){3}{\circle{3.5}}
\multiput(140.5,56.5)(-10.6,0){3}{\circle{3.5}}}

\multiput(155.5,79.5)(10.6,0){12}{\circle*{3}}
\multiput(155.5,89.5)(10.6,0){12}{\circle*{3}}
{\color{red}\multiput(155.5,79.5)(10.6,0){3}{\circle{3.5}}
\multiput(155.5,89.5)(10.6,0){3}{\circle{3.5}}}

\multiput(157.5,62.5)(10.6,0){12}{\circle*{3}}
\multiput(157.5,52.5)(10.6,0){12}{\circle*{3}}
{\color{red} \multiput(157.5,62.5)(10.6,0){3}{\circle{3.5}}
\multiput(157.5,52.5)(10.6,0){3}{\circle{3.5}}}

\put(92,62){\oval(100,10)}
\put(216.5,58){\oval(120,10)}
\put(215,84){\oval(120,10)}

\put(96,81){\makebox(30,10){$(\mathbf 1^{k-2},a_{k-1})$}}
\put(127,80){\vector(1,-1){10}}

\put(108,96){\makebox(30,10){$(\mathbf 1^{k-1},0)$}}
\put(137,95){\vector(1,-1){13}}

\put(108,33){\makebox(30,10){$(\mathbf 1^{k-2}\mathbf 0,0)$}}
\put(141.5,46){\vector(1,1){12}}

\put(0,33){\makebox(30,10){$(\mathbf 1^{k-2},0)$}}
\put(33,44){\vector(1,1){10}}

\put(120.6,23.5){
\put(53,73){\makebox(30,10){$\mathfrak m_k$}}
\put(74,74){\vector(2,-1){10}}
{\color{blue}\put(87.6,66){\circle{5}}}}

\put(285,105){\makebox(30,10){$(\mathbf 1^{k-1},a_k)$}}
\put(285,104){\vector(-1,-1){10}}

\end{picture}\end{center}\end{figure}

\begin{lemma} \label{lem-com}
Fix $0\le t\le (a_{k}/2)-1$. Let $\omega=\gamma_1\dots\gamma_q$ 
be a word of length $q$ such that $\gamma_j\dots \gamma_q \cdot 
\mathfrak m_k\in B_k(t)$ for each $1\le j\le q$. Then there exists $s_\omega$
such that 
$$\gamma_1\dots \gamma_q\cdot \mathfrak m_k=\alpha^{s_\omega} \cdot \mathfrak m _k
\mbox{ and } \gamma_1^{-1} \dots \gamma_q^{-1} \cdot \mathfrak m_k= 
\alpha^{-s_\omega}\cdot \mathfrak m_k.$$
Further, let $s$ be an integer  such that $|s|+t\le (a_{k}/2)-1$
and set $x = \alpha^s \mathfrak m_k$.  Then
$$ \alpha ^s \omega \cdot \mathfrak m_k= \omega \alpha^s \cdot \mathfrak m_k=
\omega \cdot x.$$
In particular, we have
$$d(\omega\cdot \mathfrak m_k, \omega\cdot x)=d(\mathfrak m_k,x).$$
\end{lemma}
\begin{proof} Writing $\mathfrak m=\mathfrak m_k$, 
the first assertion is easily proved by considering
$\gamma_j\dots \gamma_q\cdot \mathfrak m$ and by descending induction 
on $j\le q$. 

The second assertion is also proved by descending induction on $j\le q$. 
The statement is obvious for 
$j=q$.  Assume the claim is true for $j+1$.  By hypothesis,
$d(\gamma_{j+1}\dots \gamma_q \cdot \mathfrak m 
,\mathfrak m )\le t$. 
Since $|s|+t\le (a_{k}/2)-1$, it follows that
the action of $\alpha^s \gamma_{j}$ on $\gamma_{j+1}\dots \gamma_q \cdot \mathfrak m $
 is the same 
as the action of $\gamma_{j}\alpha^s$.  Hence 
$$\alpha^s \cdot (\gamma_{j}\dots\gamma_q \cdot \mathfrak m)
= \gamma_j  \alpha^s \gamma_{j+1}\dots \gamma_q\cdot \mathfrak m
= \gamma_j  \dots\gamma_q \alpha^s 
 \cdot \mathfrak m .$$
\end{proof}

\begin{lemma} \label{lem-com2}
Assume that $\omega=\gamma_1\dots\gamma_p \in \Omega_k(l)= 
\Omega\left(\mathfrak{m}_k,B_k(l)\right)$ with $l\leq (a_{k}/4)-1$.
Then, for any $1\le j\le p$, we have $\gamma_1\dots \gamma_j=\alpha^{s_j}
\mathfrak m_k$ with $|s_j|\le l$ and
$$\gamma_1\dots \gamma_{j-1}\gamma_j\cdot \mathfrak m_k
=\gamma_j\gamma_{j-1}\dots\gamma_1\cdot \mathfrak m_k.$$
Further, for every subword $u=\gamma_{i}...\gamma_{j}$
of $w$,  $$u\cdot \mathfrak m_k= \alpha^{s_j-s_{i-1}}\cdot \mathfrak m_k 
\mbox{ and }
 u^{-1}\cdot \mathfrak m_k= \alpha^{-s_j+s_{i-1}}\cdot \mathfrak m_k.$$
\end{lemma}
\begin{proof} We set $\mathfrak m=\mathfrak m_k$.
The first statement is proved by induction on the
length $p$ of the word $\omega$. 
The desired property is obviously true for $p=1$.  
Assume that it has been proved for all words of length at most $p$. 

Consider $ \omega=\gamma_1\dots \gamma_{p+1}\in \Omega_k(l)$ and write
$$\gamma_1\dots \gamma_{p+1}\cdot \mathfrak m
= (\gamma_1\dots \gamma_p) \cdot( \gamma_{p+1}\cdot \mathfrak m).$$
By hypothesis, $\gamma_1\cdot \mathfrak m, \gamma_1\gamma_2 \cdot \mathfrak m, \dots, \gamma_1\dots \gamma_p\cdot \mathfrak m$ are in $B(l)$. 
By the induction hypothesis,
$\gamma_{i}\dots \gamma_p \cdot \mathfrak m = 
\alpha^{s_p-s_{i-1}}\cdot\mathfrak m$, $1\le i\le p$.
Hence Lemma \ref{lem-com} and the induction hypothesis give
\begin{eqnarray*}
\gamma_1  \dots\gamma_p  \gamma_{p+1} 
 \cdot \mathfrak m  &=& \gamma_{p+1}\cdot (\gamma_1\dots\gamma_p \cdot \mathfrak m)\\
&=& \gamma_{p+1}\cdot (\gamma_p\dots\gamma_1 \cdot \mathfrak m)\\
&=& \gamma_{p+1} \gamma_p\dots\gamma_1 \cdot \mathfrak m.
\end{eqnarray*}

We need to show that, for any subword 
$\omega'= \gamma_j\dots\gamma_{p+1}$,
$$\omega' \cdot \mathfrak m= \alpha^{s_{p+1}-s_{j-1}}\cdot \mathfrak m
\mbox{ and }
(\omega')^{-1} \cdot \mathfrak m= \alpha^{-s_{p+1}+s_{j-1}}\cdot \mathfrak m
$$
We write
\begin{eqnarray*} \omega' \cdot \mathfrak m &=&
(\gamma_1\dots \gamma_{j-1} )^{-1}\gamma_1 \dots \gamma_{p+1} \cdot \mathfrak m\\
&=& (\gamma_1\dots \gamma_{j-1} )^{-1}\alpha^{s_{p+1}} \cdot \mathfrak m
\end{eqnarray*}
and
\begin{eqnarray*} (\omega')^{-1} \cdot \mathfrak m &=&
\gamma^{-1}_{p+1}\dots \gamma_1^{-1} (\gamma_{1}\dots \gamma_{j-1}) 
\cdot \mathfrak m\\
&=& 
\gamma^{-1}_{p+1}\dots \gamma_1^{-1} \alpha ^{s_{j-1}} 
\cdot \mathfrak m.
\end{eqnarray*}
We have proved that  $\gamma_{i} \dots \gamma_1 \cdot \mathfrak m \in B(l)$
for each $1\le i\le p+1$.  Lemma \ref{lem-com} shows that 
$\gamma_{i}^{-1} \dots \gamma_1^{-1} \cdot \mathfrak m \in B(l)$. 
By the second part of Lemma \ref{lem-com}, we have
\begin{eqnarray*} 
\omega' \cdot \mathfrak m &=&
(\gamma_1\dots \gamma_{j-1} )^{-1}\alpha^{s_{p+1}} \cdot \mathfrak m
= \alpha ^{s_{p+1}} \alpha^{-s_{j-1}} \cdot \mathfrak m.
\end{eqnarray*}
and
\begin{eqnarray*} 
(\omega')^{-1} \cdot \mathfrak m &=&
(\gamma_{p+1}^{-1}\dots \gamma_{1} ^{-1}) \alpha^{s_{j-1}} \cdot \mathfrak m
= 
\alpha^{s_{j-1}} \alpha ^{-s_{p+1}}  \cdot \mathfrak m.
\end{eqnarray*}

\end{proof}

The following observation regarding repetition of orbits is straightforward.
Given a word $w=w_{1}...w_{p}$, consider the orbit $\left\{ w_{s}^{-1}...w_{1}^{-1} \cdot x\right\} _{1\le s\le p}$
of $x$. Note that this can be thought of as a forward orbit, it behaves
very differently from the inverted orbit $\mathcal{O}\left(w;x\right)$. 

\begin{lemma}[Identification of orbits]\label{identification}
Fix $k\ge 0$ and $0\le l\le (a_{k}/4)-1$.
For any $t\ge k+1$ and any vertex $x=(w,u)$ with $|w|=t-1$ 
and $0\le u\le 2a_t-1$
(i.e., any vertex at level $t$), let $w'$ be the parent of $w$ and set
\[
\hat{x}=\left\{ \begin{array}{ccl}
\mathfrak{m}_k & \mbox{if } & d\left(x, \{(w,0),(w,a_t))\right)>l,\\
\iota_{k}^{w'}(x) & \mbox{if } & d\left(x, (w,0)\right)\le l,\\
 \iota_{k}^{w}(x) & \mbox{if } & d\left(x, (w,a_t)\right)\le l.\\ 
\end{array}\right.
\]
In the first case, define $\hat{\iota}=\hat{\iota}_x$ 
to  be the obvious map taking
the segment $\{(w,v): |v-u|\le l\}$ to the segment 
$\{(1^{k-1}, v)): |v-(a_{k}/2)| \le l\}$. In the second and third cases, 
set $\hat{\iota}= \iota^{w'}_{k}$ and $\hat{\iota}= \iota^{w}_{k}$, 
respectively. Let $\omega=\gamma_1\dots\gamma_p$ be a word that belongs to $\Omega_k(l)$.
Then, for any $1\le j\le p$,
$$ \hat{\iota}_x(\gamma^{-1}_j\dots \gamma_1^{-1}\cdot x)=\gamma^{-1}_j\dots \gamma_1^{-1}\cdot \hat{x}.$$
\end{lemma}

\begin{proof} As before, we write $\mathfrak m=\mathfrak m_k$.
For a point $x$ at distance $>l$ from any branching cycle, we show
that the orbit $\left\{ w_{s}^{-1}\dots w_{1}^{-1}\cdot x\right\} _{1\le s\le p}$
up to time $p$ is contained in $B(x,l)$. If not, let $s$ be
the first time that $d\left(w_{s}^{-1}\dots w_{1}^{-1}\cdot x,x\right)=l+1$.
Since, up to time $s$, the orbit must remain within in the bubble containing 
$x$, it can be identified (using $\hat{i}$) with 
$\left\{ w_{j}^{-1}\dots w_{1}^{-1}\cdot \mathfrak{m}\right\} _{1\le j\le s}.$
It follows that 
$d\left(w_{s}^{-1}\dots w_{1}^{-1}\cdot \mathfrak{m},\mathfrak{m}\right)=l+1$.
By lemma \ref{lem-com2}, we have
 $d\left(\mathfrak{m},w_{1}\dots w_{s}\cdot \mathfrak{m}\right)=l+1$. This
contradicts the assumption that 
$w\in\Omega\left(\mathfrak{m},B(\mathfrak{m},l)\right)$.
Therefore the whole orbit up to time $p$ is in $B(x,l)$.

Next, let $x=(w,u)$ be within distance $l$ from one of the branching
cycle, say $d(x,(w,0))\le l$ (the other case is treated in the same way).
We claim that the orbit
$\left\{ w_{s}^{-1}\dots w_{1}^{-1}\cdot x\right\} _{1\le s\le p}$ cannot
leave the set  
$$\mathfrak N(w',3l+1)
= \left\{ z:\ d\left(z,\mathfrak{b}(w')\right)
\le 3l+1\right\} .$$
Suppose on the contrary that the orbit exits $\mathfrak N(w',3l+1)$. 
Let $t\le p$ be the first time such that 
$d\left(w_{t}^{-1}\dots w_{1}^{-1}\cdot x,\mathfrak{b}(w')\right)=3l+2$.
Let $y=w_{t}^{-1}\dots w_{1}^{-1}\cdot x$, since the orbits starts at $x$ 
and $t$ is the first time $d(w_{t}^{-1}\dots w_{1}^{-1}\cdot x,
\mathfrak{b}(w'))=3l+2$, there must be a largest time $s\le t$ such that
$d(w_{s}^{-1}\dots w_{1}^{-1}\cdot x,\mathfrak{b}(w'))=l+1$ and $
d(w_{s}^{-1}\dots w_{1}^{-1}\cdot x,y)=2l+1$. Set 
$z= \gamma_s^{-1}\dots\gamma_1^{-1}\cdot x$. 
By definition of $s$ and $t$, $z$ orbit 
$$\{ \gamma_{s+i}^{-1}\dots \gamma_{s+1}^{-1}\cdot z : 0\le i\le t-s\}$$
remains in the segment between $y$ and $z$ and we must have
\[
2l+1=d\left(z,\gamma_{t}^{-1}\cdots \gamma_{s+1}^{-1} \cdot z\right)=
d\left(\mathfrak{m},\gamma_{t}^{-1}\cdots \gamma_{s+1}^{-1}\cdot \mathfrak{m}\right),
\]
By Lemma \ref{lem-com2}, this contradicts the assumption that 
$\omega\in\Omega_k(l)$.
\end{proof}

Recall that $\mathbf S=\{\alpha^{\pm1 },\beta^{\pm 1}\}$ and that we
have the evaluation map
\[
\theta: \mathbf S^{(\infty)}\to\Gamma_{\mathbf a,\mathbf b},\;\;
\theta_{k}:\mathbf S^{(\infty)}\to\Gamma^{k}_{\mathbf a,\mathbf b}.
\]

\begin{lemma}[Local embeddings]\label{localembddings}
Fix $k\ge 0$.
Let $l$ be an integer such that $0\le l<(a_{k}/4)-1$.
The restriction of the map $\vartheta_k=\theta_{k+1}:\mathbf{S}^{(\infty)}
\to \mathbf G^k= \Gamma^{k+1}_{\mathbf a,\mathbf b}$
to the set $\Omega_k(l)$ satisfies
the conditions in part 2 of Hypothesis $\left(\Omega\right)$.
\end{lemma}

\begin{proof}
Let $\theta:\mathbf S^{(\infty)}$ be the evaluation map in $\Gamma$.
By the definition of $\vartheta_{k}$, it remains to show
that for any $w,w'\in\Omega_k(l)$, $\theta(w)=\theta(w)$ if and only if
$\vartheta_{k}(w)=\vartheta_{k}(w')$.
In other words, setting $\ker (\vartheta_k)=\{\omega\in \mathbf S^{(\infty)}: 
\vartheta_k(\omega)=e\}$,
we need to show
\[
\ker\left(\vartheta_{k}\right)\cap\Omega_k(l)=
\ker(\theta)\cap\Omega_k(l).
\]

We first show that 
\[
\ker\left(\vartheta_{k}\right)\cap\Omega_k(l)
\subset\ker(\theta)\cap\Omega_k(l).
\]
Suppose there exists $\omega\in\ker\left(\vartheta_{k}\right)\cap\Omega_k(l)$
such that $\omega\neq e_{\Gamma}$ in $\Gamma$. Then there exists some point
$x\in X$ such that $\omega\cdot x\neq x$.  If $x=(w,u)$
with $|w|\le k$, we set $\hat{x}=x$. Otherwise, $\hat{x}$ is given by 
Lemma \ref{identification}.

From Lemma
\ref{identification}, we know that $\hat{x}$ is also moved by 
$\omega$.
By Lemma \ref{identification}, the constraint
$\omega\in\Omega_k(l)$ implies that
the orbit $\left\{ \gamma_{j}^{-1}\dots\gamma_{1}^{-1}\cdot \widehat{x}\right\}
_{1\le j\le p}$
never touches any point of the form $(w,a_{k+1})$ with $|w|=k$ (that is, any of 
the end points of level $k+1$). 
Therefore the orbit of $\hat{x}$ in $X$
is exactly the same as the orbit in the finite bubble graph 
$X^{k+1}$.
In particular, they will end at the same place, so that
$\vartheta_{k}(\omega)\cdot \hat{x}\neq\hat{x}$. This obviously contradicts 
the assumption that $\vartheta_k(\omega)$ is trivial.

In the other direction, we assume that there exists $\omega\in
\ker\left(\theta\right)\cap \Omega_k(l)$
such that $\vartheta_k (\omega)$ is non-trivial. In such case, there exists
a point $x=(w,u)\in X^{k+1}$, $|w|\le k$, 
such that $\vartheta_k(\omega)\cdot x\neq x$. If $|w|\le k-1$ or $|w|=k$ and
$d(x, (w,a_{k+1}))> l$ then (following the same line of reasoning as before) 
the orbit 
$$\{ \theta_k(\gamma_j^{-1}\dots \gamma_1^{-1})\cdot x\}_{1\le j\le p}$$ 
in $X^{k+1}$ can readily be 
identified with the orbit of $(w,u)$ in $X$. This gives a contradiction.
If $|w|=k$ and $d(x, (w,a_k))\le l$, then 
the condition $\omega\in \Omega_k(l)$ implies that the orbit
$\{ \pi(\gamma_j^{-1}\dots \gamma_1^{-1})\cdot (w,u)\}_{1\le j\le p}$ in $X$
stays in $\mathfrak N(w,3l+1)$. Since
$ \pi(\gamma_p^{-1}\dots \gamma_1^{-1})\cdot (w,u) =(w,u)$ in $X$,
it now follows by inspection that
$ \vartheta_k(\gamma_p^{-1}\dots \gamma_1^{-1})\cdot x =x$ in $X^{k+1}$. 
This is the desired contradiction.
\end{proof}

\subsection{Isoperimetric profiles of bubble groups}

In this section, we prove and illustrate the following Theorem.

\begin{theorem} \label{bubblebound}
Let $\Gamma=\Gamma_{\mathbf{a},\mathbf b}$ be a bubble group with scaling
sequence $\mathbf{a}$ and branching sequence $\mathbf{b}$ that satisfy
Assumption {\em (A)}. Let $G=\mathbb Z\wr_X \Gamma$ be the associated
permutation wreath product.  There exist a constant $C$ such 
that for all $r>1$,
\[
\Lambda_{1,G}(v)\ge \frac{1}{Cr} \mbox{ for all } v \le C^{-1} r^{
|B_X(o,r)|/C}.
\]
Furthermore, we also have
\[
\Lambda_{2,G}\left(v\right)\le\frac{C}{r^{2}} \mbox{ for all } v \ge 
C\left((|X^{k(r)-1}_{\mathbf a,\mathbf b}|
+ (b-1)^{k(r)}r/2)!\right)
\]
where $k(r)= \min\{k : a_{k}>2r\}$. 
\end{theorem}
\begin{remark} In the above statement, $\Lambda_{i,G}$ must be interpreted 
as a given representative of the equivalence class $\Lambda_{i,G}$, for instance,
$\Lambda_{i,G,\mathbf u}$ where $\mathbf u$ is the uniform measure on a fixed 
finite symmetric set of generators of the group $G$, and the constant 
$C$ involved in these inequalities  depends on the representative. We 
will use this convention throughout.
\end{remark}
\begin{proof}
The lower bound follows directly from the volume-diameter bound. See Example 
\ref{exa-Dpro-low}. Recall that $\mathfrak{m}_k=(1^{k-1},a_{k}/2)$ 

For the upper bound, we apply Proposition \ref{function}. Indeed, 
Lemma \ref{localembddings} shows that assumption 
$(\Omega)$ is satisfied with the choice 
$$J_{r}=\{o,\mathfrak m_{k(r)}\},
B_r=W(k(r),r/4)
\mbox{ and }\mathbf \Gamma^r=\Gamma^{k(r)+1}_{\mathbf a,\mathbf b}$$ where
\begin{eqnarray*}
W(k,t)&=&\{(w,u): |w|\le k-2, u\in \{1\dots, 2a_{|w|+1}-1\} \} \\
&&\bigcup \left(\bigcup_{|w|=k-2} \mathfrak N( w,t)\right)\;  \bigcup B(\mathfrak m_{k},t),
\end{eqnarray*}
and $\Gamma^{k+1}_{\mathbf a,\mathbf b}$ is
the finite bubble group
acting on the first $k+1$ levels of the bubble graph. Note that it 
follows from the various definitions that
$$\Omega(J_r,B_r)\subset \Omega(\mathfrak m_{k(r)}, B(\mathfrak m_{k(r)}, r/4))
.$$  Note that we abuse notation a little here by indexing $J_r,B_r,\Gamma^r$ 
by $r$ instead of an integer as in Definition \ref{def-Omega}. Also, 
in Lemma \ref{localembddings}, the group $\mathbf \Gamma^r$ 
is denoted by $\mathbf \Gamma^{k(r)}$. 

\begin{figure}[h]
\begin{center}\caption{The set $W(k,t)$ on a sketch 
of the Schreier graph $X$ }\label{BW}

\begin{picture}(320,150)(0,0) 
\put(10,0){
\put(-5,50){\makebox(0,0){$o$}}
{\color{red}
\put(0,50){\circle*{2}}
\put(10,50){\oval(20,4.2)}
\put(22,50){\circle{4}}

\put(37.5,53){\oval(30,4.2)}
\put(37.5,47){\oval(30,4.2)}
\put(54,55.5){\circle{4}}
\put(54,44){\circle{4}}

\multiput(56,50)(10,0){4}{\line(1,0){5}} }}

{\color{red}
\put(100,75){\put(0,0){\oval(40,4.2)[r]}
\put(21,3){\circle{4}}
}

\put(142,81){\put(0,0){\oval(40,4.2)}
\put(21,3){\circle{4}}
}

{\thicklines\put(170,87){\oval(10,4.2)[l]}}
\multiput(165,68)(0,-10){4}{\line(0,-1){5}} 
\put(160,-5){\makebox(30,5){$\mathfrak b(w)$, $|w|=k-2$}} 
\put(165,10){\vector(0,1){10} 
}}

\put(195,87){\put(0,0){\oval(60,4.2)}
\put(31,3){\circle{4}}}

\put(262,92.5){\put(0,1){\oval(70,4.2)}
\put(36,4){\circle{4}}
\put(0,-6){\oval(70,4.2)}
\put(36,-9){\circle{4}}

}

\multiput(260,78)(0,-10){6}{\line(0,-1){5}} 

\multiput(290,50)(10,0){3}{\line(1,0){5}} 

\put(10,0){\put(24,115){\vector(0,-1){10} 
\makebox(-5,10){$\mathbf 1$}}

\put(30,30){\makebox(17,-15){$a_2$} 
\put(-11,-15){\vector(1,0){10}} \put(-11,-15){\vector(-1,0){10}} 
}}

\put(166,115){\vector(0,-1){10} 
\makebox(-2,10){$\mathbf 1^{k-1}$}}

\put(173,30){\makebox(43,-15){$a_{k}$} 
\put(-20,-15){\vector(1,0){10}} \put(-20,-15){\vector(-1,0){10}} 
}
\put(228,115){\vector(0,-1){10} 
\makebox(-4,10){$\mathbf 1^{k}$}}

\put(195,94){\makebox(4,4){$\mathfrak m_k$}}
\put(195,89){\circle*{2}}
{\color{red}\thicklines \put(190,89){\line(1,0){10}}}

\end{picture}\end{center}\end{figure}
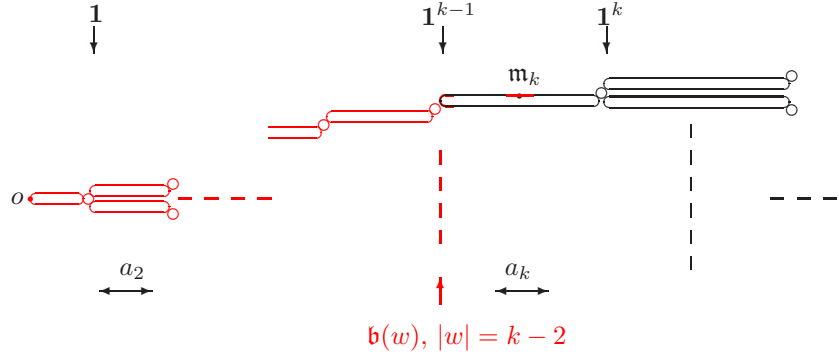

To apply Proposition \ref{function} we need a $(J_r,B_r)$-admissible 
function $F$ charging translates of a fix set $A_r$ together with a test function $\psi_r$ on $\mathbf \Gamma^r$.  We pick 
$$A_r=W(k(r),r/8)$$
and observe that if $ |t|\le r/8$ then
$$ \beta \alpha^t\cdot A_r= \alpha^t\cdot A_r.$$
Set 
\begin{equation}\label{BFr}
F_r(Y)=\left\{\begin{array}{cl} 1- 8|t|/r &\mbox{ if } Y=\alpha^t\cdot A_r\\
0& \mbox{ otherwise}.
\end{array}\right.
\end{equation}
This is a $(J_r,B_r)$-admissible function and, obviously,
$$\mathcal Q_{\mathcal P_f(X),\mu}(F_r) \le C/r^2.$$
For the function $\psi_r$ on $\mathbf \Gamma^{r}$, 
we simply take $\psi_r\equiv 1$.

To finish the proof, we need to estimate the size of the set
$$\overline{\Omega}(J_r,B_r) \subset \overline{\Omega}(\mathfrak m_{k(r)},
B(\mathfrak m_{k(r)},r/4))
.$$ By Lemma \ref{identification}, 
any element $\gamma\in 
\overline{\Omega}(\mathfrak m_{k(r)},
B(\mathfrak m_{k(r)},r/4)$ is determined by its action on $A_r$ and 
the image of $A_r$ is contained in $W(k(r),r/2)$. Consequently,
$$|\overline{\Omega}(\mathfrak m_{k(r)},
B(\mathfrak m_{k(r)},r/4)|\le |W(k(r),r/2)|!\le 
\left(|X^{k(r)-1}_{\mathbf a,\mathbf b}|+ (b-1)^{k(r)}r/2\right)!.$$
\end{proof}

\subsection{Bubble group examples}
We now discuss these results for a variety of examples of sequences 
$\mathbf a=(a_n)_1^\infty$ when $b_i=3$ for all $i$. The quality of the 
results depends on whether  $s_k=\sum_1^k a_i \asymp  a_k$ or not. 
The partial sum $s_k$ which is equal, essentially, to the distance between 
the root $o$ and the branching cycles $\mathfrak b(w)$ with $|w|=k-1$, 
that is, the branching cycles  at level $k$.  If 
$a_k$ is much smaller than $s_k$ then the upper and lower bounds for 
$\Lambda_{1,G}$ and $\Lambda_{2,G}$ in Theorem \ref{bubblebound}
do not match. 
See the first example below. If, $a_k\asymp s_k$ then the bound matches and we 
obtain good results.

\begin{example}Assume that $a_k=k$. Then $s_k= k(k+1)/2$.  If 
$r \in [s_k, s_{k+1}]$, $|B_X(o,r)|\asymp (1+k)2^k$. 
If $\log v\asymp |B_X(o,r)|\asymp (1+k)2^k$, we obtain 
$$C(\log\log v)^{-1}\ge \Lambda_{1,G}(v)\ge (C\log\log v)^{-2}$$ and 
$$(C\log\log v)^{-4}\le \Lambda_{2,G}(v)\le C(\log\log v)^{-2}.$$   
\end{example}

\begin{example}\label{exa-bubb1}
Fix $\beta>0$, take $a_{k}=\left\lfloor 2^{\beta k}\right\rfloor $. Then
$s_k\asymp a_k\asymp 2^{\beta k}$ and, for
$r\asymp 2^{\beta k}$, we have $|B_{X}(o,t)|\asymp 2^{(\beta+1)k}.$ It follows 
that Theorem \ref{bubblebound} gives
$$ \Lambda_{1,G}^2(v)\asymp \Lambda_{2,G}(v)\asymp \left(\frac{\log\log v}
{\log v}\right)^{2\beta/(\beta+1)}.$$
From this estimate on $\Lambda_{2,G}$, we deduce also that
\[
\Phi_G(n)\simeq\exp\left(-n^{\frac{\beta+1}{3\beta+1}}(\log n)^{\frac{2\beta}{3\beta+1}}\right).
\]
In particular, as $\beta$ varies in $(0,\infty)$, the exponent
$\frac{\beta+1}{3\beta+1}$ 
varies in $\left(\frac{1}{3},1\right)$. 
\end{example}

\begin{example} \label{exa-bubb2}
Assume that $a_k=\lfloor e^{f(k)}\rfloor$ where $f$ is positive 
increasing function such that $f^{-1}$ is a regularly varying function
of index strictly less than $1$ (including, possibly, $0$).  
This implies
that $s_k\asymp a_k$ and for $r\in [a_k/4, 4a_{k+1}]$ with $k$ large enough,
we have $|B_X(o,r)| \asymp  r2^k \asymp r2^{f^{-1}(\log r)}$.
Hence, if $\log v \asymp  r 2^{f^{-1}(\log r)} \log r$, we have
$$\Lambda_{1,G}(v)\ge \frac{1}{Cr}\ge \frac{c2^{f^{-1} (c\log\log v)}
\log\log v}{\log v}.
$$

Also, for $k$ large enough, we must have
$k(r)\le k+1$. It follows that
$$[|X^{k(r)-1}_{\mathbf a,\mathbf b}|+ (b-1)^{k(r)}r/2]\asymp 2^k r $$   
and
$$(|X^{k(r)-1}_{\mathbf a,\mathbf b}|+ (b-1)^{k(r)}r/2)!\le e^{ C(\log r )2^k r}.$$   
Hence, assuming again  that $\log v \asymp  r 2^{f^{-1}(\log r)} \log r$,
$$\Lambda_{2,G}(v)\le \frac{C}{r^2}\le C'\left(\frac{2^{f^{-1} (C'\log\log v)}
\log\log v}{\log v}\right)^2.
$$
In particular, if $2^{f^{-1}}$ is regularly varying, then we have
$$\Lambda_{1,G}(v)^2\asymp \Lambda_{2,G}(v)\asymp
\left(\frac{2^{f^{-1} (\log\log v)}
\log\log v}{\log v}\right)^2$$
and
$$\Phi_G(n) \simeq  \exp\left( -n^{1/3} 
\left(2^{f^{-1}(\log n)} \log n \right)^{2/3}\right).$$

As an explicit example, take $2^{f^{-1}(t)}= t^\kappa$, $\kappa>0$, (that is, 
$f(t)=2^{\kappa^{-1} t}$). In this case, we have
$$\Phi_G(n) \simeq  \exp\left( -n^{1/3} (\log n)^{2(1+\kappa)/3}\right).$$

Finally, if $f(t) =t^\kappa$ with $\kappa>1$, a slightly more careful 
computation is needed but the end result is that, in that case,
$$\Lambda_{1,G}(v)^2\asymp \Lambda_{2,G}(v)\asymp
\left(\frac{2^{(\log\log v)^{1/\kappa}}
\log\log v}{\log v}\right)^2$$
and
$$\Phi_G(n) \simeq  \exp\left( -n^{1/3}(\log n)^{2/3} 
2^{\frac{2}{3}(\log n)^{1/\kappa}} \right).$$

\end{example}

\subsection{Amenability} \label{sec-am}

In this section we prove the following statement.

\begin{proposition}\label{pro-am}
Assume that the sequence $\mathbf a = (a_n)$ satisfies
$\liminf a_n = \infty$. 
For an arbitrary $\mathbf b=(b_i)_1^\infty$, $b_i\ge 2$, 
the bubble group $\Gamma_{\mathbf a,\mathbf b}$ 
is amenable.
\end{proposition}

Our goal is to apply the general technique of Section
 \ref{sub:A-unifying-framework} and the main ingredients used 
for this purpose are versions of Lemmas \ref{identification}
and \ref{localembddings}. Under assumption (A), these lemmas apply to 
any level $k$ in the Schreier graph. The problem we face in the present 
setting is to find appropriate levels $k$ where the same ideas can be applied.

Given $r\ge 1$, we say that a level $k$ is  appropriate for $r$ if there exists $k'< k$ so that  \begin{enumerate}
\item $\inf_{j\ge k'}\{a_{j-1}\}> r$ 
\item For any $b\le r$ such that $b=b_i$ for some $i\ge  k$, 
there exist $k(b)\in \{k',k'+1,\dots, k-1\}$ such that $b_{k(b)}=b$.
\item If there exists $i\ge k$ such that $b_i>r$, then there exists 
$k_*\in \{k',k'+1,\dots, k-1\}$ such that $b_{k_*}>r$.
\end{enumerate}
Since $\liminf a_n = \infty$, it is easy to see that,
for any $r$, there is a $r$-appropriate level $k$ (with a finite $k$).

At level $k_*$ consider the points $\mathfrak m_{k_*}=(1^{k_*-1},a_{k_*}/2)$ 
and $\mathfrak n_{k_*}=(1^{k_*-2},a_{k_*-1})$. 
Let $B(\mathfrak m_{k_*},r/4)$
be the ball of radius $r/4$ around $\mathfrak m_{k_*}$ in $X$. 

On the branching cycle $\mathfrak b(1^{k_*-2})$ 
(i.e, the cycle which contains $\mathfrak n_k$), let $J_*(r)$
be arc of radius $r/4$ centered at $\mathfrak n_{k_*}=(1^{k_*-2},a_{k_*-1})$.
Each of the point in  $ J_*(r)\setminus \{\mathfrak n_{k_*}\}$
belongs to a unique bubble of total length $2a_{k_*}$ whereas $\mathfrak n_{k_*}$ belongs to a bubble of length $2a_{k_*-1}$. On each of the bubbles containing 
$x\in J_{*}(r)$, let $\mathcal I_*(x,r)$ the arc of radius 
$r/4$ centered at $x$. Let 
$$\mathcal U_{k_*}(r)=\bigcup _{x\in J_*(r)}\mathcal I_*(x,r).$$

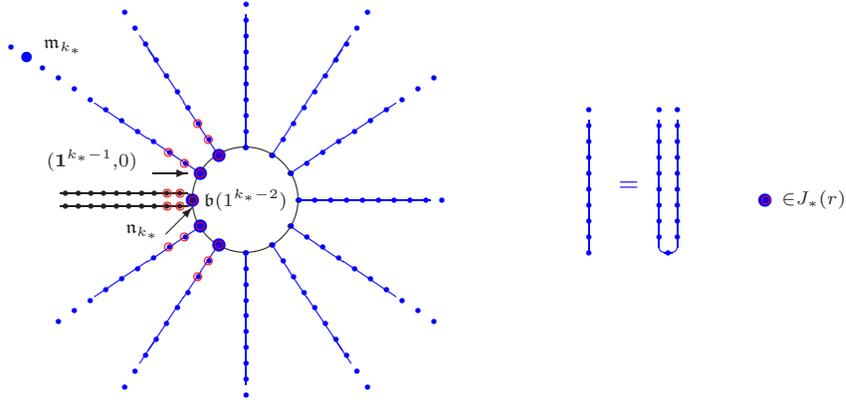
\begin{figure}[h]
\begin{center}\caption{Sketch 
showing $\mathfrak m_{k_*}$, $J_*(r)$ and $\mathcal U_{k_*}(r)$}\label{BAm1}

\begin{picture}(320,160)(0,0) 

\put(90,65){\makebox(20,10){$\scriptstyle
\mathfrak b(1^{k_*-2})$}}

\put(100,70){\circle{40}}
\put(120,70){\circle*{2}}
\put(100,90){\circle*{2}}
\put(100,50){\circle*{2}}

\put(110,53){\circle*{2}}
\put(110,87){\circle*{2}}
\put(90,53){\circle*{2}}
\put(117,80){\circle*{2}}
\put(117,60){\circle*{2}}

{\color{blue}
\put(80,70){\circle*{5}}
\put(83,60){\circle*{5}}
\put(83,80){\circle*{5}}
\put(90,87){\circle*{5}}
\put(90,53){\circle*{5}}}

\put(30,70){\oval(100,5)[r]}
\multiput(32,72.5)(4.8,0){10}{\circle*{2}}
\multiput(32,67.5)(4.8,0){10}{\circle*{2}}

{\color{blue} \put(120,70){\line(1,0){50}}
\put(100,90){\line(0,1){50}}
\put(100,50){\line(0,-1){50}}
\put(90,87){\line(-2,3){30}}\put(110,53){\line(2,-3){30}}
\put(110,87){\line(2,3){30}}\put(90,53){\line(-2,-3){30}}
\put(117,80){\line(3,2){40}}\put(83,60){\line(-3,-2){40}}
\put(117,60){\line(3,-2){40}}\put(83,80){\line(-3,2){40}}}

\put(-3,0){
{\color{blue} 
\multiput(120,70)(6,0){10}{\circle*{2}}
\multiput(100,90)(0,6){10}{\circle*{2}}
\multiput(100,50)(0,-6){10}{\circle*{2}}
\multiput(90,87)(-4,6){10}{\circle*{2}}
\multiput(110,53)(4,-6){10}{\circle*{2}}
\multiput(110,87)(4,6){10}{\circle*{2}}
\multiput(90,53)(-4,-6){10}{\circle*{2}}
\multiput(117,80)(6,4){10}{\circle*{2}}
\multiput(83,60)(-6,-4){10}{\circle*{2}}
\multiput(117,60)(6,-4){10}{\circle*{2}}
\multiput(83,80)(-6,4){13}{\circle*{2}}
\put(17,124){\circle*{4}}}}

{\color{red} 
\multiput(90,87)(-4,6){3}{\circle{3}}
\multiput(90,53)(-4,-6){3}{\circle{3}}
\multiput(83,60)(-6,-4){3}{\circle{3}}
\multiput(83,80)(-6,4){3}{\circle{3}}}

\put(21,123){\makebox(20,10){$\scriptstyle \mathfrak m_{k_*}$}}
\put(27,80){\makebox(30,10){$\scriptstyle(\mathbf 1^{k_*-1},0)$}}
\put(65,80){\vector(1,0){13}}

\put(46,53){\makebox(30,10){$\scriptstyle\mathfrak n_{k_*}$}}
\put(70,57){\vector(1,1){10}}

{\color{blue} \put(230,50){\line(0,1){50}}
\multiput(230,50)(0,6){10}{\circle*{2}}
\put(240,70){\makebox(10,10){$=$}}
\put(260,100){\oval(7,100)[b]}
\put(260,50){\circle*{2}}
\multiput(256.5,56)(0,6){9}{\circle*{2}}
\multiput(263.5,56)(0,6){9}{\circle*{2}}
}

{\color{red} \multiput(72,72.5)(-4.8,0){2}{\circle{3}}
\multiput(72,67.5)(-4.8,0){2}{\circle{3}}
\put(77,70){\circle{3}}
}

{\color{blue} \put(290,70){\circle*{5}}}
{\color{red} \put(287,70){\circle{3}}}
\put(300,65){\makebox(10,10){$\scriptstyle \in J_*(r)$}}

\end{picture}\end{center}\end{figure}

\begin{lemma}[Local Embeddings] 
\label{localembddings*}
For any $r\ge 1$, let $k$ be an $r$-appropriate level.
The restriction of the map $\vartheta_k=\theta_{k+1}:\mathbf{S}^{(\infty)}
\to\Gamma^{k+1}_{\mathbf a,\mathbf b}$
to the set 
$$\Omega_k=\Omega(\mathfrak n_{k_*},\mathcal U_{k_*}(r)) \cap 
\Omega(\mathfrak m_{k_*}, B(\mathfrak m_{k_*},r/4) $$ 
satisfies
the conditions in part 2 of Assumption $\left(\Omega\right)$.
\end{lemma}
\begin{proof}
The proof is along the same lines as the proof of Lemma \ref{localembddings}.
The key ingredient is an identification of orbits similar to 
Lemma \ref{identification} which we explain below.

Recall that $$\mathfrak N (w,r)=\{x\in X: d(x,\mathfrak b(w)\} \le r\} .$$
For any $w$ of length $|w|=s$ with $s\ge k$ such that the branching 
cycle $\mathfrak b(w)$ is of size $b$ at most $r$, we have an obvious 
bijective map
$$\iota^w_{k(b)}: \mathfrak N (w,r)\mapsto \mathfrak N(1^{k(b)-1},r)$$ 
which can be used to identify these vertex sets.

For any $w$ of length $|w|=s$ with $s\ge k$ such that the branching 
cycle $\mathfrak b(w)$ is of size $b$ greater than $r$ and for any 
$w_{s+1}\in \{0\}\cup\{1, \dots,b_{s}-1\}$, let 
$\mathcal W_{w_{s+1}}(w,r)$ be the union of the bubble arcs of 
radius $r$ centered at the bubble roots $y=(wz,0)$ or $(w,a_s)$ 
with $d((ww_{s+1},0),y)<r/2$ if $w_{s+1}\in \{1,\dots,b_s-1\}$ and 
$d((w,a_{s}),y)<r/2$ if $w_{s+1}=0$.  Define the map
$$\iota ^{w, w_{s+1}}_*:\mathcal W_{w_{s+1}}(w,r)\to 
\mathcal W_{0}(1^{k_*-1},r)$$
as follows.  If $w_{s+1}=0$, identify $\mathcal W_{0}(w,r)$ with 
$\mathcal W_{0}(1^{k_*-1},r)$ in the obvious way. If $w_{s+1}\neq 0$, 
use the same obvious identification after having
rotated  $\mathcal W_{w_{s+1}}(w,r)$ along the branching 
cycle $\mathfrak b(w)$ to bring the point $(ww_{s+1},0)$
to $(w,a_s)$.

For any $t\ge k$ and any vertex $x=(w,u)$ with $|w|=t$, $w=w_1\dots w_t$ 
and $0\le u\le 2a_t-1$
(i.e., any vertex at level $t$), let $w'=w_1\dots w_{t-1}$ 
be the parent of $w$. 
For any such $x$ we define a ``reference'' point $\hat{x}$ and a map 
$\hat{\iota}_x$ that carries bijectively a certain neighborhood of $x$ 
to a similar neighborhood of $\hat{x}$. The following specifies case by case
how to construct $\hat{x}$ and $\hat{\iota}_x$. 
\begin{itemize}
\item If $ d\left(x, \{(w,0),(w,a_t))\right)>r/4$ then
$\hat{x}=\mathfrak{m}_{k_*}$. The map $\hat{\iota}_x$ takes the arc of radius 
$r/4$ centered $x$ to the similar arc centered at $\hat{x}=\mathfrak m_{k_*}$.
\item If $d\left(x, (w,0)\right)\le r/4$ and $\mathfrak b(w')$ is such that 
its size is greater than $r$ then
$\hat{x}= (1^{k_*-1}, a_{k_*}+ u)$ if $0\le u\le r/4$ and
$\hat{x}= (1^{k_*-1}, a_{k_*}- 2a_t+u)$ if $ u> a_t $. In this case, set 
$\hat{\iota}_x=\iota_*^{w',w_t}$.
\item If $d\left(x, (w,0)\right)\le r/4$ and $\mathfrak b(w')$ is 
such that its size $b$ is at most $r$ then
$\hat{x}= (1^{k(b)-1}w_t,u)$  if $0\le u\le r/4$ and
$\hat{x}= (1^{k(b)-1}w_t, a_{k(b)}- 2a_t+u)$ if $ u> a_t $.
In this case, set 
$\hat{\iota}_x=\iota^{w'}_{k(b)}$.
\item If $d\left(x, (w,a_t)\right)\le r/4$ and $\mathfrak b(w)$ is such that 
its size is greater than $r$ then
$\hat{x}= (1^{k_*-1}, a_{k_*-1}+u-a_t)$. 
In this case, set 
$\hat{\iota}_x=\iota^{w,0}_{*}$.
\item If $d\left(x, (w,a_t)\right)\le r/4$ and $\mathfrak b(w)$ is 
such that its size $b$ is at most $r$ then 
$\hat{x}= (1^{k(b)-1}, a_{k(b)-1}+u-a_t)$.
In this case, set 
$\hat{\iota}_x=\iota^{w}_{k(b)}$.
\end{itemize}
Let $\omega=\gamma_1\dots\gamma_p$ be a word that belongs to $\Omega_k$.
Then, for any $1\le j\le p$,
$$ \hat{\iota}_x(\gamma^{-1}_j\dots \gamma_1^{-1}\cdot x)=\gamma^{-1}_j
\dots \gamma_1^{-1}\cdot \hat{x}.$$
This is proved by inspection as in the proof of Lemma \ref{identification}.
\end{proof}

\begin{proof}[Proof of Proposition \ref{pro-am}] 
Using  Lemma \ref{localembddings*} and Section \ref{sec-Test}, we build 
test functions that serve as witnesses for the amenability of 
$\mathbb Z\wr _X\Gamma$.

For $r\ge 1$, let $k$ be an $r$-appropriate level with associated 
$k',k_*,k(b)$ as above. 
Let 
$$J_r=\{o,\mathfrak n_{k_*},\mathfrak m_{k_*+1}\}$$
and 
\begin{eqnarray*}
\Xi(k,t)&=&\left\{(w,u): |w|\le k'-1, u\in \{1\dots, 2a_{|w|+1}-1\} \right\}\\
&&\bigcup \left(\bigcup_{|w|=k'-1} \mathfrak N( w,t/4)\right) \\
&& 
\bigcup B(\mathfrak m_{k_*},t/4)\;\;
\bigcup \mathcal U_{k_*}(t).
\end{eqnarray*}
For $t\le r$, this set is made of 
3 disjoint parts $ B(\mathfrak m_{k_*},t/4)$, $\mathcal U_{k_*}(t)$ and the 
rest, and each of these parts contains 
exactly one of the points in $J_k$. Using this notation, we set
$$B_r= \Xi(k,r).$$. 
\end{proof}

\section{Neumann-Segal type groups} \label{sec-NS}

In \cite{Segal2000}, D. Segal constructed finitely generated branch
groups that contain every non-abelian finite simple group as homomorphic
image, and proved that there is no gap in subgroup growth of finitely
generated groups. A similar construction also appeared in 
P. Neumann \cite{Neumann1986}. A version of these constructions
can be described as follows. 

Given a sequence of finite sets $(X_j)_1^\infty$, we obtain a rooted tree
$\mathcal T=\mathcal T_0$ 
with root $\emptyset$, first level $X_1$, and so that each 
vertex at level $i$ has children encoded by a copy of $X_{i+1}$.

Let $\left(G_{i},X_{i}\right)_{i=1}^{\infty}$
be a sequence of groups acting transitively on finite sets $X_{i}$ with 
the property that
each $G_{i}$ is $k$-generated and marked with a generating $k$-tuple 
$(s_{i,1},s_{i,2},...,s_{i,k})$.
In each $X_{i}$, choose two distinct points $x_{i},y_{i}$. Define
automorphisms $\alpha_{i,j}$ and $\beta_{i,j}$ of the tree 
$\mathcal T_{i}=\left(X_{i+1},X_{i+2},...\right)$
recursively as follows. For $i\ge 0$, the automorphism $\alpha_{i,j}$ 
is a rooted
permutation
\[
\alpha_{i,j}(xw)=s_{i+1,j}(x)w,
\]
and $\beta_{i,j}$ is a directed automorphism defined at $xw$
with $x\in X_{i+1}$, $w\in\left(X_{i+2},X_{i+3},...\right)$
\[
\beta_{i,j}(xw)=\left\{ \begin{array}{cc}
x_{i+1}\beta_{i+1,j}(w) & \mbox{ if }x=x_{i+1},\\
y_{i+1}\alpha_{i+1,j}(w) & \mbox{ if }x=y_{i+1},\\
xw & \mbox{otherwise}.
\end{array}\right.
\]
For each $i\ge 0$, let 
$$\Gamma_{i}=\left\langle \alpha_{i,j},\beta_{i,j},
\ 1\le j\le k\right\rangle $$
be the group generated by $\alpha_{i,j},\beta_{i,j}$, $1\le j\le k$, 
acting on the subtree $\mathcal T_{i}=\left(X_{i+1},X_{i+2},...\right)$.

Set  $\alpha_j= \alpha_{0,j}$, $\beta_j=\beta_{0,j}$, $1\le j\le k$ and let
$$\Gamma=\Gamma_0=\left\langle \alpha_{j},\beta_{j},\ 
1\le j\le k\right\rangle $$ be the group
generated by the rooted automorphisms $\alpha_{j}$ and directed automorphisms
$\beta_{j}$, $1\le j\le k$. 
Such groups are called  groups of Neumann-Segal type. They also 
belong to the class of directed tree automorphism groups. Specifically, 
the generators $\beta_j$ are directed along the ray $o=x_1x_2\dots$ in 
$\mathcal T$. In particular, $\beta_j$ leave this ray invariant. See 
\cite{Brieussel2011,Grigorchuk2011}. They are also branch groups. See
\cite{Grigorchuk2000,Grigorchuk2011}.

\begin{figure}[h]
\begin{center}\caption{The level $3$ Schreier graph of the  Neumann-Segal 
group $\Gamma$ with sequence $(l_n)$ starting with $(2,4,4)$.}\label{NS1}
\begin{picture}(150,200)(0,40) 

\multiput(0,0)(0,60){3}{
\put(75,95){\circle*{3}}
\put(75,55){\circle*{3}}
\put(95,75){\circle*{3}}
\put(55,75){\circle*{3}}
\put(75,75){\circle{40}}
\put(75,105){\circle{20}}
\put(75,45){\circle{20}}
\put(105,75){\circle{20}}
\put(45,75){\circle{20}}
\put(75,115){\circle*{3}}
\put(75,35){\circle*{3}}
\put(35,75){\circle*{3}}
\put(115,75){\circle*{3}}}

\multiput(-60,60)(60,0){3}{
\put(75,95){\circle*{3}}
\put(75,55){\circle*{3}}
\put(95,75){\circle*{3}}
\put(55,75){\circle*{3}}
\put(75,75){\circle{40}}
\put(75,105){\circle{20}}
\put(75,45){\circle{20}}
\put(105,75){\circle{20}}
\put(45,75){\circle{20}}
\put(75,115){\circle*{3}}
\put(75,35){\circle*{3}}
\put(35,75){\circle*{3}}
\put(115,75){\circle*{3}}}







\end{picture}\end{center}\end{figure}
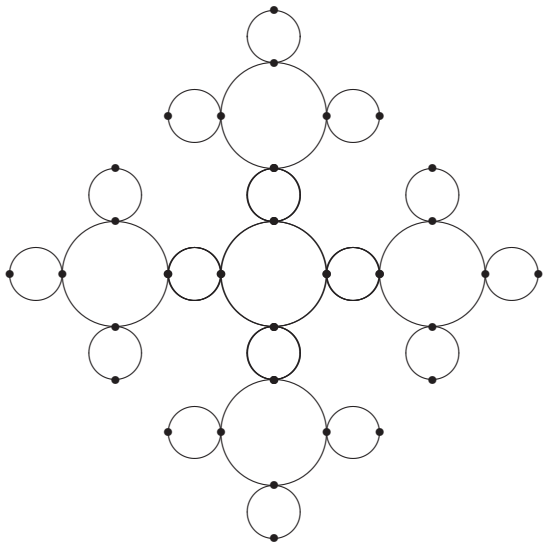

\subsection{Neumann-Segal groups with cyclic $G_i$}
In what follows we focus on the (very) special case where each $G_{i}$ is
a finite cyclic group of even order $l_i$ marked with one generator 
$s_{i}$, and 
$X_{i}=G_{i}= \left\{1,...,2l_{i}-1\right\} $,
that is,  $G_{i}$ acts on itself by multiplication. We also pick 
$$x_i=0 \mbox{ and }y_i=l_i/2.$$ 
These are also examples of generalized Fabrykowski-Gupta groups. 
See \cite{FabGupta}.

As in the case of other directed groups, the Schreier graph $\mathcal S$
of $o=00\dots=0^\infty$ can be constructed via a global substitution 
procedure which we now describes. See \cite{Grigorchuk2011} and 
the references therein. 

For $i=1$, let $\mathcal S_1$ be the cycle $\{0,\dots,2l_1-1\}$. 
For $i\ge 2$, the Schreier graph $\mathcal{S}_{i}$ is constructed by taking 
the cycle $G_i$ and
$l_{i}$ copies of the Schreier graph $\mathcal{S}_{i-1}$ constructed
in the previous step. 
For each $z\in G_i$, attach a copy of the graph $\mathcal{S}_{i-1}$
to $z$ by identifying $z$ with the vertex labeled  $0\dots 0y_{i-1}$ 
(with $i-1$ zeros)
in that copy of $\mathcal{S}_{i-1}$.
Finally, relabel the vertex originally labeled $\omega$ 
in this copy of $\mathcal{S}_{i-1}$ by giving it the label 
$\omega z$ in $\mathcal{S}_{i}$. 

The diameters of the graphs $\mathcal S_i$, $i\ge 1$, 
satisfy  
\[
\mbox{Diam}\left(\mathcal{S}_{i}\right)=
2\mbox{Diam}\left(\mathcal{S}_{i-1}\right)+ l_{i}/2,
\]
therefore 
\[
\mbox{Diam}\left(\mathcal{S}_{i}\right)=\sum_{m=1}^{i}2^{i-m-1}l_{m}.
\]
The cardinality of $\mathcal S_i$ is 
\[
V_{i}=|\mathcal{S}_{i}|=\prod_{1\le m\le i}l_{m}.
\]

The infinite Schreier graph $\mathcal S$ is the orbital Schreier 
graph of the ray $0^\infty$. Viewed from $0^\infty$, the finite Schreier 
graphs $\mathcal S_i$ describe growing pieces of $\mathcal S$. 

In \cite{Fink2014}, E. Fink studies some algebraic  properties of a 
similar class of group where the choice of the point $y_i$ is different, 
namely, instead of $y_i=l_i/2$, Fink makes the choice $y_i=1$, 
a neighbor of $x_i=0$. She also consider mostly the case when $(l_i)$ 
is a sequence of pairwise distinct primes. One of Fink's results is that 
every proper quotient of the groups  she considers are solvable.  
This and many of the other algebraic properties 
she proves carry over without difficulties to the groups we consider here.
See also \cite{Grigorchuk2000,Segal2000}.

The volume growth properties of Neumann-Segal groups is a 
subject of great interest. The cases that has been most studied is when the 
level-$i$ permutation group $G_i$ is the full alternating group on $X_i$.
See \cite{WilsonJS,WilsonJSFurther,BrieusselMZ} 
and the references therein. If the sequence $|X_i|$ is bounded and 
$|X_i|\ge 5$ then the volume growth is exponential (\cite{BrieusselMZ}).  
If $|X_i|$ 
is unbounded, the group contains a free group on $2$ generators 
(\cite{WilsonJSFurther}).

The most striking result is perhaps the fact that when $|X_i|=3$ 
for all $i$, the group has subsexponential volume growth 
(in this case, $G_i=A_3$ is the same as $G_i$ is cyclic!). 
This was first announced in \cite{FabGupta}. Explicit bounds are given in
\cite{Bartholdigrowthtorsion}.  

In general, we understand little about the volume growth of the groups  
$\Gamma$ we consider here. We note that \cite[Proposition 5.4]{Fink2014}
is in error and so is the proof of \cite[Theorem 5.5]{Fink2014} which relies
on it. The proof of the following lemma is along the same lines as  
to the proof of the volume lower bound in \cite{Bartholdigrowthtorsion}.     

\begin{lemma} \label{lem-NSvol}
For any even sequence $(l_i)_1^\infty$, we have
$$|B_\Gamma (V_i r)|\ge 2^{2^{-i-4}V_ir},\;\;  
4\le r \le (l_{i+1}/2)-1,\;\; V_i=l_1\dots l_i,\;\;i\ge 1$$
In particular, if $\liminf_{i\rightarrow \infty} l_i\ge 2^\kappa$,
then 
$$|B_{\Gamma}(r)|\gtrsim \exp\left( r^{(\kappa-1)/\kappa}\right).$$
\end{lemma}
\begin{proof}

First we show that, for any $r\ge 1$,
$$ (|B_{\Gamma_{i+1}}(r)|^{l_i/2}\le  |B_{\Gamma_i}( (1+1/(2r))r l_i).$$
This will follow if we can show that there are elements 
$(g_x)_{0\le x\le l_i-1}\sigma$
in  $B_{\Gamma_i}(r l_i)$ with
$(g_x)_{0\le x\le (l_i/2)-1}$ arbitrarily chosen in $B_{\Gamma_{i+1}}(r)^{l_i/2}$.
Recall that 
$\Gamma_i \subset
\Gamma_{i+1}\wr_{X_i} G_i$ where $G_i=\{0,\dots,l_i-1\}$ 
is the cyclic group  of order $l_i$ with $\alpha_i=(1,\dots, 1)\sigma_i$ 
where $\sigma_i$ the cyclic 
permutation and $\beta_i=(\beta_{i+1},1\,\dots,1, \alpha_{i+1}, 1, \dots, 1)\mbox{ id}$ with $\alpha_{i=1}$ in position $l_i/2$. The key point is that $l_i$ 
is even with $\beta_{i_+1}$ and $\alpha_{i+1}$ at opposite locations. 
By inspection, it is easy to produce the desired elements while 
using $\alpha_i$ to move along the cycle clockwise. 
For instance, consider the case when 
$$g_0=\alpha^{n^0_0}_{i+1}\beta^{n^0_1}_{i+1},\;\; 
g_{(l_i/2)-1}=\beta_{i+1}^{n^{(l_i/2)-1}_1}$$ and  $g_x=1$ if $0<x<(l_i/2)-1$. It takes
$$(l_i/2)+|n_0^0| + (l_i/2) +|n^0_1| + ((l_i/2)-1) +|n^{(l_i/2)-1}_1|$$  
where  each $(l_i/2)$ or $(l_1/2)-1$ 
represent a run  of cycling move using $\alpha_i$ and the $|n^x_j|$ counts
uses of $\beta_i$ to insert either $\alpha_{i+1}$ (when $j$ is even) 
or $\beta_{i+1}$ (when $j$ is odd). 
The total number of moves for this example  is $(l_i/2) + 2((l_i/2) -1
\sum_x |g_x|$.  In general, an arbitrary
$(g_x)_{0\le x\le l_i/2}$ with $|g_x|\le r$ and a maximal number of switches 
from $\alpha_{i+1}$ to $\beta_{i+1}$ in any $g_x$, $0\le x\le (l-i/2)-1$ 
equal to $m$ can be produce in at most 
$$(l_1/2)+ (l_i/2)m +\sum_0^{l_i/2-1}|g_x| -1.$$
This is less than  $$(l_i/2)+ l_ir= (1+ 1/(2r))rl_i$$
as desired. By induction, this gives
$$ |B_{\Gamma}( R_i)|\ge |B_{\Gamma_{i+1}}(r)|^{2^{-i-1}v_i}$$
where
$$R_i= \left(\prod_1^{i}(1+1/(2 r w_{j,i}))l_j\right) r,\;\;
w_{j,i}= l_{j+1}\dots l_i,\;\; w_{i,i}=1.$$
Note that, since $l_i\ge 2$ and $r\ge 1$, we have
$\prod_1^i(1+1/(2rw_{j,i}))\le  e$.  Hence $R_i\le e V_i r$ and we have
$$|B_\Gamma (eV_ir)|\ge |B_{\Gamma_{i+1}}(r)|^{2^{-i-1}V_i}.$$
Further, a simple version of the previous argument shows that
$$|B_{\Gamma_{i+1}}(r)|\ge 2^{r/2},\;\; 1\le r\le (l_{i+1}/2)-1.$$
Hence, we obtain 
$$|B_\Gamma (V_ir)|\ge 2^{2^{-i-4}V_ir},\;\;   4\le r \le (l_{i+1}/2)-1,\;\; i=1,2,\dots..$$

\end{proof}

\subsection{Lower bound on the isoperimetric profile}
\label{sub:Neumann-Segal-Lowerbound}

By applying Proposition \ref{volume-diameter} in the case of 
the above group $\Gamma$, we obtain a lower bound on the isoperimetric 
of $\mathbb Z\wr _{\mathcal S} \Gamma$.
\begin{corollary}[{Corollary of Proposition \ref{volume-diameter}}] 
\label{cor-VD} For any sequence of even integers $l_i$ with 
$\limsup l_i=\infty$, We have
\[
\Lambda_{1,\mathbb Z\wr_{\mathcal S}\Gamma}(v)\ge\frac{1}{Cr } \;\;
\mbox{for all }v\le C^{-1} r^{\left|B_{\mathcal{S}}(0^{\infty},r)\right|/C}.
\]
\end{corollary}

\begin{example} \label{exa-wreathNS}
Let $l_n= 2n^d$, $n\ge 1$, for some integer $d$.
Diameter and volume are given by $\mbox{Diam}(\mathcal S_n)\asymp 2^n$
and $|\mathcal S_n|= V_n=2^{n}(n!)^d$.  This gives
\[
\Lambda_{1,\mathbb Z\wr_{\mathcal S}\Gamma}(v)\gtrsim 
c\, 2^{-\frac{1}{d}\frac{\log\log v}{\log\log\log v}\left(1+
2\frac{\log\log\log \log v}{\log\log\log v}\right)
}
\]

If instead we assume that $l_n=2^{1+\lfloor n ^\gamma\rfloor}$ with 
$\gamma>0$ then
$\mbox{Diam}(\mathcal S_{n})
\asymp 2^n$ if $\gamma\in (0,1)$,
$\mbox{Diam}(\mathcal S_{n})\asymp n 2 ^{n}$ if $\gamma=1$ and 
$\mbox{Diam}(\mathcal S_n)\asymp 2 ^{n^\gamma}$ if $\gamma>1$.
Also, $\log_2( V_{n})\sim (1+\gamma)^{-1} n^{1+\gamma}$.  
The bound of the previous corollary
yields
\[
\Lambda_{1,\mathbb Z\wr \Gamma}(v)\ge\frac{1}{C 2^{n^{\max\{1,\gamma\}}}} \;\;
\mbox{for all }v\le C^{-1} \exp\left( 2^{ [(1+\epsilon) (1+\gamma)]^{-1}n^{\gamma+1}}\right),
\]
for any $\epsilon>0$. That is,
\[
\Lambda_{1,\mathbb Z\wr \Gamma}(v)\ge c 2^{- (1+\epsilon)[(1+\gamma)\log_2\log v ]^{\max\{1,\gamma\}
/(1+\gamma)}}.
\]
\end{example}

\subsubsection{First lower bound for $\Gamma$ itself}

This subsection shows how the very general Proposition \ref{rigidlamp-1} 
can be applied to the Neumann-Segal groups $\Gamma=<\alpha,\beta>$, 
 under certain hypotheses on the length sequence $(l_n)_1^\infty$.  

In order to apply Proposition \ref{rigidlamp-1}, for each level $n$, we need 
to find an element $\rho_n$ in $\Gamma$ that belongs to the rigid stabilizer 
$\mbox{rist}_\Gamma(u_n)$ of a vertex $u_n$ at level $n$ and to control the 
length $|\rho_n|_\Gamma$ of $\rho_n$ in $\Gamma$. 
For this purpose, we assume that $\overline{\lim}_{n\to\infty}l_{n}=\infty$. Let 
\[
u_n=0^{n} \in \mathcal S_{n}
\]
and
\[
\rho_{n}=\beta ^{M_{n+1}},
\]
where $M_{n+1}$ is a chosen common multiple of 
$\left\{ l_{2},\dots,l_{n+1}\right\}$.
By direct inspection, we have 
$$\rho_{n}=(\beta_{n+1}^{M_{n+1}},1,\dots,1)     \mbox{id}$$
and $\beta_{n+1}^{M_{n+1}}$ is non-trivial because $\beta$ is of infinite order
thanks to the assumption that $(l_i)$ is unbounded. 

\begin{corollary}[{Corollary of Proposition \ref{rigidlamp-1}}] 
\label{cor-RL}
Suppose $\overline{\lim}_{n\to\infty}l_{n}=\infty$,
let $M_{n+1}$ denote a chosen common multiple of 
$\left\{ l_{2},\dots,l_{n+1}\right\} $. Then there exists a constant $C\ge 1$ such that
\[
\Lambda_{1,\Gamma}(v)\ge\frac{1}{C\max\left\{ M_{n+1},r\right\} } \;\;
\mbox{for all }v\le C^{-1} 2^{\left|B_{\mathcal{S}_{n}}(u_n,r)\right|/C}.
\]
\end{corollary}

\begin{remark} \label{rem-cor}
We can always take 
$$M_{n+1}=V_{n+1}=l_1\dots l_{n+1} \mbox{ and } 
r=\mbox{Diam}(\mathcal S_{n}(u_n)) \le V_{n+1}.$$ This gives
\[
\Lambda_{1,\Gamma}(v)\ge\frac{1}{CV_{n+1}} \;\;
\mbox{for all }v\le C^{-1} 2^{V_{n}/C}
\]
but this estimate is too weak to be useful because $V_{n+1}/V_{n}=l_{n+1}$ 
is unbounded.
\end{remark}

\begin{example} Assume that $l_n=2^{1+ \lfloor n ^\gamma\rfloor}$, $n\ge 1$, with 
$\gamma>0$. This gives
$\mbox{Diam}(\mathcal S_{n})
\asymp 2^n$ if $\gamma\in (0,1)$,
$\mbox{Diam}(\mathcal S_{n})\asymp n 2 ^{n}$ if $\gamma=1$ and 
$\mbox{Diam}(\mathcal S_{n})\asymp 2 ^{n^\gamma}$ if $\gamma>1$.
Also, $\log_2( V_{n})\asymp (1+\gamma)^{-1}n^{1+\gamma}$ and 
$M_{n+1}=l_{n+1}=2^{1+ \lfloor (n+1) ^\gamma\rfloor}$.  
The bound of the previous corollary
yields (with a different constant $C$ depending of $\gamma>0$)
\[
\Lambda_{1,\Gamma}(v)\ge\frac{1}{C 2^{(n+1)^{\max\{1,\gamma\}}}} \;\;
\mbox{for all }v\le C^{-1} \exp\left( 2^{(1-\epsilon)(1+\gamma)^{-1}n^{\gamma+1}}
\right).
\]
That is,
\[
\Lambda_{1,\Gamma}(v)\ge  c 2^{- (1+\epsilon)[(1+\gamma) \log_2\log v ]^{\max\{1,\gamma\}
/(1+\gamma)}}.
\]
Note that this is the same bound we obtained for 
$\Lambda_{1,\mathbb Z\wr \Gamma}$ at the end of Example \ref{exa-wreathNS}.
Because of the appearance of the quantity $M_{n+1}$ in Corollary \ref{cor-RL},
when $l_n=2n^d$ as in the first part of Example \ref{exa-wreathNS}, 
we cannot give a lower bound similar to that obtained 
for $\Lambda_{1,\mathbb Z\wr \Gamma}$. 
\end{example}

\subsubsection{Improved lower bound for $\Gamma$ itself}

The main drawback of Corollary \ref{cor-RL} is the fact that
the bound involves the quantities 
$M_n$ and $|B_{\mathcal S_{n}}(u_n,r)|$ instead of 
$M_{n-1}$ and $|B_{\mathcal S_{n}}(u_n,r)|$.  See Remark \ref{rem-cor}.
In this section we show that in the special case of the group $\Gamma$ 
studied in this section, a slightly 
sharper version of Proposition \ref{rigidlamp-1} 
can be obtained and that fixes this drawback. 

\begin{proposition}[{Improved version of Corollary \ref{cor-RL}}] 
\label{pro-RL}
Suppose $\overline{\lim}_{n\to\infty}l_{n}=\infty$,
let $M_{n}$ denote a chosen common multiple of 
$\left\{ l_{2},...,l_{n}\right\} $. 
Then there exists a constant $C\ge 1$ such that
\[
\Lambda_{1,\Gamma}(v)\ge\frac{1}{C\max\left\{ M_{n},r\right\} } \;\;
\mbox{for all }v\le C^{-1} 2^{\left|W_{n}(r)\right|/C}.
\]
\end{proposition}
\begin{proof} In $\mathcal S_{n}=\mathcal S_{n}(u_{n})$, $u_n=0^{n}$, 
consider the set $W_{n}$ of those vertices 
$v=z_1 z_{n-1}z_{n}$ with $0\le z_{n}\le \lfloor l_{n}/4\rfloor$.
Set
$$W_{n}(r)= W_{n}\cap B_{\mathcal S_{n}}(u_{n},r).$$
Note that $B_{\mathcal S_{n}}(u_{n},r)\subset W_{n}$ if and only if $r\le 
\mbox{Diam}( \mathcal S_{n-1})+l_{n}/4$.

For each such $v$, pick $g^v$ such that $g^vu_{n}=v$ and
$|g^v|=d_{\mathcal S_{n}}(u_{n},v)$.
Set $\varrho_{n}= \beta^{M_{n}}$. By construction
$$\varrho_{n}=(\tilde{\varrho}_x)_{x\in \mathcal T^{n}}\mbox{ id}$$
where $\tilde{\varrho}_x$ is the identity except at
$x=u_{n}$ and $x=\bar{u}_{n}=0^{n-1}(l_{n}/2)$. Further,
$$\tilde{\varrho}_{u_{n}}=\beta_{n+1}^{M_{n}} \mbox{ and } 
\tilde{\varrho}_{\bar{u}_{n-1}}=\alpha_{n+1}^{M_{n}}.$$ 
For $v\in W_{n}$, write $g^v=(g^v_x)_{x\in \mathcal T^{n}}\sigma^v$. We have
$\sigma^v(u_{n})=v$. Because a minimal length representation of $g^v$ as a word in $\alpha,\beta$ provides a geodesic from $u_{n}$ to $v$ in 
$\mathcal S_{n}$, we also have $\sigma^v(\bar{u}_{n})
=z_1\dots z_{n-1}
\bar{z}_{n}$ where $\bar{z}_{n}=z_{n}+l_{n}/2$. 
Set
$\overline{W}_{n}=\{\sigma^v(\bar{u}_{n}): v\in W_{n}\}$ and observe 
that $W_{n}$ and $\overline{W}_{n}$ are disjoint subsets of 
$\mathcal S_{n}$. Then 
$$g^{v}\varrho_{n}\left(g^{v}\right)^{-1}=(\theta^v_x)_{x\in \mathcal T^{n}}
\mbox{id}
$$
where all $\theta_x^v$ are trivial except two, namely,
$$\theta^v_v=\theta^v_v=
g_v^{v}\beta_{n+1}^{M_{n}}(g^v_{v})^{-1} \mbox{ and }
\theta^v_{\bar{v}}=g_{\bar{v}}^{v}\alpha_{n+1}^{M_{n}}(g^v_{\bar{v}})^{-1}$$
where $\bar{v}=\sigma^v(\bar{u}_{n})$.

Now, let $\zeta$ be the symmetric probability measure on
the subgroup 
$$\left\langle g^{v}\varrho_{n}\left(g^{v}\right)^{-1}:\ 
v\in W_{n}\right\rangle $$
defined be
\[
\zeta(\gamma)=\frac{1}{2|W_{n}(r)|}
\sum_{v\in W_{n}(r)}
\mathbf{1}_{\left\{ g^{v}\varrho_{n}^{\pm1}\left(g^{v}\right)^{-1}\right\}}(\gamma).
\]
As $\zeta$ has the form (\ref{zetaprod}) on the product of cyclic groups
\[
X= \prod_{v\in W_{n}(r)}<g^{v}\varrho_{n}\left(g^{v}\right)^{-1}>,
\]
Comparison of $\zeta$ with simple random walk on $\Gamma$
and Proposition \ref{product} gives the desired result.
\end{proof}

\begin{corollary}
Suppose $\overline{\lim}_{n\to\infty}l_n=\infty$. Then, for $c\in (0,1)$ 
small enough,  we have
\[
\Lambda_{1,\Gamma}(v)\ge \frac{c}{\log v} \mbox{ for } v\in [2^{c^2 V_{n}}, 
2^{c V_{n}}]. \]
\end{corollary}
Note that a better lower bound would be available 
if we knew that $\Gamma$ has exponential volume growth.

\subsection{Upper bound on the isoperimetric profile}

In this section, we provide  upper bounds on $\Lambda_{2,\mathbb Z\wr \Gamma}$
by applying the general method explained in Section \ref{sec-UB}.
\begin{proposition} \label{pro-NS2}
There exists a constant $C$ such that, 
for any sequence $(l_i)_1^\infty$ of 
even natural numbers and  any $n$, we have
$$\Lambda_{2,\mathbb Z\wr _{\mathcal S}\Gamma}(v)\le \frac {1}{C R_n}
\mbox{ for any }\;\;v\ge C R_n^{CV_n } $$
where $R_n= \sum_1^n 2^{n-j-2}l_j$ and $V_n=l_1\dots l_n$.
\end{proposition}
The quantity $R_n$ which appears here is (essentially) the resistance
between the root  $o=0^\infty$ and the set $\mathcal S_n^c$ in the Schreier graph
$\mathcal S$.

\begin{example}
Let $l_n= 2n^d$, $n\ge 1$ for some integer $d$.
Resistance and volume are given by $R_n\asymp 2^n$
and $|\mathcal S_n|= V_n=2^{n}(n!)^d$.  This gives
\[
\Lambda_{2,\mathbb Z\wr_{\mathcal S}\Gamma}(v)\lesssim 
C\, 2^{-\frac{1}{d}\frac{\log\log v}{\log\log\log v}
\left(1+\frac{\log\log\log\log v}{2\log\log\log v}\right)}
.\]

If instead we assume that $l_n=2^{1+ \lfloor n ^\gamma\rfloor}$ with 
$\gamma>0$ then
$R_n \asymp 2^n$ if $\gamma\in (0,1)$,
$R_n\asymp n 2 ^{n}$ if $\gamma=1$ and 
$R_n\asymp 2 ^{n^\gamma}$ if $\gamma>1$.
Also, $\log_2( V_{n})\asymp (1+\gamma)^{-1}n^{1+\gamma}$.  This gives
\[
\Lambda_{2,\mathbb Z\wr \Gamma}(v)\le \frac{C}{2^{ n^{\max\{1,\gamma\}}}} \;\;
\mbox{for all } v\ge C \exp\left( 2^{[(1-\epsilon)(1+\gamma)]^{-1}
n^{\gamma+1}}\right).
\]
That is,
\[
\Lambda_{2,\mathbb Z\wr \Gamma}(v)\le 2^{- (1-\epsilon)
[(1+\gamma)\log_2\log v ]^{\max\{1,\gamma\}
/(1+\gamma)}}.
\]
\end{example}

\begin{proposition} \label{pro-NS3}
There exists a constant $C$ such that, 
for any sequence $(l_i)_1^\infty$ of 
even natural numbers and  any $n$, we have
$$\Lambda_{2,\mathbb Z\wr_{\mathcal S} \Gamma}(v)\le \frac {C}{r^2}
\mbox{ for any }\;\;v\ge \exp\left( C V_{n-1} r\log r\right)  $$
where $r\in (0,l_n/4)$ and $V_{n-1}=l_1\dots l_{n-1}$.
\end{proposition}
\begin{remark} Proposition \ref{pro-NS3} gives a better result than 
Proposition \ref{pro-NS2} when $l_n\gg \sqrt{R_n}$ 
(by inspection, we always have $R_n\ge l_n/4$). 
\end{remark}
\begin{example}
Consider the case when $l_n=2^{1+ \lfloor n ^\gamma\rfloor}$ with 
$\gamma>0$.  If $\gamma\in (0,1)$, Proposition \ref{pro-NS2} gives 
a better than Proposition \ref{pro-NS3} whereas, for $\gamma\ge 1$, 
Proposition \ref{pro-NS3} yields a  
much  better result. Namely, using $r=l_n/4$,
 Proposition \ref{pro-NS3} gives
\[
\Lambda_{2,\mathbb Z\wr \Gamma}(v)\le \frac{C}{2^{2 n^{\max\{1,\gamma\}}}} \;\;
\mbox{for all } v\ge C \exp\left( 2^{[(1-\epsilon)(1+\gamma)]^{-1}
n^{\gamma+1}}\right).
\]
That is,
\[
\Lambda_{2,\mathbb Z\wr \Gamma}(v)\le 2^{- 
2(1-\epsilon)[(1+\gamma)\log_2\log v ]^{\max\{1,\gamma\}
/(1+\gamma)}}.
\]
\end{example}
The next theorem applies to $\Gamma$ as well as to $\mathbb Z\wr_\mathcal S \Gamma$.
\begin{theorem} \label{theo-NSgamma}
For $l_n=2^{1+\lfloor n^\gamma\rfloor}$, $n\ge 1$, $\gamma\ge 1$,
we have, for any $\epsilon>0$ and constants $c,C$ that depends 
only on $\epsilon$,
$$   c 2^{- 
(1+\epsilon)
[(1+\gamma)\log_2\log v ]^{\gamma/(1+\gamma)}}\le 
\Lambda_{1,\Gamma}(v)  \le  C 2^{- (1-\epsilon)
[(1+\gamma)\log_2\log v ]^{\gamma/(1+\gamma)}}, $$
$$   c 2^{- 2(1+\epsilon)
[(1+\gamma)\log_2\log v ]^{\gamma/(1+\gamma)}}\le 
\Lambda_{2,\Gamma}(v)  \le  C 2^{- 2
(1-\epsilon)
[(1+\gamma)\log_2\log v ]^{\gamma/(1+\gamma)}}, $$
and
$$\frac{cn}{ 2^{ 2(1-\epsilon)
[(1+\gamma)\log_2 n ]^{\gamma/(1+\gamma)}}} \le 
-\log \Phi_{\Gamma}(n) \le  \frac{Cn}{ 2^{2(1+\epsilon)
[(1+\gamma)\log_2 n ]^{\gamma/(1+\gamma)}}}, $$ 
as well as the same estimates for $\Lambda_{1,\mathbb Z\wr_{\mathcal S}\Gamma}$,
$\Lambda_{2,\mathbb Z\wr_{\mathcal S}\Gamma}$ and $\Phi_{\mathbb Z\wr_{\mathcal S }\Gamma}$.
\end{theorem}

\begin{proof}[Proof of Proposition \ref{pro-NS2}]
First, we explain how we arrange for assumption $(\Omega)$ to be satisfied.
For each $n$, consider the copy of $\mathcal S_{n}$ in $\mathcal S$ 
which is anchored at $0^\infty$. In $\mathcal S$, let $B_n$ (resp. 
$\overline{B}_n$) be the set of 
those points that are (strictly) closer to $0^\infty$ than 
to $0^{n-1}(l_{n}/2)0^\infty$ (resp. closer to $0^{n-1}(l_{n}/2)0^\infty$ than 
to $0^\infty$). Set $J_n=\{0^\infty\}$. Property 1 in Definition 
\ref{def-Omega} is obviously satisfied. 

For each $n$, let $\pi_n(\Gamma)$ be the projection of $\Gamma$ at level $n$, 
that is the group defined naturally by the marked Schreier graph 
$\mathcal S_{n}$.  Consider the abelian group $<a_{n+1}>\times 
<b_{n+1}>$ with generators $a_{n+1},b_{n+1}$  where $a_{n+1}$ has order 
$l_{n+1}$ and $b_{n+1}$ has the same order as $\beta_{n+1}$ 
(possibly, infinity).  Set
$$\Gamma^{n+1}= [<a_{n+1}>\times <b_{n+1}>]\wr_{\mathcal S_{n}} 
\pi_n(\Gamma).$$

Let $\omega$ be a reduced word in $s_1=\alpha$, $s_2=\beta$ (an element in 
$\mathbf F= <\alpha> * <\beta>$) which belong to $\Omega(J_n,B_n)$. 
By construction, its projection $g\in \Gamma$ has the form 
$(g_x)_{x\in \mathbb T^{n}}\sigma $ where $\sigma=\pi_n(g)\in \pi_n(\Gamma)$.
Further, for all $x\in B_n$, $g_x$ is a power of $\beta_{n+1}$, for all 
$x\in \overline{B}_n$, $g_x$ is a power of $\alpha_{n+1}$ and for all 
$x\notin B_n\cup \overline{B}_n$, $g_x=1$.
For $\omega \in \Omega(J_n,B_n)$, set
$$\vartheta_n(\omega)=  ((\tilde{g}_x)_{x\in \mathcal S_{n}},\sigma)
\mbox{ with } \tilde{g}_x=\left\{\begin{array}{cc} b_{n+1}^q &\mbox{ if } 
x\in B_n \mbox{ and } g_x=\beta_{n+1}^q\\
a_{n+1}^q &\mbox{ if } 
x\in B_n \mbox{ and } g_x=\alpha_{n+1}^q.\end{array}\right.
 $$
By inspection and the definition of $\Omega(J_n,B_n)$, if 
$\omega_1,\omega_2$ and $\omega_1\omega_2$ are all in $\Omega(J_n,B_n)$,
we have
$\vartheta_n(\omega_1,\omega_2)=\vartheta_n(\omega_1)\vartheta_n(\omega)2)$ 
as desired. 

Having verified that assumption $(\Omega)$ holds, it remains to construct 
test functions $F_n, \psi_n$ as in Section \ref{sec-Test}.

Set
$$\psi_n(((\tilde{g}_{x})_{x\in \mathcal S_{n}},\sigma))= 
\prod_{x\in \mathcal S_{n}}\mathbf 1_{[-r,r]^2}(\tilde{g}_x).$$
Here $r$ is a parameter to be specified later and 
$[-r,r]^2$ is understood as the set 
$$ \{a_{n+1}^{-r},\dots,a_{n+1}^r\}\times 
\{b_{n+1}^{-r},\dots,b_{n+1}^{r}\}$$
in $<a_{n+1}>\times <b_{n+1}>$. Obviously,
$$Q_{\Gamma^n,\mu}(\psi_n)\lesssim r^{-2}.$$

The needed $(J_n,B_n)$-admissible function $F_n$ on finite subsets of 
$\mathcal S$ is provided by Lemma \ref{lem-ResA} and we 
have
$$Q_{\mathcal P_f(\mathcal S),\mu}(F_n)\lesssim 
\frac{1}{\mathcal R(J_n, B_n^c)}.$$
Here, we have $\mathcal R(J_n,\mathcal S_{n}^c)=
\sum_1^{n}2^{n-j}(l_j/4)$ and
$$\mathcal R(J_n, B_n^c)
=\mathcal R(J_n,\mathcal S_{n-1}^c)+\lfloor l_{n}/8\rfloor \asymp 
\mathcal R (J_n,\mathcal S_{n}^c).$$
Picking $r\asymp \sqrt{\mathcal R(J_n,\mathcal S^c_{n})}$ and 
applying Proposition \ref{function} 
\end{proof}

\begin{proof}[Proof of Proposition \ref{pro-NS3}]
To prove this proposition, we will estimate $\Lambda_{2,\Gamma, \mu}$
and later $\Lambda_{2,\mathbb Z\wr _{\mathcal S}\Gamma, \mu}$ using the random 
walk on the right driven by $\mu$. Here $\mu$ is the uniform probability on 
$\{\alpha^{\pm 1},\beta^{\pm 1}\}$.

We first explain how to prove the bound
$$\Lambda_{2,\Gamma,\mu}(v)\le \frac {C}{r^2}
\mbox{ for any }\;\;v\ge \exp\left( C V_{n-1} r\log r\right)  $$
where $r\in (0,l_n/4)$ and $V_{n-1}=l_1\dots l_{n-1}$ (the statement of 
Proposition \ref{pro-NS3} is for $\mathbb Z\wr_{\mathcal S} \Gamma$ instead of 
$\Gamma$ itself).
 
We embed $\Gamma$ in 
a larger group and use remark \ref{enlarge}.  
Using the wreath recursion 
to level $n-1$ in $\mbox{Aut}(\mathcal T)$, we embed $\Gamma$ in
$$\widetilde{\Gamma}=  \Gamma_n\wr_{\mathcal S_{n-1}} \pi_{n-1}(\Gamma).$$
Recall that $\Gamma_n=<\alpha_n,\beta_n>$ and $\pi_{n-1}(\Gamma)$ 
is the finite group corresponding to the projection of $\Gamma$ 
acting on the finite subtree up to level $n-1$.
We construct a test function $\Phi$ on $\widetilde{\Gamma}$. 

Namely, any element of 
$g\in \widetilde{\Gamma}$ can be written as
$$ g=(g_x)_{x\in \mathcal S_{n-1}}\sigma $$
with $g_x\in \Gamma_{n}$ and $\sigma\in \pi_{n-1}(\Gamma)$.
Pick a test function $\phi_r$ defined on $\Gamma_{n}$ and set
$$\Phi(g)= \prod_{\mathcal S_{n-1}}\phi_r ( g_x).$$
We have
$$ g \alpha  =(g_{x})_{x\in \mathcal S_{n-1}}\pi_{n-1} \sigma (\alpha)  $$
and
$$ g\beta =(g'_{x})_{x\in \mathcal S_{n-1}}\sigma \pi_{n-1}(\beta ) $$
with $g'_x=g_{x}$ except for $z_1=\sigma\cdot 0^{n-1}, z_2=\sigma \cdot 
0^{n-2}y_{n-1}$ where
 $g'_{z_1}= g_{z_1}\beta_n$ and $g'_{z_2}
=g_{z_2}\alpha_n $. If $\mu$ is the uniform measure on 
$\{\alpha^{\pm 1},\beta^{\pm 1}\}$ and $\mu_n$ is the uniform measure on 
$\{\alpha_n^{\pm 1},\beta_n^{\pm 1}\}$, this yields
$$\frac{\mathcal E^r_{\widetilde{\Gamma},\mu}(\Phi,\Phi)}{\|\Phi\|^2_{L^2(\widetilde{\Gamma})}}
\le 4 \frac{\mathcal E^r_{\Gamma_n,\mu_n}(\phi_r,\phi_r)}{\|\phi_r
\|^2_{L^2(\Gamma_n)}}.$$

Recall that our convention is that $\mathcal E^r$ is the Dirichlet form for 
the random walk on the right (The $r$ in $\mathcal E^r$ has nothing 
to do with the real parameter $r$). 

We know little about $\Gamma_n$ whose structure is similar to that of $\Gamma$
but we have
$$\Gamma_n \subset \Gamma_{n+1}\wr_{X_n} G_n.$$
where $G_n$ is the cyclic group of order $l_n$ and 
$X_n=\{0,\dots, l_{n}-1\}$ and every element in $\Gamma_n$ is of the form
$$g=(g_x)_{x\in X_n} \alpha_{n}^t, \;\;g_x\in \Gamma_{n+1},\; 
t\in \{-l_{n}+1,\dots, l_n-1\}$$
where $\alpha_n$ is understood as either $\alpha_n$ itself or as the 
corresponding element in $G_n$. 

Pick a parameter $r\in (0, l_{n}/4)$ and set $\psi_r(m)=(1- |m|/r)_+$.  
In $\prod_{x\in X_n}(\Gamma_{n+1})_x$, consider the set $\Sigma_r$ 
parametrized by
$$k_x\in \{-r^2,r^2\}, \;\;x\in \{0, r\}\cup \{l_n-r-1,\dots, l_n-1\}.$$
of 
those elements of the form  
$$g_x= \left\{\begin{array}{cl}\beta_{n+1}^{k_x} & \mbox{ for } 
x\in \{0, r\}\cup \{l_n-r-1,\dots, l_n-1\},\\
\alpha_{n+1}^{k_{x-l_n/2}} & \mbox{ for } 
x\in \{l_n/2, l_{n}/2+r\},\\
\alpha_{n+1}^{k_{x+l_n/2}} & \mbox{ for } 
x\in \{l_n/2-r, l_{n}/2-1\},\\
e_{\Gamma_{n+1}}& \mbox{ otherwise}\end{array}\right.
$$
Set 
$$\phi_r(g) = \mathbf 1_{\Sigma_r}((g_x)_{X_n})\psi_r(t),\;\;g=(g_x)_{X_n} \alpha_{n}^t. $$
Observe that for $g$ in the support of $\phi_r$,
$$g\alpha^{\pm 1}_n= (g_x)_{X_n} \alpha_{n}^{t\pm 1}$$
and
$$g\beta_{n}^{\pm 1}= (g'_x)\alpha_{n}^t$$
where $g'_x=g_x$ except at $z_1=\alpha_n^t(0)$ and $z_2=\alpha_n^t(l_n/2)$.
At this two locations, $g'_{z_1}= \beta_{n+1}^{k_{z_1}\pm 1}$ and
$g_{z_2}=\alpha_{n+1}^{k_{z_1}\pm 1}$. This implies that 
$$\mathcal E^r_{\Gamma_n,\mu_n}(\phi_r,\phi_r)
\le  C r^{-2} \|\phi_r\|^2_{L^2(\Gamma_{n})}.$$
Moreover, when $r$ is an integer,
$$|\mbox{supp}(\phi_r)|= (2r^2)^{2r}(2r-1).$$
Returning to the test function $\Phi$ on 
$\widetilde{\Gamma}=\Gamma_n\wr_{\mathcal S_{n-1}}\pi_{n-1}(\Gamma)$, we have
$$\frac{\mathcal E_{\widetilde{\Gamma},\mu}(\Phi,\Phi)}
{\|\Phi\|^2_{L^2(\widetilde{\Gamma})}}
\le Cr^{-2}$$
and
\begin{eqnarray*}
|\mbox{supp}(\Phi)| &=& [(2r^2)^{2r}(2r-1)]^{V_{n-1}} |\pi_n(\Gamma)|\\
&=& [(2r^2)^{2r}(2r-1)]^{V_{n-1}} \prod_1^{n-1}(l_j)^
{l_1\dots l_{j-1}} \\
&\le &  \exp\left( C V_{n-1} r\log r\right)
\end{eqnarray*}
for $r\in (0,l_n/4)$. 
This yields
$$\Lambda_{2,\Gamma}(v)\le  \frac{C}{r^2}  \mbox{ for all } v\ge 
\exp \left( C V_{n-1} r\log r\right) \mbox{ and } r\in (0,l_n/4).$$

Next, we prove the similar result on $\mathbb Z\wr_{\mathcal S}\Gamma$.
We observe that the test function $\Phi$ on $\widetilde{\Gamma}\supset \Gamma$
produces a test function $\Psi$ on $\Gamma$ itself with
$$\frac{\mathcal E^r_{\Gamma,\mu}(\Psi,\Psi)}{\|\Psi\|_{L^{2}(\Gamma)}^2}\le \frac{\mathcal E^r_{\widetilde{\Gamma},\mu}(\Phi,\Phi)}{\|\Phi
\|_{L^{2}(\widetilde{\Gamma})}^2}\le \frac{C}{r^2}$$
and, using the wreath recursion to level $n-1$, any $g\in\Gamma $ 
is of the form $(g_x)_{\mathcal \mathcal S_{ n-1}}\sigma$ with $g_x\in \Gamma_n$
and $\sigma\in \pi_{n-1}(\Gamma)$ and
 $$\mbox{supp}(\Psi) \subset \{g=(g_x)_{\mathcal S_{n-1}} \sigma: 
\phi_r (g_x)\neq 0, x\in \mathcal S_{n-1}\}$$    
For any $g\in \mbox{supp}(\Psi)$, we have $g\cdot 0^\infty\in U_n(r)$
where $U_n(r)$ is described (see Figure \ref{NS2})
using the tree indexing of the points in 
$\mathcal S$ as 
$$U_n(r)=\{ x_1\dots x_{n-1}y0^\infty: x_i\in X_i,  i\le n-1,y\in \{0,\dots, r\}\cup \{ l_n-r,\dots, l_n-1\}\}.$$

\begin{figure}[h]
\begin{center}\caption{Sketch of the set $U_n(r)$ (in black) with the root $o$
marked in red. Each little circle represents a copy of the graph $\mathcal S_{n-1}$. }\label{NS2}
\begin{picture}(100,80)(0,40)

{\color{red} \put(49.5,75){\circle*{2}}}

\put(75,98){\circle{5}}
\put(75,52){\circle{5}}
\put(98,75){\circle{5}}
\put(97,81){\circle{5}}
\put(97,68.5){\circle{5}}
\put(94.5,87){\circle{5}}
\put(94.5,63){\circle{5}}
\put(90.5,91.5){\circle{5}}
\put(90.5,58.5){\circle{5}}
\put(86.5,94.5){\circle{5}}
\put(86.5,55.5){\circle{5}}
\put(81,97){\circle{5}}
\put(81,53){\circle{5}}

\put(52.3,75){\circle*{5}}
\put(53,80.8){\circle*{5}}
\put(55,86){\circle*{5}}
\put(58.5,91){\circle*{5}}

\put(53,69.2){\circle*{5}}
\put(55,64){\circle*{5}}
\put(58.5,59){\circle*{5}}

\put(63,94.5){\circle{5}}
\put(68.5,97){\circle{5}}

\put(63,55.5){\circle{5}}
\put(68.5,53){\circle{5}}

\put(75,75){\circle{40}}

\end{picture}\end{center}\end{figure}
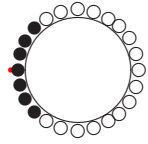

On $\mathbb Z\wr_{\mathcal S} \Gamma$, consider the test function
$$\Phi'((f,g))= \prod_{x\in U_n(r)}\mathbf 1_{[-r^2,r^2]}(f(x)) \Psi (g).$$
On $\mathbb Z\wr_{\mathcal S}\Gamma$ and for $(f,g)\in \mbox{supp}(\Phi')$,
we have
$$ (f,g) (\mathbf 1^{0^\infty}_{\pm 1}
,e_\Gamma)=   (f \mathbf 1^{g\cdot 0^\infty}_{\pm 1}   , g)$$
and
$$ (f,g) (\mathbf e^{\mathcal S}_{\mathbb Z},\gamma)
=   (f, g\gamma ), \gamma\in \Gamma.$$
This gives 
$$\mathcal E^r_{\mathbb Z\wr_{\mathcal S}\Gamma,\mathfrak q}(\Phi',\Phi')\le \frac{C}{r^2}$$
where $\mathfrak q =\frac{1}{2}(\nu+\mu)$ and $\nu$ is the uniform measure
on $\pm 1$. Finally,
$$|\mbox{supp}(\Phi')|=  (1+2r^2)^{|U_N(r)|}|\mbox{supp}(\Psi)| 
\le \exp\left( C V_{n-1} r\log r\right),$$
with $0\le r\le l_n/4$
The desired bound follows.
\end{proof}

\subsection{Concrete examples of the type $\mathbb Z\wr_{\mathcal S}\Gamma$}
This section is devoted to spelling out examples that illustrate our results 
in the case $G=\mathbb Z\wr_{\mathcal S}\Gamma$ when $\Gamma$ belongs to 
the family of a cyclic 
Neumann-Segal groups associated with sequences $(l_n)_1^\infty$ of 
even integers that are growing fast enough.  

\begin{theorem} \label{th-NSmain} Assume that $l_n=2^{\kappa(n)}$ with 
$\kappa(n+1)\ge \kappa(n)+1$.  
Then, for $r\in [l_{n-1},l_n]$, we have
$$\Lambda_{1,\mathbb Z\wr _{\mathcal S}\Gamma}(v)\le \frac{c}{r} \mbox{ for }
 v \le \exp(c V_{n-1} r \log r),$$ 
and
$$\Lambda_{2,\mathbb Z\wr _{\mathcal S}\Gamma}
(v)\le \frac{C}{r^2} \mbox{ for }
 v \ge \exp(C V_{n-1} r \log r).$$ 
Equivalently,
$$\Lambda^2_{1,\mathbb Z\wr_{\mathcal S}\Gamma}(v)\asymp
\Lambda_{2,\mathbb Z\wr_{\mathcal S}\Gamma}(v)\asymp \left\{\begin{array}{cl}
r^{-2} &\mbox{ for } \log v = V_{n-1}r\log r,\;\; r \in 
[l_{n-1},l_{n}]\\
l_n^{-2} & \mbox{ for } \log v\in [V_{n-1} l_n\log l_n, V_{n} l_n\log l_n].
\end{array}\right.$$
\end{theorem}
\begin{proof}This follows from Corollary \ref{cor-VD} and Proposition 
\ref{pro-NS3}.
\end{proof}
\begin{corollary} \label{cor-NSmain}
Assume that $l_n=2^{\kappa(n)}$ with 
$\kappa(n+1)\ge \kappa(n)+1$.  
The random walk invariant $\Phi_{\mathbb Z\wr _{\mathcal S}\Gamma}$
is given as follows:
\begin{itemize}
\item For  $t\in [V_{n-1}l_{n-1}^3 \log l_{n-1},V_{n-1}l_{n}^3 \log l_{n}]$,
$$
-\log \left(\Phi_{\mathbb Z\wr _{\mathcal S}\Gamma}(t)\right)
\asymp V_{n-1}^{2/3}t^{1/3}\left(\log (t/V_{n-1})\right)^{2/3} .$$
\item For  $t\in
[V_{n-1}l_{n}^3 \log l_{n},V_{n}l_{n}^3 \log l_{n}]$,
$$ 
-\log \left(\Phi_{\mathbb Z\wr _{\mathcal S}\Gamma}(t)\right)
\asymp t/l_n^2.$$
\end{itemize} 
\end{corollary}
\begin{proof}
This follows by (somewhat lengthy) inspection from the previous theorem and 
the well-know relations between $\Lambda_2$ and $\Phi $ See, e.g., 
 \cite{BPS,Saloff-Coste2014} and the references therein.
\end{proof}

Corollary \ref{cor-NSmain} gives a complete picture of the behavior of 
the probability of return for simple random on the 
cyclic Neumann-Segal groups
$\Gamma$ considered here when $l_n=2^{\kappa(n)}$  with $\kappa $ growing 
at least linearly. The result provides a continuum of distinct explicit
behaviors for the random walk invariant $\Phi_{G}$ as well as for the  
profiles $\Lambda_{1,G}, \Lambda_{2,G}$. 

\begin{remark}\label{rem-exaNS}
Theorem \ref{th-NSmain} and Corollary \ref{cor-NSmain} and 
the examples described below are all concerned with the 
group $G=\mathbb Z\wr_{\mathcal S}\Gamma$ and not with $\Gamma$ itself. 
Since $\Gamma$ is a subgroup of $G$
we have $\Lambda_{1,\Gamma} \le \Lambda_{1,G}$, 
$\Lambda_{2,\Gamma} \le \Lambda_{2,G}$ and 
$-\log \Phi_{\Gamma} \lesssim -\log \Phi_G$.

Under the hypotheses of Theorem \ref{th-NSmain} and
Corollary \ref{cor-NSmain}, we also have that
$$\Lambda_{1,\Gamma}(v)\ge \frac{c}{l_n} \mbox{ for }
\log v\in [V_{n-1}l_{n},V_nl_n].$$
On the interval $[V_{n-1}l_{n_1},V_{n-1}l_n]$ which is not covered 
by these estimates, 
we have no better lower bounds than those obtained by monotonicity or 
by using the classical volume bound of \cite{CSCiso} together with 
Lemma \ref{lem-NSvol}.  This means that, in general, we are not able to 
provide matching two-sided bounds for the isoperimetric profiles of the group 
$\Gamma$ itself.

Nevertheless, there are cases 
where using monotonicity is sufficient to obtain a 
satisfactory result such as in Theorem \ref{theo-NSgamma} where $l_n=
2^{1+\lfloor n^\gamma\rfloor}$, with $\gamma>1$. Even in this case, 
Theorem \ref{th-NSmain} and Corollary \ref{cor-NSmain} gives 
more precise results on $\mathbb Z\wr _{\mathcal S} \Gamma$ than what we know 
for $\Gamma$.  Namely, a careful
application of Theorem \ref{th-NSmain} and Corollary \ref{cor-NSmain} show that,
when $l_n=
2^{1+\lfloor n^\gamma\rfloor}$ with $\gamma>1$, we have  
$$
\Lambda_{1,\Gamma}(v)^2 \asymp
\Lambda_{2,\Gamma}(v)  \asymp  \left(\frac{2^{[(1+\gamma)\log_2\log v ]^{\gamma/(1+\gamma)}}}{\log v}\right)^2. $$
and
$$-\log \left(\Phi_{\mathbb Z\wr _{\mathcal S}\Gamma}(t)\right)\asymp 
\frac{t}{2^{2((1+\gamma)\log_2 t)^{\gamma/(1+\gamma)}}}.$$
Compare with Theorem \ref{theo-NSgamma}.
\end{remark}

The following examples illustrate what happen when $\kappa$ grows faster
than $2^{1+\lfloor n^\gamma\rfloor}$, $\gamma>1$.  
In each of these examples, we describe the extreme behaviors of the 
function $-\log \left(\Phi_{\mathbb Z\wr _{\mathcal S}\Gamma}\right)$ which, 
according to  Corollary \ref{cor-NSmain}, 
are obtained at the points  
$t=V_{n-1}l_{n}^3 \log l_{n}$ and $t=V_{n}l_{n}^3 \log l_{n}$. Note that
Corollary \ref{cor-NSmain} provides complementary sharp estimates 
at all times.

\begin{example} \label{exa-01}
Let $l_n=2^{\varkappa^n}$, that is $\kappa(n)=\varkappa^n$,
$\varkappa>1$. 
In this case,  
$$\log_2 V_n= \frac{\varkappa}
{\varkappa-1}  (\varkappa^{n}-1), \;\;V_{n-1}l_{n}^3\asymp 
l_n^{\frac{3\varkappa-2}{\varkappa-1}}
, \;\;V_{n}l_{n}^3\asymp 
l_n^{\frac{4\varkappa-3}{\varkappa-1}} .$$ 
This gives:
\begin{itemize}
\item For $t\asymp V_{n-1}l_{n}^3\log l_n \asymp l_{n}^{\frac{3\varkappa-2}{\varkappa-1}}\log l_n$,
$$-\log \left(\Phi_{\mathbb Z\wr _{\mathcal S}\Gamma}(t)\right)
\asymp t^{\frac{\varkappa}{3\varkappa-2}} 
\left(\log t \right)^{\frac{2\kappa-2}{3\kappa-2}}   $$
\item For $t\asymp V_{n}l_{n}^3\log l_n \asymp 
l_{n}^{\frac{4\varkappa-3}{\varkappa-1}}\log l_n$,
$$-\log \left(\Phi_{\mathbb Z\wr _{\mathcal S}\Gamma}(t)\right)
\asymp t^{\frac{2\varkappa-1}{4\varkappa-3}}  
\left(\log t \right)^{\frac{2\kappa-2}{4\kappa-3}} .$$
\end{itemize}
In addition, the corollary also gives that, for all $t>1$,
$$
ct^{\frac{\varkappa}{3\varkappa-2}} 
\left(\log t \right)^{\frac{2\kappa-2}{3\kappa-2}} \le
-\log \left(\Phi_{\mathbb Z\wr _{\mathcal S}\Gamma}(t)\right)
\le C t^{\frac{2\varkappa-1}{4\varkappa-3}}  
\left(\log t \right)^{\frac{2\kappa-2}{4\kappa-3}}  .$$
\end{example}

\begin{example} \label{exa-O2}
Let $l_n$ be such that $l_n=2^{V_{n-1}}$. In this case,
$$ V_{n-1}l_n^3\log l_n \asymp  V_{n-1}^2 2^{3V_{n-1}},\;\;
V_nl_n^3\log l_n\asymp  V_{n-1}^2 2^{4V_{n-1}}.$$
This gives:
\begin{itemize}
\item For $t\asymp V_{n-1}l_{n}^3\log l_n 
\asymp V_{n-1}^2 2^{3V_{n-1}}$,
$$-\log \left(\Phi_{\mathbb Z\wr _{\mathcal S}\Gamma}(t)\right)
\asymp t^{\frac{1}{3}} 
\left(\log t \right)^{\frac{4}{3}} ;  $$
\item For $t\asymp V_{n}l_{n}^3\log l_n \asymp 
V_{n-1}^2 2^{4V_{n-1}}$,
$$-\log \left(\Phi_{\mathbb Z\wr _{\mathcal S}\Gamma}(t)\right)
\asymp t^{\frac{1}{2}}  
\log t   .$$
\end{itemize}
Further, for all $t>1$,  we have
$$c t^{\frac{1}{3}} 
\left(\log t \right)^{\frac{4}{3}} \le 
-\log \left(\Phi_{\mathbb Z\wr _{\mathcal S}\Gamma}(t)\right)
\le C t^{\frac{1}{2}}  
\log t   .$$
\end{example}

\begin{example} \label{exa-O3}
Suppose now we have a sequence of integer $n_i$ tending 
to infinity such that   
$$l_{n_i}=2^{V_{n_i-1}},  n_{i+1}=V_{n_i-1}+n_i+1   
\mbox{ and }  l_{n_i+j}=2l_{n_i+j-1}, \; 1\le j\le n_{i+1}-n_i-1
.$$
On the one hand, at time  $t_i\asymp V_{n_i-1}l_{n_i}^3\log l_{n_i}$, by the 
same computation as in the previous example, we have
$$-\log \Phi_{\mathbb Z\wr _{\mathcal S}\Gamma}(t_i) \asymp
t_i^{\frac{1}{3}} \left(\log t_i \right)^{\frac{4}{3}}. $$
On the other hand, 
\begin{eqnarray*}
V_{n_{i+1}-1}= V_{n_i-1} 
l_{n_i}\dots l_{n_{i+1}-1}
&=& 2^{(n_{i+1}-n_i)V_{n_i-1}+ \sum_1^{n_{i+1}-n_{i}-1}j} V_{n_i-1} \\
&= &2^{\frac{3}{2}(V_{n_i-1}^2+V_{n_i-1})}V_{n_i-1},
\end{eqnarray*}
$$l_{n_{i+1} -1}= 2^{2V_{n_i-1}}.$$
Hence, at time 
$$t'_i \asymp V^2_{n_{i}-1} 2^{\frac{3}{2}(V_{n_i-1}^2+\frac{15}{2}V_{n_i-1})}$$
we have
\begin{eqnarray*}
-\log \Phi_{\mathbb Z\wr _{\mathcal S}\Gamma}(t'_i) &\asymp &
V^2_{n_{i}-1} 2^{\frac{3}{2}(V_{n_i-1}^2+\frac{7}{2}V_{n_i-1})}\\
&\asymp& t'_i 2^{-4V_{n_i-1}} 
\asymp \frac{t'_i}{2^{4(\frac{2}{3}\log_2 t'_i)^{1/2}}}. 
\end{eqnarray*}
Also,  for all $t>1$, we have
$$
ct^{\frac{1}{3}} \left(\log t \right)^{\frac{4}{3}}
\le
-\log \left(\Phi_{\mathbb Z\wr _{\mathcal S}\Gamma}(t)\right)
\le C \frac{t}{2^{4(\frac{2}{3}\log_2 t)^{1/2}}} .$$
\end{example}

\begin{example} \label{exa-O4}
Suppose now we have a sequence of integer $n_i$ tending 
to infinity such that   
$$l_{n_i}=V_{n_i-1}^{\varkappa-1},  n_{i+1}=\log_2 V^{\varkappa-1}_{n_i-1}+n_i+1   
\mbox{ and }  l_{n_i+j}=2l_{n_i+j-1}, \; 1\le j\le n_{i+1}-n_i-1,
$$
for some $\varkappa>1$. By the 
same computation as in Example \ref{exa-O1}, we have
$$-\log \Phi_{\mathbb Z\wr _{\mathcal S}\Gamma}(t_i) \asymp
t_i^{\frac{\varkappa}{3\varkappa -2}} \left(\log t_i \right)^{\frac{2\varkappa -2}{
3\varkappa-2}}. $$
On the other hand, if we set $N_i= n_{i+1}-n_{i}-1= \log _2(V_{n_i-1}^{\varkappa-1})$,
\begin{eqnarray*}
V_{n_{i+1}-1}= V_{n_i-1} 
l_{n_i}\dots l_{n_{i+1}-1}
&=& V_{n_i-1}^{1+(\varkappa-1)(N_i+1)} 
2^{\frac{1}{2}N_i(N_i+1)} \\
&= & 2^{\frac{3}{2}N_{i}^2+  (\frac{1}{2}+\frac{\varkappa}{\varkappa-1})  N_{i}},
\end{eqnarray*}
$$l_{n_{i+1} -1}= V^{\varkappa-1}_{n_i-1}2^{N_i}=2^{2N_i}.$$
Hence, at time 
$$t'_i \asymp V_{n_{i+1}-1} l^3_{n_{i+1}-1} \log l_{n_{i+1}-1}\asymp
2N_i 2^{\frac{3}{2}N_{i}^2+  (6+\frac{1}{2}+\frac{\varkappa}{\varkappa-1})  N_{i}}, $$
we have
$$
-\log \Phi_{\mathbb Z\wr _{\mathcal S}\Gamma}(t'_i) \asymp 
 \frac{t'_i }{2^{-4N_i}} 
\asymp \frac{t'_i}{2^{4(\frac{2}{3}\log t'_i)^{1/2}}}. 
$$
In addition, Corollary \ref{cor-NSmain} also gives that, for all $t>1$,
$$
ct^{\frac{\varkappa}{3\varkappa-2}} \left(\log t \right)^{\frac{2\varkappa-2}{3\varkappa -2}}
\le
-\log \left(\Phi_{\mathbb Z\wr _{\mathcal S}\Gamma}(t)\right)
\le C \frac{t}{2^{4(\frac{2}{3}\log t)^{1/2}}} .$$
\end{example}

\appendix
\section{Appendix:  action on finite sets \label{sub:A-key-estimate}}

As in the core of the paper, let $\Gamma$ be a finitely generated group,
with generating set $S$,
acting on space $X$ with a reference point $o$ chosen in $X$.
Let $\mathcal{S}$ denote the orbital Schreier graph of $o$ under
the action of $G$. Let $\mu$ be a symmetric probability measure
on $\Gamma$. We are concern here with Consider the action of $\Gamma$ on
$\mathcal{P}_{f}(\mathcal{S})=\oplus_{\mathcal{S}}\mathbb{Z}_{2}$. 

\begin{definition} Let $J,B$ be  fixed finite subsets of $\mathcal S$ 
and $X$, respectively, with $J\subset B$. 
Set
$$ L^{2}(\mathcal{P}_{f}(X);J;B)=\{\Psi\in L^{2}(\mathcal{P}_{f}(X)):
 A\in\mbox{supp}\Psi \Rightarrow J\subseteq 
A\subseteq B\}$$
and 
\[
\lambda_{\mathcal{P}_{f}(\mathcal{S}),\mu}(B;J)
=\inf\left\{ \frac{\sum\mu(s)\left\Vert s\cdot \Psi-\Psi\right\Vert^2 _2}{\left\Vert \Psi\right\Vert _{2}^{2}}:\ 
0\neq \Psi\in L^{2}(\mathcal{P}_{f}(X);J;B)
\right\} .
\]
\end{definition}
Here as usual the action of $\Gamma$ on functions is given 
by $(g\cdot F)(A)=F(g^{-1}\cdot A)$,
$A\in\mathcal{P}_{f}(X)$. The requirement that $J\subset A$ for
every $A$ in the $\mbox{supp}\Psi$ needs some justification. If, instead,
we look at 
\[
\inf\left\{ \frac{\sum\mu(s)\left\Vert s\cdot
F-F\right\Vert _{L^{2}(\mathcal{P}_{f}(X))}^{2}}{\left\Vert F\right\Vert _{L^{2}(\mathcal{P}_{f}(X))}^{2}}:\ \mbox{supp}F\subseteq\{0,1\}^{B}\right\} ,
\]
then it agrees with the usual notion of Markov chain spectral gap 
(with Dirichlet boundary condition) on $B$.
But this infimum is $0$ because we can take the function
$F=1$ on the empty set (all $0$ configuration), and $F=0$ everywhere
else. The important additional requirement is
that every set in the support of $\Psi$ to contain a specific set
$J$. This requirement is also justified by the fact that
the action of $\mathbb{Z}_{2}\wr_{X}\Gamma$ on $\mathcal{P}_{f}(X)$ is
amenable if and only if the action of $\Gamma$ on $\mathcal{P}_{f}(X)$
admits an invariant mean giving full weight to the collection of sets
containing a fixed finite set, see \cite[Lemma 3.1]{Juschenko2013a}.
In the context of Section \ref{sub:A-unifying-framework}, $J=J_{n}$
is chosen as a set of ``special points'' which have the property that 
having control over their inverted orbit implies the existence of 
the local embedding $\vartheta_{n}$ in Definition \ref{def-Omega}.  

We now describe an upper bound on 
$\lambda_{\mathcal{P}_{f}(\mathcal{S}),\mu}(B;J)$
based on arguments inspired by \cite{Juschenko2013}. Here the notions of energy
of functions on the graph $\mathcal{E}_{\mathcal S,\mu}$, resistance 
$\mathcal{R}_{\mathcal S,\mu}(U\leftrightarrow V)$
are all standard. Namely,
\[
\mathcal{E}_{\mathcal S,\mu}(h,h)=\frac{1}{2}\sum_{g\in \Gamma}
\sum_{x\in \mathcal S}|h(x)-h(g\cdot x)|^{2}\mu(g),
\]
and, for $U\subset V\subset \mathcal S$,
\[
\mathcal{R}_{\mathcal S,\mu}(U\leftrightarrow V)^{-1}=
\inf\left\{ \mathcal{E}_{\mathcal S,\mu}(h,h):\ h=1  \mbox{ on }U,\ 
h=0\mbox{ on }V\right\} .
\]

\begin{lemma}[Compare to 
part of {\cite[Theorem 2.8]{Juschenko2013}}] \label{resistance}
Fix finite subsets $J\subset B\subset \mathcal S$.
Given a function $h:X\rightarrow[0,1]$
such that $h=1$ on $J$, $h=0$ on $B^{c}$ and 
$\mathcal{E}_{\mathcal S,\mu}(h,h)\le 1/2$,
there exists a function $F_{h}:\mathcal{P}_{f}(X)\rightarrow [0,1]$ such that 
$\left\Vert F_{h}\right\Vert _{L^{2}(\mathcal{P}_{f}(X))}^{2}=1$,
$A \in \mbox{\em supp}(F_{h}) \Rightarrow J\subseteq A\subseteq B$ and 
\[
\sum_{s\in \Gamma}\mu(s)\left\Vert s\cdot F_{h}-F_{h}
\right\Vert _{L^{2}(\mathcal{P}_{f}(X))}^{2}\le
\frac{\pi^{2}}{2}\mathcal{E}_{\mathcal S,\mu}(h,h).
\]
In particular, we have
\[
\lambda_{\mathcal{P}_{f}(\mathcal{S}),\mu}(J;B)\le
\frac{\pi^{2}/2}{\mathcal{R}_{\mathcal S,\mu}(J\leftrightarrow B^{c})}.
\]
\end{lemma}

\begin{proof} 
Given a function $h:X\rightarrow[0,1]$, define 
$F_{h}:\mathcal{P}_{f}(X)\rightarrow[0,1]$ by setting
\[
F_{h}(\eta):=\prod_{v\in X}\xi_{h(v)}(\eta(v)),\;\;
\eta\in \mathcal P_f(X)=\oplus_X\mathbb Z_2
\]
where
\[
\xi_{a}(0):=\cos\left(\frac{\pi a}{2}\right),\ \xi_{a}(1):=\sin\left(\frac{\pi a}{2}\right).
\]

Note that if $\eta(u)=0$ for some $u\in J$, we must have $h(u)=1$ and 
it follows that
$\xi_{h(u)}(\eta(u))=\xi_{1}(0)=\cos(\pi/2)=0$. Therefore
$A\in\mbox{supp}(F_{h}) \Rightarrow  J\subseteq A$. Similarly, if
$\eta(v)=1$ for some $v\in B^{c}$ then $h(v)=0$, $\xi_{h(v)}(\eta(v))=\xi_{0}(1)=\sin(0)=0$. In particular,
we have $\mbox{supp}(F_{h})\subseteq\{0,1\}^{B}$.

To compute the relevant $L^{2}$ norms, write
\begin{eqnarray*}
\left\Vert F_{h}\right\Vert _{L^{2}(\mathcal{P}_{f}(X))}^{2}&=&
\sum_{\eta}F_{h}(\eta)^{2}=\sum_{\eta}\prod_{v\in X}\xi_{h(v)}(\eta(v))^{2}\\
&=&\prod_{v\in\Omega}\left[\xi_{h(v)}(1)^{2}+\xi_{h(v)}(0)^{2}\right]=1,
\end{eqnarray*}
and
$$
\left\Vert s\cdot F_{h}-F_{h}\right\Vert _{L^{2}(\mathcal{P}_{f}(X))}^{2}=
2\left\Vert F_{h}\right\Vert _{L^{2}(\mathcal{P}_{f}(X))}^{2}-
2\left\langle s\cdot F_{h},F_{h}\right\rangle _{L^{2}(\mathcal{P}_{f}(X))}^{2}.$$
Note that 
\begin{eqnarray*}
s\cdot F_{h}(\eta)&=&F_{h}(s^{-1}\cdot \eta)=\prod_{v\in X}\xi_{h(v)}(s^{-1}\cdot
\eta(v))\\
&=&\prod_{v\in X}\xi_{h(v)}(\eta(s\cdot v))=\prod_{v\in X}\xi_{h(s^{-1}\cdot 
v)}(\eta(v)).
\end{eqnarray*}
Therefore, we have
\begin{eqnarray*}
\lefteqn{\left\langle s\cdot F_{h},F_{h}\right\rangle=
\sum_{\eta}F_{h}(s^{-1}\cdot \eta)\cdot F_{h}(\eta)}\hspace{.2in}&&\\
&=&
\sum_{\eta}\prod_{v}\xi_{h(s^{-1}\cdot v)}(\eta(v))\xi_{h(v)}(\eta(v))
\\
&=&\prod_{v\in B}\left[\cos\left(\frac{\pi}{2}h(s^{-1}\cdot
v)\right)\cos\left(\frac{\pi}{2}h(v)\right)+\sin\left(\frac{\pi}{2}h(s^{-1}
\cdot v)\right)\sin\left(\frac{\pi}{2}h(v)\right)\right]\\
&=&\prod_{v\in B}\cos\left(\frac{\pi}{2}\left(h(s^{-1}
\cdot v)-h(v)\right)\right).
\end{eqnarray*}
Now, use the fact that $\cos(x)\ge e^{-x^{2}}$ if $|x|\le\pi/4$, together
with the assumption that 
$\left\Vert s\cdot h-h\right\Vert _{L^{2}(X)}^{2}\le\frac{1}{2}$,
to obtain
\begin{eqnarray*}
\prod_{v\in\Omega}\cos\left(\frac{\pi}{2}\left(h(s^{-1}
\cdot v)-h(v)\right)\right)
&\ge&
\prod_{v\in G}\exp\left(-\frac{\pi^{2}}{4}\left(h(s^{-1}
\cdot v)-h(v)\right)^{2}\right)
\\
&=&\exp\left(-\frac{\pi^{2}}{4}
\left\Vert s\cdot h-h\right\Vert _{L^{2}(X)}^{2}\right)\\
&\ge &1-\frac{\pi^{2}}{4}\left\Vert s.h-h\right\Vert _{L^{2}(X)}^{2}.
\end{eqnarray*}
We conclude that 
\[
\sum_{s\in \Gamma} \mu(s)\left\Vert s\cdot F_{h}-F_{h}\right\Vert _{L^{2}(\mathcal{P}_{f}(X))}^{2}\le\frac{\pi^{2}}{2}\mathcal{E}_{\mathcal S,\mu}(h,h).
\]
\end{proof}

\begin{remark}
The function $F_{h}$ is a product of functions at each point in $X$ and
the previous computation does not involve information about the relations 
between different orbits. Asking for a better function is related 
to the problem of finding a better method than using an union bound. 
As the action of $\Gamma$ on subsets of $X$ is usually quite intricate, 
it is a rather difficult question. 
But in simple cases such as the dihedral group
(Subsection \ref{dihedral}) and the bubble groups (Subsection \ref{bubble}),
one can find a better function by inspecting how certain subsets move
under the group action.
\end{remark}

\begin{lemma}\label{lem-ResA}
Let $(\Gamma,X, o)$ and
$J_n,B_n$ be as in {\em Definition \ref{def-Omega}}. For each $n$, 
there exists a $(J_n,B_n)$-admissible function $F_n$ such that
$$\mathcal Q_{\mathcal P(\mathcal S),\mu}(F_n)\le 
\frac{\pi^{2}/2}{\mathcal{R}_{\mathcal S,\mu}(J_n\leftrightarrow B_n^{c})}.$$
\end{lemma}
\begin{proof}
To obtain a function $F_n$ that is $(J_n,B_n)$-admissible as required 
in \ref{sub:A-unifying-framework} and which has small Rayleigh quotient, 
consider the function $F^*_n=F_h$ of 
Lemma \ref{resistance} associated with $J_n,B_n$ and the corresponding 
optimal choice of $h$ so that $\mathcal E _{\mathcal S,\mu}(h,h)=
\mathcal R_{\mathcal S,\mu}(J_n \leftrightarrow B_n^c)^{-1}$.
Observe that the support of $F^*_n$ can be partitioned into orbits
of certain finite subsets of $B_n$. In particular, there must exists a finite
subset $A_n$ such that the restriction of $F^*_n$ to the orbit of $A_n$
has Rayleigh quotient bounded above by 
$
\frac{\pi^2}{2} \mathcal R_{\mathcal S,\mu}(J_n \leftrightarrow B_n^c)^{-1}$.
 
In particular, the function
\[
F_{n}(Y)=\left\{ \begin{array}{cl}
F^*_n(Y)& \mbox{ if }Y=g\cdot A_n  \mbox{ for some } g\in G,\\
0 & \mbox{ otherwise,}
\end{array}\right.
\]
is $(J_n,B_n)$-admissible and satisfies
\[
\mathcal{Q}_{\mathcal{P}_{f}(\mathcal{S}),\mu}(F_{n})\le\frac{\pi^{2}/2}{\mathcal{R}_{\mu}(J_n\leftrightarrow B_n^{c})}
\]
as desired. \end{proof}

\bibliographystyle{plain}

\end{document}